\documentclass{article}

\usepackage{vmargin,epsfig}
\usepackage[english]{babel}
\usepackage[latin1]{inputenc}
\usepackage[tbtags]{amsmath}
\usepackage{amsthm}
\usepackage{amssymb}
\usepackage{mathrsfs}
\usepackage{tikz}
\usetikzlibrary{matrix,patterns,arrows, decorations.pathmorphing,positioning,calc}
\usepackage{microtype}
\usepackage[algoruled,vlined,english,onelanguage,linesnumbered]{algorithm2e}
\usepackage{enumerate}

\newcommand{\Ap}{\mathscr{P}}

\newcommand{\Rpi}{\varpi}
\newcommand{\Ri}{\mathfrak{R}}
\newcommand{\Cond}{\mathrm{Cond}}
\newcommand{\Z}{\mathbb{Z}}
\newcommand{\N}{\mathbb{N}}
\newcommand{\Q}{\mathbb{Q}}
\newcommand{\R}{\mathbb{R}}

\newcommand{\mS}{S}
\newcommand{\im}{\mathrm{Im}}

\newcommand{\mSloc}{S_{\nu,loc}}
\newcommand{\mSnu}{\mS_{\nu}}
\newcommand{\mSnup}{\mS_{\nu'}}
\newcommand{\mSpnu}{\mS_{\nu,\pi}}
\newcommand{\mSunu}{\mS_{\nu,u}}
\newcommand{\mSunup}{\mS_{\nu',u}}
\newcommand{\eSnu}{\mS'_{\nu}}
\newcommand{\eSpnu}{\mS'_{\nu,\pi}}
\newcommand{\eSunu}{\mS'_{\nu,u}}
\newcommand{\mE}{\mathscr{E}}
\newcommand{\mO}{\mathscr{O}}
\newcommand{\mSp}{\mS_\pi}
\newcommand{\mSu}{\mS_u}
\newcommand{\Frac}{\mathrm{Frac}}
\renewcommand{\ker}{\mathrm{ker}\,}
\newcommand{\coker}{\mathrm{coker}\,}
\newcommand{\M}{\mathscr{M}}
\newcommand{\Nc}{\mathscr{N}}
\newcommand{\Rc}{\mathscr{R}}
\newcommand{\sO}{\tilde{O}}

\newcommand{\qis}{\mathrm{qis}}
\newcommand{\Max}{\mathrm{Max}}
\newcommand{\Free}{\mathrm{Free}}

\newcommand{\free}{\mathrm{max}}
\newcommand{\uFree}{\underline \Free}
\newcommand{\uMax}{\underline \Max}

\newcommand{\Mod}{\mathrm{Mod}}
\newcommand{\uMod}{\underline \Mod}
\newcommand{\uModtf}{\underline \Mod^{\mathrm{tf}}}

\newcommand{\sat}{\mathrm{sat}}
\newcommand{\Hi}{\mathrm{Hi}}
\newcommand{\Lo}{\mathrm{Lo}}
\newcommand{\minplus}{\min\nolimits^+}
\newcommand{\la}{\leftarrow}
\newcommand{\repr}{\mathrm{repr}}

\newcommand\proche[2]{{\{ #2 \}_#1}}
\newcommand\procheplus[2]{{\{ #2 \}^+_#1}}

\newtheorem{theorem}{Theorem}[section]
\newtheorem{lemma}[theorem]{Lemma}
\newtheorem{proposition}[theorem]{Proposition}
\newtheorem{corollary}[theorem]{Corollary}
\newtheorem{definition}[theorem]{Definition}

\newtheorem{remark}[theorem]{Remark}
\newtheorem{example}[theorem]{Example}

\title{Linear Algebra over $\Z_p[[u]]$ and related rings}
\author{Xavier Caruso,
David Lubicz
}

\begin{document}
\maketitle

\begin{abstract}  Let $\Ri$ be a complete discrete valuation ring,
$\mS=\Ri[[u]]$ and $n$ a positive integer. The aim of this paper is to
explain how to compute efficiently usual operations such as sum and
intersection of sub-$\mS$-modules of $\mS^d$.  As $\mS$ is not
principal, it is not possible to have a uniform bound on the number of
generators of the modules resulting of these operations. We
explain how to mitigate this problem, following an idea of Iwasawa, by
computing an approximation of the result of these operations up to a
quasi-isomorphism.  In the course of the analysis of the $p$-adic and
$u$-adic precisions of the computations, we have to introduce more
general coefficient rings that may be interesting for their own sake.  Being
able to perform linear algebra operations modulo quasi-isomorphism
with $\mS$-modules has applications in Iwasawa theory and $p$-adic
Hodge theory.  \end{abstract}

\setcounter{tocdepth}{2}
\tableofcontents

\section{Introduction}
\label{sec:intro}

Let $\Ri$ be a complete discrete valuation ring (see \S
\ref{sec:notation} for 
a reminder of the definition) whose valuation is denoted by $v_\Ri$.
Let $K$ denote its fraction field with valuation $v_K$ and $\pi$ be a uniformizer of $\Ri$.
We set $\mS = \Ri[[u]]$; it is the ring of formal series over $\Ri$.  
Our aim is to provide efficient algorithms to deal with finitely 
generated modules over $\mS$. Since, we can always represent a torsion 
module as the quotient of two torsion-free modules, we shall focus on 
torsion-free modules.

Any finitely generated torsion-free $\mS$-module $\M$ can be considered as a
submodule of $\mS^d$ for $d$ big enough.  As a consequence, we can
represent $\M$ by a matrix whose columns are the coefficients of
generators of $\M$ in the canonical basis of $\mS^d$.  Thus we can
reformulate our problem as follows: given $M_1$ and $M_2$ two matrices
representing respectively the $\mS$-modules $\M_1$ and $\M_2$ embedded
in $\mS^d$, give algorithms to compute a matrix representing $\M_1
\cap \M_2$ or $\M_1+\M_2$. We would like also to be able to check
membership, equality of sub-$\mS$-modules, inclusions, \emph{etc.}
As $\mS$ is not a principal ideal domain, in order to control the
number of generators of the sub-$\mS$-modules of $\mS^d$, we propose,
following an idea of Iwasawa, to compute approximations of the
submodules resulting of aforementioned operations in the following
sense: we say that a morphism $\M_1 \rightarrow \M_2$ is a
quasi-isomorphism if its kernel and co-kernel have both finite length,
and we want to make computations modulo quasi-isomorphisms. We propose
two different approaches, each of them having its own advantages and
disadvantages.


First, we notice that there exists a 
correspondence between the set of classes of
modules modulo quasi-isomorphism and modules over the rings $\mS_\pi$
and $\mS_u$ defined respectively as the localization of $\mS$ with
respect to $\pi$ and the completion of the localization of $\mS$ with
respect to $u$.  For $A=\mSp, \mSu$, let $\Free_A^d$ be the set of
free sub-$A$-modules of $A^d$, and denote by $\Mod^d_{\mS/\qis}$ the set of
quasi-isomorphism classes of sub-$\mS$-modules of $\mS^d$, there is an
injective morphism $\Psi': \Mod^d_{\mS/\qis} \rightarrow \Free^d_{\mSp}
\times \Free^d_{\mSu}, \overline{\M} \mapsto (\M \otimes_{\mS} \mSp, \M
\otimes_{\mS} \mSu)$, where $\M$ is any representative in the class
$\overline{\M}$. The image of $\Psi'$ can be precisely characterized
(see Theorem \ref{introtheo:bijection} below). Using this correspondence,
operations with modules with coefficients in $\mS$ reduces to the
computation with modules over $\mSp$ and $\mSu$. As these two last
rings are Euclidean, there exist classical canonical representations
and algorithms to manipulate modules over these rings.

A second approach consists in finding a canonical representative in a
class of modules modulo quasi-isomophism which is amenable to
computations. Such a representative is provided by the \emph{maximal
module} of a $\mS$-module $\M$. It can be defined as the unique free
module in the class of quasi-isomorphism of $\M$.  We present an
algorithm to compute the maximal module associated to a
sub-$\mS$-module of $\mS^d$ which is inspired by a construction of
Cohen, presented in \cite[p.  131]{MR0485768}, to obtain a
classification up to quasi-isomorphism of finitely generated
$\mS$-modules.
We can then compose this algorithm with algorithms to compute basic 
operations on free modules in order to compute with representatives up 
to quasi-isomorphisms.

In order to obtain real algorithms (\emph{i.e.} something computable by 
a Turing machine) we have to consider the fact that elements of $\mS$ 
and its localized are not finite. In this paper we consider an approach 
in two steps in order to solve this problem. First, we give the ability 
to Turing machines, to manipulate, by the way of oracles, elements of 
$\mS$, $\mSp$, $\mSu$. More precisely, we suppose given oracles able to 
store elements of the base ring, compute valuation, multiplication, 
addition, inversion, and Euclidean division. We express the complexity 
of an algorithm with oracle by the number of calls to the oracles to 
compute ring operations. Once we have well defined algorithm with 
oracles to compute with modules, we study in a second time the problem 
of turning them into real algorithms.

Much in the same way as for floating point arithmetic, the actual
computations with modules with coefficients in $\mS$ are done with
approximations up to certain $\pi$-adic and $u$-adic precisions. It is
important to ensure that the (truncated) outputs of our algorithms are
correct which means that they do not depend on the $\pi$ or $u$ powers
of the input that we have forgotten. In order to deal with this
precision analysis, it is convenient to consider a generalisation of
the family of ring coefficients $\mS$.  Namely, given $\alpha,\beta$
relatively prime integers, we write $\nu = \beta/\alpha$ and set
$\mSnu = 
\{ \sum a_i u^i \in K[[u]] | v_K(a_i)+\nu i \geq 0,\, \forall i \in
\N\}$. We have $\mS_0=\mS$. In this
paper, we develop a theory of $\mSnu$-modules which encompass modules
over $\mS$ and use it in order to obtain algorithm with complexity
bounds and proof of correctness.

More precisely, we generalize the definition of a maximal module for
finitely generated torsion-free $\mSnu$-modules. Denote by
$\Max^d_{\mSnu}$ the set of maximal sub-$\mSnu$-modules of $\mSnu^d$.
We prove the following theorem (see Theorem \ref{prop:bijection}),
which generalize the above mentioned decomposition:

\begin{theorem}
\label{introtheo:bijection}
The natural map
$$\begin{array}{rcl}
  \Psi \, : \, \Max_{\mSnu}^d & \longrightarrow &
\Free_{\mSpnu}^d \times \Free_{\mSunu}^d \\
\M & \mapsto & (\M_\pi,\M_u).
\end{array}$$
is injective and its image consists of pairs $(A,B)$ such that $A$ and 
$B$ generate the same $\mE$-vector space in $\mE^d$. If a pair $(A,B)$ 
satisfies this condition, its unique preimage under $\Psi$ is given by 
$A \cap B$.
\end{theorem}
In the theorem, $\mE$ is a field containing $\mSnu$ and its
localized $\mSpnu$ and $\mSunu$ which is precisely defined in Section
\ref{sec:deffirst}.
We give an algorithm with oracles to compute the maximal 
module associated to a finitely generated torsion-free $\mSnu$-module. 
In general, it is not true that the maximal module of a torsion-free 
$\mSnu$-module is free, although this property holds when $\nu=0$.
Nonetheless, by using the theory of continued fraction, it is possible
to obtain a tight upper bound on the number of generators of a maximal
module embedded in $\mSnu^d$. If $\nu$ is rational, it admits a unique finite
development as a continued fraction that we denote by  $[a_0; a_1,
\ldots, a_n]$ (here, we suppose that $a_n\ne 1$). We can prove the following (see Theorem \ref{th:main}):
\begin{theorem}
Let $\nu = [a_0; a_1, \ldots, a_n]$.
Let $\M$ be a sub-$\mSnu$-module of $\mSnu^d$. Then $\Max(\M)$ is 
generated by at most $d\cdot (2+\sum_{i=1}^{\lceil n/2
\rceil} a_{2i})$ elements.
\end{theorem}

We provide some simple examples to show that a lot of basic operations
that we need in order to compute with modules over $\mSnu$, such as the
computation of the Gauss valuation, are not stable. This means that, in
general, the computation with approximations of the input data does not yield
approximation of the result. This is where it becomes
interesting to use the possibility to change the slope $\nu$ of the
base ring $\mSnu$. In the context of our computation, 
a bigger slope plays the role of a loss of precision in the computation
of an approximation of a module over $\mSnu$. In this direction, we
can prove the following theorem (see Theorem \ref{th:main2} for a
precise statement):

\begin{theorem}
  Let $\M_1$ and $\M_2$ be two finitely generated sub-$\mSnu$-modules
  of $\mSnu^d$ such that $\M_2 \subset 1/\pi^c \M_1$ for a positive
  integer $c$. Let $M_1$ and $M_2$ be the
  matrices with coefficients in $\mSnu$ of generators of $\M_1$ and
  $\M_2$ in the canonical basis of $\mSnu^d$. 
  Suppose we are given approximations $M^r_1$ and $M^r_2$ of $M_1$ and $M_2$
  respectively.
  Then, for a well chosen $\nu' >\nu$,
 there exists a polynomial time algorithm in the length of the
representation of $M^r_1$ and $M^r_2$ to compute a matrix 
$M^r_3$ which is an approximation of the maximal module associated to $(\M_1 \otimes_{\mSnu}
\mSnup) +
(\M_2 \otimes_{\mSnu} \mSnup)$.
\end{theorem}

The organisation of the paper is as follows: in \S \ref{sec:arithm}, we 
introduce the rings $\mSnu$, and their basic arithmetic and analytic 
properties. In \S \ref{sec:modules}, we generalize some classical 
results of Iwasawa to the case of finitely generated $\mSnu$-modules and 
then give an algorithm with oracle to compute the maximal module 
associated to a torsion-free $\mSnu$-module and obtain an upper bound on 
the number of generators of a maximal module. Note that \S 
\ref{sec:arithm} and \S \ref{sec:modules}, we only describe algorithms 
with oracles. In \S \ref{sec:precision}, we study the problem of 
$p$-adic and $u$-adic precisions and turn the algorithms with oracles 
obtained in the previous sections into real algorithms.

\section{Arithmetic of the rings  $\mSnu$}\label{sec:arithm} In order to compute with
modules over $\mSnu$ we first have to study the basic arithmetic
properties of their base ring. In this section, we show that
its localized with respect to
$u^\alpha/\pi^\beta$
and $\pi$ becomes Euclidean. We provide algorithms with oracles to compute the
Euclidean division in these rings which will be very useful for our purpose along with their
complexity expressed in term of the number of ring operations. They will 
be turned into real algorithms (\emph{i.e.} working on a real Turing 
machine) in \S \ref{sec:precision} where we study the problem of 
precision of computation in the rings $\mSnu$.

\subsection{Notations}\label{sec:notation}

We fix the notations for the rest of the paper. Let $\Ri$ be a ring 
equipped with a discrete valuation $v_\Ri$, that is a map 
$v_\Ri : \Ri \to \N_{\geq 0} \cup \{+\infty\}$ satisfying the following
conditions:
\begin{itemize}
\item for all $x \in \Ri$, $v_\Ri(x) = +\infty$ if and only if $x = 0$;
\item for all $x \in \Ri$, $v_\Ri(x) = 0$ if and only if $x$ is invertible;
\item for all $x,y \in \Ri$, $v_\Ri(xy) = v_\Ri(x) + v_\Ri(y)$;
\item for all $x,y \in \Ri$, $v_\Ri(x+y) \geq \min (v_\Ri(x), v_\Ri(y))$.
\end{itemize}
Let $a$ be a fixed real number in $(0,1)$. One can define a 
distance $d$ on $\Ri$ by the formula $d(x,y) = a^{v_\Ri(x-y)}$ ($x,y \in \Ri$) 
where we use the convention that $a^{+\infty} = 0$. For the rest of the 
paper, we assume that $\Ri$ is complete with respect to $d$. We recall
that $\Ri$ is a local ring whose maximal ideal is
$\mathfrak{M}=\{ x \in \Ri | v_\Ri(x) > 0 \}$. Up to renormalizing $v_\Ri$, 
it is safe to assume that it is surjective, what we do.
We denote by $\pi$ a \emph{uniformizer} of $\Ri$, 
that is an element of $\Ri$ whose valuation is $1$. Every element $x$ in
$\Ri$ can then be written $x = \pi^r u$ where $r = v_\Ri(x)$ and $u \in \Ri$
is invertible.
Here are several classical examples of such rings $\Ri$:
\begin{itemize}
\item the ring $\Z_p$ of $p$-adic integers equipped with
the usual $p$-adic valuation;
\item more generally, the ring of integers of any finite extension 
of $\Q_p$;
\item for any field $k$, the ring $k[[u]]$ of formal power series with
coefficients in $k$.
\end{itemize}
We now go back to a general $\Ri$. It follows easily from the definition
that the field of fractions of $\Ri$ is just $\Ri[1/\pi]$. Let's denote it
by $K$ and set $\mS = \Ri[[u]]$, the ring of formal series over
$\Ri$. The valuation $v_\Ri$ can be extended uniquely to a valuation
$v_K$ on $K$.

\subsection{Definition and first properties of
$\mSnu$}\label{sec:deffirst}

Denote by $K[[u]]$ the power series ring with coefficients in $K$.  It
is classical to define the Gauss valuation of an element $\sum a_i u^i
\in K[[u]]$ as the smallest $v_K(a_i)$ if it exists. The ring of
elements of $K[[u]]$
with non negative Gauss valuation is nothing but $\Ri[[u]]$.  In this
section, we are
going to consider more generally a family of valuations parametrized by a
slope $\nu \in \Q$ so as to define the subring of $K[[u]]$ of elements with positive
valuation. 

\begin{definition}
Let $\nu \in \Q$.  We define the \emph{Gauss valuation} $v_\nu: K[[u]]
\rightarrow \Q \cup\{+\infty, -\infty \}$ by $v_\nu(x)=+\infty$ if
$x=0$, $v_\nu(\sum a_i u^i)=\min \{ v_K(a_i)+\nu i, i \in \N \}$ if
$x\neq0$ and this minimum exists and $v_\nu(x)=-\infty$ otherwise.
The \emph{Weierstrass degree} of $x$ denoted $\deg^{\nu}_W(x)$ is given by
$\deg^{\nu}_W(0)=-\infty$, $\deg^{\nu}_W(x)=\min\{ i | v_K(a_i)+\nu i = v_\nu(x)\}$ if
$v_\nu(x) \neq -\infty$ and $\deg^{\nu}_W(x)=+\infty$ otherwise. When
no confusion is possible, we will use the notation $\deg_W$ instead of
$\deg^{\nu}_W$.
\end{definition}

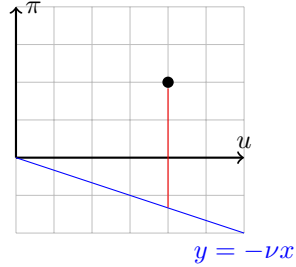
\begin{figure}[h]
\begin{center}
\begin{tikzpicture}
\draw[step=0.5cm,lightgray, very thin] (0,-1) grid (3,2);
\draw[->, thick] (0,0) -- (0,2);
\draw[->, thick] (0,0) -- (3,0);
\draw[color=blue] (0,0) -- (3,-1) node[anchor=north] {$y=-\nu x$};
\node[circle,fill=black,inner sep=1.5pt] (X) at (2,1) {};
\node[right] (p) at (0,2) {$\pi$};
\node[above] (u) at (3,0) {$u$};
\draw[red] (X) -- (2,-0.66);
\end{tikzpicture}
\end{center}
\caption{The Gauss valuation of $\pi^2\cdot u^4$ with $\nu=1/3$ is $10/3$.}
\end{figure}

The following lemma gives some basic properties of $v_\nu$ and
$\deg_W$. In particular, it shows that $v_\nu$ has the usual
properties of a valuation:
\begin{lemma}\label{sec2:lemma0}
For $x,y \in K[[u]]$ we have:
\begin{enumerate}
\item $v_\nu(x)=+\infty$ if and only if $x=0$;
\item $v_\nu(x\cdot y)=v_\nu(x)+v_\nu(y)$;
\item $v_\nu(x+y) \geq \min ( v_\nu(x), v_\nu(y) )$.
\end{enumerate}
Moreover for all $x,y \in K[[u]]$ with finite Gauss valuation,  $\deg_W(x.y)=\deg_W(x)+\deg_W(y)$.
\end{lemma}
\begin{proof}
To prove 2., we first suppose that $x =\sum a_i u^i$ and $y=\sum b_i
u^i$ have finite valuation. Let $z=x\cdot y= \sum c_i u^i$. We have
$v_K(c_i)+\nu i = v_K(\sum_{j=0}^i a_j\cdot b_{i-j})+\nu i \geq \min_j \{
v_K(a_j)+ \nu \cdot j + v_K(b_{i-j}) + \nu\cdot (i-j) \} \geq v_\nu(x)+ v_\nu(y)$.
Moreover, by taking $i=\deg_W(x)+\deg_W(y)$ in the previous computation, we
obtain that $v_K(c_{\deg_W(x)+\deg_W(y)})=v_\nu(x)+v_\nu(y)$. If
$v_\nu(x)=-\infty$ and $y \ne 0$, we can apply the previous result to the series
obtained by truncating $x$ up to a certain power to show that
$v_\nu(x\cdot y) = -\infty$. The proof of the rest of the lemma is left to
the reader.
\end{proof}

We let $\mSnu = \{ x \in K[[u]] | v_\nu(x) \geq 0 \}$. By
definition, an element $x \in \mSnu$ can we written as a series $$x
= \sum_{i \in \N} a_i u^i,$$ where $a_i \in K$ and $v_K(a_i) \geq -\nu
i$.

\begin{remark}
It is clear that $\mSnu$ is complete for the valuation $v_\nu$.
Nonetheless, the ring $\mSnu$ is not a valuation ring. In fact, although $v_\nu
(u^\alpha/\pi^\beta)=0$ for $\nu \neq 0$ (resp. $v_\nu(u)=0$ for $\nu
=0$), $u^\alpha/\pi^\beta$ (resp. $u$) is not invertible in $\mSnu$.
\end{remark}

We let $$\mSpnu = \mSnu [1/\pi]=\big\{ \sum_{i \in \N} a_i u^i,
a_i \in K  \, \text{such that} \, v_K(a_i)+\nu i \,
\text{bounded below} \big\}.$$ 
In the same way, it is clear that one can extend the $v_\nu$ valuation
of $\mSnu$ over $\mSnu[\pi^\beta/u^\alpha]$ and we let $\mSunu = \widehat{\mSnu
[\pi^\beta/u^\alpha]}$ where the hat stands for the completion of $\mSnu
[\pi^\beta/u^\alpha]$ with
respect to the topology defined by $v_\nu$. 

Put in another way, $$\mSunu = \big\{ \sum_{i \in \Z} a_i u^i, a_i \in \mSnu
\, \text{and} \, \lim_{ i
\rightarrow -\infty} v_K(a_i)+\nu i = +\infty \big\}.$$
We moreover define
$$\mE = \big\{ \sum_{i \in \Z} a_i u^i, a_i \in K \, v_K(a_i)+\nu i \,
\text{bounded below} \, \text{and} \lim_{i \rightarrow -\infty}
v_K(a_i)+\nu i =+\infty \big\}.$$

We have the following commutative diagram of inclusions:
\begin{equation}\label{sec2.1:diaginclus}
\raisebox{-1cm}{
\begin{tikzpicture}[normal line/.style={->}, font=\scriptsize]
\matrix(m) [ampersand replacement=\&, matrix of math nodes, row sep =
  1em, column sep =
2.5em]
{  \& \mSpnu \&  \\
  \mSnu   \&  \& \mE \\
\& \mSunu \& \\};
\path[normal line]
(m-2-1) edge node[auto] {} (m-1-2)
(m-2-1) edge node[auto] {} (m-3-2)
(m-1-2) edge node[auto] {} (m-2-3)
(m-3-2) edge node[auto] {} (m-2-3);
\end{tikzpicture}
}
\end{equation}

As $\mSpnu$ is a subring of $K[[u]]$, it is equipped with the $v_\nu$
valuation and the Weierstrass degree associated to $v_\nu$. Moreover,
one can extend, in an obvious manner, the definition of $v_\nu$ and
the Weierstrass degree for $\mSunu$ and $\mE$. 

We can interpret the ring $\mSnu$ in terms of the analytic functions
on the $\pi$-adic disc.  In order to explain this, for $\nu \in \Q$,
we consider the open disk $D_\nu = \{ x \in K | v_K(x) > \nu \}$.
Denote by $\mO_\nu$ the ring of convergent series $\mO_\nu = \{
\sum_{i \in \N} a_i u^i | a_i \in K, \liminf_{i \rightarrow +\infty}
\frac{v_K(a_i)}{i} \geq -\nu \}$ in the disk $D_\nu$. It is clear that
$\mSpnu$ is exactly the set $\{f \in K[[u]] |\,
v_K(f(x))\,\text{bounded below on}\, D_\nu \}$ and $\mSnu$ can be
described as $\{f \in K[[u]]| \, v_K(f(x))\, \text{bounded below by 0
on}\, D_\nu \}$.  Thus, there are obvious inclusions $\mSnu \subset
\mSpnu \subset \mO_\nu$ but one should beware of the fact that the
last
inclusion is strict. Indeed for instance, for $\Ri=\Z_p$, $\nu=0$ the
series $\sum_{i > 0} \frac{u^i}{i}$ which defines the function
$\log(1-u)$ is convergent in the unity disk but is obviously not in
$S_{0,\pi}$ since $v_\pi(1/i)$ has no lower bound. Assuming that
$\nu$ is rational (what we do), the following proposition gives another
characterisation of elements of $\mO_\nu$ that lies in $\mSpnu$.

\begin{proposition}\label{prop:analytic} An element $x \in \mO_\nu$ is in $\mSpnu$ if and
only if $x$ has only a finite number of zeros in $\mO_\nu$.
\end{proposition}

\begin{proof}
Let $x \in \mO_\nu$. The number of zeros of $x \in D_\nu$ is equal to
the length of the interval above which the Newton polygon of $x$ has a
slope $<-\nu$. If this length is finite, it is clear that $v_p(a_i)$
is bounded below by a line of the form $-\nu i +c$ with $c$ a constant
and as a consequence is an element of $\mSpnu$.

Conversely, suppose that $x \in \mSpnu$. This means that
$v_p(a_i)+\nu i$ is bounded below and is contained in $\Z+\nu \Z$
which is a discrete subgroup of $\R$ (as $\nu$ is rational).
Thus, the set $\{v_p(a_i)+\nu i,i \in N \}$ reaches a
minimum for a certain index $i_0$. This means that for all $i > i_0$,
the slope of the Newton polygon of $x$ is greater than $-\nu$ and $x$
has a finite number of zeros in $D_\nu$.
\end{proof}

We end up this section, by remarking that up to an extension of the
base ring $\Ri$ all the $\mSnu$'s are isomorphic to a $\mS_0$. Indeed,
write $\nu=\beta/\alpha$ with $\alpha,\beta$ relatively prime numbers
and let $\Rpi$, in an algebraic closure of $K$, be such that
$\Rpi^\alpha=\pi$.  Let $\Ri'=\Ri[\Rpi]$, $K'$ be the fraction field
of $\Ri'$ (and a finite extension of $K$). The valuation on $\Ri$
extends uniquely on $\Ri'$ by setting $v_{K'}(\Rpi)=1/\alpha$.  For
$\mu =0,\nu$, let ${\mS_\mu}'=\mS_\mu \otimes_{\Ri} \Ri'$.  The
valuation $v_{K'}$ defines a Gauss valuation on ${\mS_\mu}'$ that we
denote also by $v_\mu$.

\begin{lemma}\label{sec1:lempiso}
  Keeping the notations from above, the morphism of ring $\rho : {\mS_0}'
  \rightarrow {\eSnu}$, defined by $\rho(1)=1$ and
  $\rho(u)=\frac{u}{\Rpi^{\beta}}$
  is an isomorphism. Moreover, if $x \in {\mS_0}'$ we have
  $v_0(x)=v_\nu(\rho(x))$ and
  $\deg^{0}_W(x)=\deg^{\nu}_W(\rho(x))$.
\end{lemma}
\begin{proof}
  By definition, ${\eSnu}=\{ \sum a_i u^i | v_{K'}(a_i)+\nu i
  \geq 0 \}=\{ \sum a_i (u/\Rpi^{\beta})^i | v_{K'}(a_i) \geq 0 \}$
  from which it is clear that $\rho$ is an isomorphism. The rest of
  the lemma is an easy verification.
\end{proof}

\subsection{Division in $\mSnu$}

The Weierstrass degree allows us to describe an Euclidean division in
$\mSnu$. Although, the existence of such a division is classical (see
for instance \cite{MR0485768}) at least over $\mS_0=\Ri[[u]]$, we give
here a proof for all $\nu$ which provides an algorithm with oracles
to compute the Euclidean division. 

In order to study divisibility in $\mSnu$, we have a first result:
\begin{lemma}\label{sec1:lemma2}
Let $x,z\in \mSnu$. We suppose that $\deg_W(x)=0$ then there exists $y
\in \mSnu$ such that $x.y=z$ if and only if $v_\nu(x) \leq
v_\nu(z)$. 
\end{lemma}
\begin{proof}
We suppose that $\deg_W(x)=0$.  If there exists $y \in \mSnu$ such
that $x\cdot y=z$ then clearly $v_\nu(x) \leq v_\nu(z)$.  Reciprocally, 
we suppose that $v_\nu(x) \leq v_\nu(z)$. Write $x=\sum_{i \in \N} a_i
u^i$ and $z=\sum_{i \in \N} c_i u^i$. 
Since $a_0$ is invertible in $K$ there exists $y\in K[[u]]$ such that $x.y=z$.
We have to prove that $v_\nu(y)\geq
0$. For this, write  $y=\sum_{i \in \N} b_i u^i$. We have $v_K(b_0) =v_K(c_0)-v_K(a_0) \geq 0$ by hypothesis. Then,
for $j \geq 1$, we prove by induction that $v_K(b_j)+\nu j \geq 0$. We
have $b_j=a_0^{-1}\cdot c_j-a_0^{-1} \sum_{i=1}^j a_i\cdot b_{j-i}$.  But
$v_K(a_0^{-1}\cdot c_j) + \nu j \geq v_\nu(z) - v_\nu(x) \geq 0$ because
$\deg_W(x)=0$.
Moreover, for $i=1\ldots j$, $v_K(a_0^{-1}\cdot a_i\cdot b_{j-i})+\nu
j=v_K(a_i)+\nu i-v_\nu(x)+v_K(b_{j-i})+\nu (j-i)$. But by definition
$v_K(a_i)+\nu i -v_\nu(x) \geq 0$ and by the induction hypothesis
$v_K(b_{j-i})+\nu (j-i)\geq 0$.  Therefore, $v_K(b_j)+\nu j \geq 0$
and we are done.
\end{proof}
Applying Lemma \ref{sec1:lemma2} to $z=1$, we get
\begin{corollary}\label{sec2:cor1}
Let $x=\sum_{i \in \N} a_i x^i \in \mSnu$, then $x$ is invertible in
$\mSnu$ if and only if $\deg_W(x)=0$ and $v_\nu(x)=0$. 
\end{corollary}

We note that the corollary implies that $\mSnu$ is a local ring.  Next, we
introduce the following notations: for $x = \sum_{i \in \N} a_i u^i
\in \mSnu$ and $d$ a positive integer, we let $\Hi(x,d)=\sum_{i \geq
d} a_i u^i$ and $\Lo(x,d)=\sum_{i=0}^{d-1} a_i u^i$. It is clear that
$x=\Lo(x,d)+\Hi(x,d)$.

\begin{proposition}\label{sec2:prop1}
Let $x,y \in \mSnu$. Suppose that $v_\nu(y) \geq v_\nu(x)$ then
there exist a unique couple $(q,r) \in \mSnu \times (K[u] \cap \mSnu)$ such that 
$\deg(r) < \deg_W(x)$ and $y=q\cdot x+r$.
\end{proposition}
\begin{proof}
First, we prove the existance of $(q,r)$.
  Let $d=\deg_W(x)$, we consider the sequences 
$(q_i)$ and $(r_i)$ defined by $q_0=0$ and $r_0=y$ and
\begin{equation}\label{eq:euclidean}
q_{i+1}=q_i + \frac{\Hi(r_i,d)}{\Hi(x,d)}, r_{i+1}=r_i -
\frac{\Hi(r_i,d)}{\Hi(x,d)}\cdot x.
\end{equation}

We are going to prove by induction that $q_i$ and $r_i$ are convergent
sequences (for the $v_\nu$ valuation) of elements of $\mSnu$. 
Let $e=v_\nu(\Lo(x,d))-v_\nu(\Hi(x,d))>0$.
Our
induction hypothesis is that $q_i$ and
$r_i$ are elements of $\mSnu$, that $v_\nu(\Hi(r_i,d)) \geq
e\cdot i+v_\nu(\Hi(y,d))$ and that $y=q_i\cdot x+r_i$. It is clearly true for $i=0$.

By the induction hypothesis, we have $v_\nu(\Hi(r_i,d)) \geq
v_\nu(\Hi(y,d))$ and by hypothesis $v_\nu(\Hi(y,d)) \geq v_\nu(y) \geq
v_\nu(x)=v_\nu(\Hi(x,d))$ so that $v_\nu(\Hi(r_i,d)) \geq
v_\nu(\Hi(x,d))$.
Applying Lemma \ref{sec1:lemma2}, we obtain
$\frac{\Hi(r_i,d)}{\Hi(x,d)} \in \mSnu$ and then $q_{i+1},r_{i+1} \in
\mSnu$. Next writing $x=\Hi(x,d)+\Lo(x,d)$, we get
\begin{equation}\label{sec2:prop1:eq1}
r_{i+1}=\Lo(r_i,d)-\frac{\Hi(r_i,d)}{\Hi(x,d)}\cdot \Lo(x,d).
\end{equation}
Applying Lemma \ref{sec2:lemma0}, we obtain that
$v_\nu(\Hi(r_{i+1},d))\geq v_\nu(\Hi(r_i,d))+v_\nu(\Lo(x,d))-v_\nu(\Hi(x,d))$.
Using the induction hypothesis, we get that $v_\nu(\Hi(r_{i+1},d))\geq
e\cdot (i+1)+v_\nu(\Hi(y,d))$. Finally, using the hypothesis that 
$y=q_i\cdot x+r_i$, we immediately check using (\ref{eq:euclidean})
that $y=q_{i+1}\cdot x+r_{i+1}$.

From the induction, we deduce that $q_i$ and $r_i$ are convergent
sequences of $\mSnu$ for the $v_\nu$ valuation.  In fact, we have
$q_{i+1}-q_i=\frac{\Hi(r_i,d)}{\Hi(x,d)}$ so that $v_\nu(q_{i+1}-q_i)
= v_\nu(\Hi(r_i,d)) - v_\nu(\Hi(x,d))\geq e\cdot
i+v_\nu(\Hi(y,d))-v_\nu(\Hi(x,d))\geq e\cdot i$. The same argument
works for $r_i$.  Denote by $q$ and $r$ the limits.  As for all $i \in
\N$, $y=q_i\cdot x+r_i$, we have $y=q\cdot x+r$. Moreover, since
$\Hi(r_i,d)\geq e\cdot i$, we have $\Hi(r,d)=0$, so that $r\in K[u]$
and $\deg(r) < \deg_W(x)$.

We prove the unicity of $(q,r)$. Let $(q',r') \in \mSnu \times (K[u]
\cap \mSnu)$ such that $y=q'\cdot x+r'$. Then $(q-q')\cdot x=r'-r$.
We have $\deg_W((q-q')\cdot x)=\deg_W(r'-r) < \deg_W(x)$ which is only
possible if $q=q'$ and $r=r'$.
\end{proof}

From the proof of Proposition \ref{sec2:prop1}, we deduce Algorithm
\ref{algo1} to compute from the knowledge of $x,y$, the elements
$q',r' \in \mSnu$ such that $v_\nu(q-q') \geq prec$  and $v_\nu(r-r')
\geq prec$. Furthermore, by the proof of the proposition, the number of
iterations of the while loop is bounded by $\lceil prec/e \rceil$. We
deduce that Algorithm \ref{algo1} needs one inversion and $3\cdot \lceil
prec/e \rceil$ multiplications in $\mSnu$.

\begin{algorithm}\label{algo1}
\SetKwInOut{Input}{input}\SetKwInOut{Output}{output}

\Input{$x,y \in \mSnu$ with $v_\nu(y) \geq v_\nu(x)$, $prec \in \N$}
\Output{$q,r \in \mSnu$ such that $y=q\cdot x+r$ and $v_\nu(\Hi(r,\deg_W(x))) \geq prec$}
\BlankLine
$q \leftarrow 0$\;
$r \leftarrow y$\;
$d \leftarrow \deg_W(x)$\;
\While{$v_\nu(\Hi(r,d))<prec$}{
$q \la q+\frac{\Hi(r,d)}{\Hi(x,d)}$\;
$r \la r-\frac{\Hi(r,d)}{\Hi(x,d)}\cdot x$\;
}
\Return $q,r$\;
\caption{EuclieanDivision}\label{algo:euclidean}
\end{algorithm}

Now, let $x \in \mSnu$, following \cite{MR0485768} we say that $x$ is
\emph{distinguished} if $v_\nu(x)=0$. With this definition, we can state the
classical Weierstrass preparation theorem:
\begin{corollary}[Weierstrass preparation]\label{cor:weierstrass}
Let $x \in \mSnu$ be a distinguished element and let $d=\deg_W(x)$.
Then we can write 
$x=q\cdot h$, where $q \in \mSnu$ is an invertible element and $h \in
K[u] \cap \mSnu$ is of the form
$h=\frac{u^d}{\pi^{\nu\cdot d}}+ \sum_{i=0}^{d-1} b_i u^i$ with $v_K(b_i)+\nu
i > 0$.
\end{corollary}
\begin{proof}
We first notice that $d \nu$ is a nonnegative integer. Indeed, it is 
clearly nonnegative, and writing $x = \sum a_d u^d$, we have $v_\Ri(a_d) 
+ d \nu = 0$ (since $x$ is assumed to be distinguished) and,
consequently, $d \nu = - v_\Ri(a_d) \in \Z$.

By proposition \ref{sec2:prop1}, there exist $q \in \mSnu$ and $r
\in K[u] \cap \mSnu$ such that $\deg r < d$ and
$$\frac{u^d}{\pi^{d\cdot \nu}} = q\cdot x+r.$$
Using Lemma \ref{sec2:lemma0}, we obtain $v_\nu(q)=0$ and
$\deg_W(q)=0$. Then, Corollary \ref{sec2:cor1} implies that $q$
is invertible. To finish the proof it suffices to remark that
$\deg_W(\frac{u^d}{\pi^{d\cdot \nu}}-r)=d$ and the result follows from the
definition of $\deg_W$.
\end{proof}

\begin{remark}
The previous proposition is closely related to the Proposition
\ref{prop:analytic} since it says that an element of $\mO_\nu$ is in
$\mSpnu$ if and only if it can be written as product of a polynomial
times a function which does not have any zero in $D_\nu$.
\end{remark}

The following proposition states that the rings $\mSpnu$ and $\mSunu$ 
are Euclidean rings and provides algorithms with oracles to compute the 
division.

\begin{proposition}
The ring $\mSpnu$ is Euclidean, the ring $\mSunu$ is a discrete
valuation ring for the valuation $v_\nu$ (and as a consequence  is also
Euclidean). 
Moreover, $\mE$ is a field.
\end{proposition}
\begin{proof}

Let $x,y \in \mSpnu$. There exist $s,t \in \N$ such that $\pi^s x, \pi^t 
y \in \mSnu$ and $v_\nu(\pi^t\cdot y) \geq v_\nu(\pi^s\cdot x)$. 
Applying Proposition \ref{sec2:prop1}, yields $q \in \mSnu$ and $r \in 
K[[u]] \cap \mSnu$ such that $\deg(r) < \deg_W(x)$ and $y=\pi^{s-t}\cdot 
q\cdot x+\pi^{-t}\cdot r$ and we are done.

In order to prove that $\mSunu$ is a discrete valuation ring, we have
to show that the set of invertible elements of $\mSunu$ is the set of
elements $x \in \mSunu$ such that $v_\nu(x)=0$. Write $\nu
=\beta/\alpha$, with $\alpha, \beta$ relatively prime numbers.  Let
$\mathfrak{m}$ be the ideal defined by $\{ x \in \mSunu, v_\nu(x) > 0
\}$, it is clear that $\mSunu / \mathfrak{m}$ is isomorphic to the
field $k((u^\alpha))$. As $\mSunu$ is complete for the $v_\nu$
valuation, the Hensel lift algorithm gives an algorithm with oracles to compute the
inverse of an element whose valuation is zero. The Algorithm
\ref{algo2} uses a fast Newton iteration to perform this computation
modulo $\mathfrak{m}^n$ at the expense of $O(\log(n))$ multiplications
in $\mSunu$. 

Let $x$ be a non zero element of $\mE$, by dividing it by a power of
$\pi$ we can suppose that $v_\nu(x)=0$ and by using the algorithm with
oracle Algorithm \ref{sec2:algo1}, we can invert it. 
\end{proof}

\begin{algorithm}\label{algo2}
\SetKwInOut{Input}{input}\SetKwInOut{Output}{output}

\Input{$x \in \mSunu$ such that $v_\nu(x)=0$, $n \in \N$}
\Output{$y \in \mSunu$ such that $x\cdot y=1 \mod \mathfrak{m}^n$}
\BlankLine
\eIf{$n=1$}{
$y \leftarrow 1/\overline{x} \mod \mathfrak{m}$\;
}{
$y \leftarrow \text{Inverse}(x,\lceil n/2 \rceil)$\;
$y \leftarrow y+y(1-ay) \mod \mathfrak{m}^n$\;
}
\caption{Inverse}\label{sec2:algo1}
\end{algorithm}

\begin{remark}

  One can use the usual Euclidean algorithm to compute the B\'ezout
  coefficients of $x,y \in \mSpnu$. This algorithm outputs $g,k,l,m,n
  \in \mSpnu$ such that $g$ is the greatest common divisor of $x$ and
  $y$, $k\cdot x+l\cdot y=g$, $m\cdot x+n\cdot y=0$ and $k\cdot n-l\cdot m=1$. It proceeds by
  using the fact that $gcd(x,y)=gcd(y,r)$ where $r$ is the rest of the
  division of $x$ by $y$ and uses $O(\deg_W(y))$ calls to the
  Euclidean division Algorithm \ref{algo1}. We remark, as the rest of
  the division of two elements of $\mSnu$ is an element of $K[u]$,
  that starting from the second iteration of this algorithm all the
  divisions to be computed are the usual division between elements of
  $K[u]$. Unfortunately, we will see that in \S
  \ref{sec:precision}, that the
  Euclidean algorithm in general is not stable, so that we might need
  extra informations, about $x$ and $y$ in order to compute an
  approximation of their gcd from the knowledge of an approximation of
  $x$ and an approximation of $y$.
\end{remark}

\section{Modules over $\mSnu$}\label{sec:modules}

Let $d$ be a positive integer and fix $\nu \in \Q$. We want to compute
with finitely generated torsion free $\mSnu$-modules. Any such module
$\M$ can be embedded in $\mSnu^d$ for $d \in \N$ and can be
represented by a matrix with coefficients in $\mSnu$ whose column
vectors are the coordinates of generators of $\M$ in the canonical
basis of $\mSnu^d$.  Indeed, we can always embed $\M$ is $\M
\otimes_{\mSnu} \Frac(\mSnu)$ and select a basis $(e_1, \ldots, e_d)$ 
of $\M \otimes_{\mSnu} \Frac(\mSnu)$
together with an element $D \in \mSnu$ such that the image of $\M$ in
$\M \otimes_{\mSnu} \Frac(\mSnu)$ is contained in the free $\mSnu$-module
generated by the $\frac 1 D . e_i$'s.

A first problem arises here: it is not possible to bound the number of
generators of the submodules of $\mSnu^d$ that we have to compute
with.  For instance, for $d=1$ and $\nu = 0$, choose a positive integer $k$ and
consider the sub-$\mS_0$-module $\M_k$ of $\mS_0$ generated by the
family $(\pi^{k-j} u^{j})_{j=0, \ldots, k}$.  Then $\M_k$ can not be
generated by less than $k+1$ elements. Indeed, let $(e_0, \ldots,
e_n)\in \mS_0^n$ be a family of generators of $\M_k$, and for $j \geq 0$
and define a filtration on $\M_k$ by letting $F^j \M_k = \M_k \cap u^j 
\mS_0$. We are going to 
prove by induction on $t \in \{0, \ldots, k\}$ that there exists a
matrix $M_t \in M_{n \times n}(\mS_0)$ such that, if we set $(e'_0,
\ldots, e'_n)=(e_0, \ldots, e_n)\cdot M_t$ then $(e'_0, \ldots, e'_n)$
is a family of generators of $\M_k$, for $j < t$, $e'_j=u^j
\pi^{k-j} \mod F^{j+1}\M_k$ and $(e'_j)_{j \geq t}$ is a family of generators of $F^t
\M_k$. This is obviously true for $t=0$. Suppose that it is true
for $t_0 \in \{0, \ldots, k\}$.  Let $(e'_0, \ldots, e'_n)= (e_0,
\ldots, e_n)\cdot M_{t_0}$.  As the morphism $(\sum_{j=t_0}^k \mS_0
e'_j) / F^{t_0+1}\M_k \rightarrow \pi^{k-t}\Ri$, defined by $u^{t_0}
\sum a_i u^i \mapsto a_0$ is an isomorphism, we can suppose if
necessary by renumbering the family $(e'_i)$ that
$e'_{t_0}=u^{t_0} \pi^{k-t_0} \mod F^{t_0+1}\M_k$.  Then, by
considering linear combinations of the form $e'_j - \lambda
e'_{t_0+1}$ for $\lambda \in \mS_0$ for $j > t_0$, one can obtain a matrix
$M_{t_0+1}$ satisfying the induction hypothesis for $t_0+1$. Finally, we
get $n \geq k$.

A second problem comes from the fact that there is no unique way to 
represent a module by a set of generators.  For computational purpose, 
in order to check equality between modules for instance, it is important 
to have a \emph{canonical representation}, that is a bijective 
correspondence between mathematical objects and data structures. An 
example of such a canonical representation exists for finitely generated 
modules with coefficients in an Euclidean ring (\cite{MR1228206}): it is 
the so-called Hermite Normal Form (HNF). It is given by a triangular 
matrix (with some extra conditions) that can be computed from an initial 
matrix $M$ by doing operations on column vectors of $M$.  Even if 
$\mSnu$ is not Euclidean, we could have hoped that such representations 
still exist for free modules.  Unfortunately, it turns out that it is 
not the case. Indeed, in general, there does not even exist a triangular 
matrix form for matrices over $\mSnu$. For instance, for $\nu=0$, take:
$$M = \Big( \begin{matrix}
  u & \pi \\ \pi & u \end{matrix} \Big) \in M_{2 \times 2}(\mS_0)$$
  and assume that $M$ can be written as a product $LP$ where $L$ is
  lower-triangular and $P$ is invertible. Let $\alpha$ and $\beta$ be
  the diagonal entries of $L$. Then, $\alpha$ and $\beta$ belong to
  the maximal ideal of $\mS_0$ (since the coefficients of $M$ all
  belong to this ideal) and the product $\alpha \beta$ is equal to a
  unit times $u^2 - \pi^2$.  Hence, by multiplying $\beta$ by an
  invertible element in $S_0$ if necessary, we can assume that
  $\beta = u \pm \pi$ since $S_0$ is a unique factorisation domain.
  On the other hand, by hypothesis, there exist $a, b \in \mS_0$ such
  that $u a + \pi b = 0$ and $\pi a + ub = \beta$. This equality
  implies that $\pi$ divides $a$ and therefore that $\beta = \pi a + u b
  \in u \mS_0 + \pi^2 \mS_0$.  This is a contradiction.

In this section, we explain how to get around these problems.  First,
we recall the notion of quasi-isomorphism and study the localisation
of the modules with respect to $\pi$ or $u^\alpha/\pi^\beta$ in order to obtain
canonical representations well suited for the computation in the
category of modules up to quasi-isomorphism. Then, we
describe a generalisation of an algorithm of Cohen to compute the
maximal module associated to a given torsion-free $\mSnu$-module and obtain a
bound on the number of generators of a maximal $\mSnu$-module.
We explain how to combine the different approaches in order to obtain a
comprehensive algorithmic toolbox for modules over $\mSnu$.

\subsection{Quasi-isomorphism and maximal modules}
In order to be able
to control the number of generators of a $\mSnu$-module,
we are going to compute up to finite modules which will be considered
as "negligible". 

\begin{definition}
A finitely generated $\mSnu$-module is said to be \emph{finite} if it has
finite length.
Let $\M$ and $\M'$ be two finitely generated $\mSnu$-modules, let $f :
\M \to 
\M'$ be a $\mSnu$-linear morphism. We say that $f$ is a 
\emph{quasi-isomorphism} if its kernel and its co-kernel are finite
modules.
\end{definition}

\begin{remark}
Since $\ker f$ and $\coker f$ are finitely generated (because $\mSnu$ is a 
noetherian ring), it is easy to check that they have finite length if 
and only if they are canceled, at the same time, by a distinguished
element of $\mSnu$ and a power of $\pi$. A quasi-isomorphism between torsion-free modules is
always injective. Indeed, its kernel, being a submodule annihilated by a
power of $u^\alpha/\pi^\beta$ and $\pi$ of a torsion free module, is zero.
\end{remark}

\begin{example}\label{mysimpleexample}
  Let $\M$ be the submodule of $\mS_0$ generated by $(\pi^2,\pi  u^3)$.
  The inclusion $\M \subset \pi \mS_0$ yields an injective morphism whose
  is annihilated by $\pi$ and $u^3$. As a consequence $\M$
  is quasi-isomorphic to the free module $\pi.\mS_0$ (see figure
  \ref{figquasi}).
\end{example}

\begin{figure}[h]
\begin{center}
\begin{tikzpicture}
\draw[step=0.5cm,lightgray, very thin] (0,0) grid (3,2);
\draw[->, thick] (0,0) -- (0,2);
\draw[->, thick] (0,0) -- (3,0);
\node[right] (p) at (0,2) {$\pi$};
\node[above] (u) at (3,0) {$u$};
\node[circle,color=red,fill=red,inner sep=1.5pt] (X) at (1.5,0.5) {};
\node[circle,color=red,fill=red,inner sep=1.5pt] (Y) at (0,1) {};
\node[below] (u3) at (X) {$u^3$};
\node[left] (p2) at (Y) {$\pi^2$};
\draw[color=red] (0,2) -- (0,1) -- (1.5,1) -- (1.5,0.5) -- (3,0.5);
\draw[pattern=north east lines, thin] (0,1) rectangle (1.5,0.5);
\end{tikzpicture}
\end{center}
\caption{The module $\M$ is quasi-isomorphic to $\pi\cdot \mS_0$.}
\label{figquasi}
\end{figure}
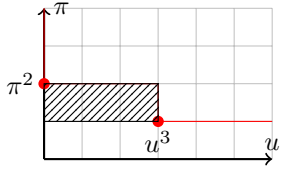

We have a canonical representative in a class of quasi-isomorphism 
which is given by the following definition.
\begin{definition}\label{def:max}
Let $\M$ be a torsion-free finitely generated
$\mSnu$-module. We say that $\M'$ together with a
quasi-isomorphism $f : \M \rightarrow \M'$ is \emph{maximal} for $\M$ if for
every $\Nc$, torsion-free $\mSnu$-module, and quasi-isomorphism
$f' : \M \rightarrow \Nc$, there exists a morphism $g : \Nc
\rightarrow \M'$ which makes the following diagram commutative:

\begin{equation}\label{sec3.2:diag1}
\raisebox{-1cm}{
\begin{tikzpicture}[normal line/.style={->}, font=\scriptsize]
\matrix(m) [ampersand replacement=\&, matrix of math nodes, row sep = 3em, column sep =
2.5em]
{ \M \& \& \M' \\
     \& \Nc \& \\};
\path[normal line]
(m-1-1) edge node[auto] {$f$} (m-1-3)
(m-1-1) edge node[auto] {$f'$} (m-2-2)
(m-2-2) edge node[auto] {$g$} (m-1-3);
\end{tikzpicture}
}
\end{equation}
\end{definition}

The morphism $g$ in the definition is unique and is in fact a
quasi-isomorphism. Indeed, by the commutativity of the diagram, the
image of $g$ contains the image of $f$. Thus, the cokernel of $g$ is
finite. Moreover, since $f$ is injective, 
$g$ is injective on $\im f'$, which is cofinite in $\Nc$. It follows
that $\ker g$ is finite and $g$ is a
quasi-isomorphism.  Moreover, for every $x \in \Nc$, there exists a
positive integer $n$ such that $\pi^n x$ is in the image of $f'$. The
image of $\pi^n x$ by $g$ is then uniquely defined by the commutativity of
the diagram (\ref{sec3.2:diag1}). The uniqueness of $g$ follows.

A maximal module for $\M$, if it
exists, is unique up to isomorphism. Indeed, if $\M'$ and $\M''$ are
two maximal modules for $\M$ then there exist two quasi-isomorphisms
$g_1 : \M' \rightarrow \M''$ and $g_2 : \M'' \rightarrow \M'$ and the
uniqueness of $g$ in the diagram (\ref{sec3.2:diag1}) implies that
$g_1 \circ g_2 = \mathrm{Id}_{\M''}$ and $g_2 \circ g_1 =
\mathrm{Id}_{\M'}$.  If it exists, we denote the maximal module of $\M$ by
$\Max(\M)$. We can rephrase the above by saying that if $\M'$ is the
maximal module for $\M$ then there is a quasi-isomorphism from
$\M$ into $\M'$ and any quasi-isomorphism $\M' \rightarrow
\M''$ is an isomorphism. In fact, this condition characterises maximal
modules:

\begin{lemma}\label{sec3.1:lemmax1}
  Let $\M$ be a finitely generated  torsion free $\mSnu$-module. Let
  $\M'$ be a $\mSnu$-module such that there is a quasi-isomorphism $f:\M
  \rightarrow \M'$. The following assertions are equivalent:
  \begin{enumerate}
    \item $\M'$ is maximal;
    \item any quasi-isomorphism $\M' \rightarrow \M''$ is an
      isomorphism.
  \end{enumerate}
\end{lemma}

\begin{proof}
We only have to prove that the second property implies that $\M'$ 
verifies the universal property of maximal modules.  For this let $\Nc$ 
be a finite type $\mSnu$-module such that there is a quasi-isomorphism 
$f': \M \rightarrow \Nc$. Let $\Delta=f \oplus f': \M \rightarrow \M' 
\oplus \Nc$ be the diagonal embedding and let $\M_0=\frac{\M' \oplus 
\Nc}{\Delta(\M)}$. It is clear that $\M_0$ is a finitely generated 
torsion free $\mSnu$-module.

There are canonical injections $i_{\M'}:\M' \rightarrow \M_0$ and
$i_{\Nc}:\Nc \rightarrow \M_0$. We claim that $i_{\M'}$ and $i_{\Nc}$
are quasi-isomorphisms. To see that, it suffices to show that the
induced injection $i_{\M}=(i_{\M'},i_{\Nc})\circ \Delta:\M \rightarrow \M_0$ has a finite cokernel.
But $$\coker i_{\M}=\frac{\coker f \oplus \coker f'}{\Delta(\M) \cap
(\coker f \oplus \coker f')}$$ 
which has finite length being a quotient of $\coker f \oplus \coker
f'$.

Next, by hypothesis $i_{\M'}$ is in fact an isomorphism so that we
have a quasi-isomorphism $g=i_{\M'}^{-1} \circ i_{\Nc}$ which sits in
the following diagram:
\begin{equation}\label{sec3.1:diagqua2}
\raisebox{-1cm}{
\begin{tikzpicture}[normal line/.style={->}, font=\scriptsize]
\matrix(m) [ampersand replacement=\&, matrix of math nodes, row sep =
  3em, column sep =
7em]
{  \M' \& \frac{\M' \oplus \Nc}{\Delta(\M)} \\
    \M \& \Nc \\};
\path[normal line]
(m-2-2) edge node[auto] {$i_{\Nc}$} (m-1-2)
(m-1-1) edge node[auto] {$i_{\M'}$} (m-1-2)
(m-2-1) edge node[auto] {$f$} (m-1-1)
(m-2-1) edge node[auto] {$f'$} (m-2-2)
(m-2-2) edge node[auto] {$g$} (m-1-1);
\end{tikzpicture}
}
\end{equation}
It is clear that the lower left triangle of the diagram is commutative
and we are done.
\end{proof}

A theorem of Iwasawa \cite{MR0124316} asserts that if $\M$ is a finitely 
generated module over $\mS_0$, then $\Max(\M)$ exists and is free of 
finite rank over $\mS_0$. The main object of \S \ref{subsec:iwasawa} is 
to extend this result to modules over $\mSnu$: we shall provide a 
\emph{constructive} proof of the existence of $\Max(\M)$ for any 
finitely generated torsion-free module $\M$ over $\mSnu$. We will see 
however that this $\Max(\M)$ is not free in general; nevertheless we 
shall provide an upper bound on the number of generators of $\Max(\M)$.

\begin{lemma}\label{sec3.2:lemma1}
Let $f : \M \rightarrow \M'$ be a quasi-isomorphism between
torsion-free finitely generated $\mSnu$-modules. Suppose that
$\M'$ is free then $\M'$ is maximal.
\end{lemma}
\begin{proof}
We use the criterion of Lemma \ref{sec3.1:lemmax1}. Let $\Nc$ be a
finitely generated $\mSnu$-module such that there is a
quasi-isomorphism $f' : \M' \rightarrow \Nc$ and we want to show that
$f'$ is an isomorphism.
As $\M'$ is torsion-free, we know that $f'$ is injective. Now, suppose
that there exists a non zero element in the cokernel of $f'$. It means that 
there exists a non zero $x \in \Nc$ which is not in the image of
$f'$. As $f'$ is a quasi-isomorphism there exists $n
\in \N$ and $\lambda \in \mSnu$ a distinguished element with $\pi^n\cdot x
\in \im f'$ and $\lambda\cdot x \in \im f'$. If we set $z_1 =
{f'}^{-1}(\pi^n\cdot x)$
and $z_2 ={f'}^{-1}(\lambda\cdot x)$, we have the relation
\begin{equation}\label{eq:sec3.2:lemma1}
\lambda z_1 - \pi^n
z_2=0,
\end{equation}
in $\M'$. Let $(e_i)_{i \in I}$ be a basis of $\M'$ and write $z_i =
\sum \mu_i^j e_j$ for $i=1,2$.
Putting this in (\ref{eq:sec3.2:lemma1}), we obtain that $\lambda
\mu_1^j= \pi^n \mu_2^j$ and thus $\pi^n | \mu_1^j$ for $j \in I$ since $\lambda$ is a
distinguished element of $\mSnu$. But then $f'(\sum \mu_1^j/\pi^n
e_j)=1/\pi^n.
f(z_1)=x$ contradicting the fact that $x$ is not in the image of $f'$.
\end{proof}
\begin{remark}
One can rephrase Iwasawa's result in a more abstract way using the 
category language. Let $\uMod_{\mSnu}$ be the category of finitely generated 
$\mSnu$-modules, that are torsion-free and let $\uModtf_{\mSnu}$ (resp. 
$\uFree_{\mSnu}$) denote its full subcategory gathering all torsion-free 
modules (resp. all free modules). We also introduce the category 
$\uMod^\qis_{\mSnu}$, which is by definition the category of finitely 
generated $\mSnu$-modules up to quasi-isomorphism, \emph{i.e.} 
$\uMod^\qis_{\mSnu}$ is obtained from $\uMod_{\mSnu}$ by inverting formally 
quasi-isomorphisms. We have a natural functor $\uMod_{\mSnu} \to 
\uMod^\qis_{\mSnu}$, whose restriction to $\uModtf_{\mSnu}$ defines a pylonet in 
the sense of \cite{MR2951749}, \S 1. It follows from the results of 
\emph{loc. cit} (see Corollary 1.2.2) that the $\Max$ construction is a 
functor: to a morphism $f : \M \to \M'$ in $\uModtf_{\mSnu}$, one can attach a 
morphism $\Max(f) : \Max(\M) \to \Max(\M')$. We recall briefly the 
construction of $\Max(f)$. Let $\M''$ be the pushout $\M' \oplus_\M 
\Max(\M)$, that is the direct sum $\M' \oplus \Max(\M)$ divided by $\M$ 
(embedded diagonally). We have a natural morphism $\M' \to \M''$ which 
turns out to be a quasi-isomorphism. Hence, there exists a map $\M'' \to 
\Max(\M')$ and we finally define $\Max(\M)$ to be the compositum $\Max(\M) 
\to \M'' \to \Max(\M')$ where the first map comes from the natural 
embedding $\Max(\M) \to \M' \oplus \Max(\M)$.
\end{remark}

If $\M$ is a submodule of $\mSnu^d$ (for some positive integer $d$), the 
following proposition gives a very explicit description of $\Max(\M)$.

\begin{proposition}
\label{prop:qis}
Write $\nu=\beta/\alpha$, with $\alpha,\beta$ relatively prime
integers.
Let $d$ be a positive integer and $\M$ be a submodule of $\mSnu^d$. 
Then $\Max(\M)$ exists and
$$\Max(\M) = \big\{ \,x \in \mSnu^d \quad | \quad \exists n \in \N, \,\,
\pi^n x \in \M \text{ and } (u^\alpha/\pi^\beta)^n\cdot x \in \M \, \big\}.$$
Furthermore the morphism $i_\M : \M \to \Max(\M)$ is the natural
embedding.
\end{proposition}

\begin{proof}
Let $\M_\free$ be the set of $x \in \mSnu^d$ such that there exists some 
$n$ such that $\pi^n x$ and $(u^\alpha/\pi^\beta)^n\cdot x$ belong to $\M$. 
We want to show that $\Max(\M)$ exists and is equal to $\M_\free$. It is 
clear that $\M \subset \M_\free$ and that the quotient $\M_\free/\M$ is 
canceled by a power of $\pi$ and a power of $u^\alpha/\pi^\beta$ which 
is a distinguished element. Hence it has finite length, and the 
inclusion $\M \to \M_\free$ is a quasi-isomorphism. Next, suppose that 
we are given a $\mSnu$-module $\M_0$ together with a quasi-isomorphism 
$g:\M_\free \rightarrow \M_0$. Then there is a quasi-isomorphism 
$i_\M:\M \rightarrow \M_0$ that sits in the following diagram:

\begin{equation}
\raisebox{-1cm}{
\begin{tikzpicture}[normal line/.style={->},
  font=\scriptsize]
  \matrix(m) [ampersand replacement=\&, matrix of math nodes, row sep = 3em, column sep =
2.5em]
{ \M \& \M_\free \& \mSnu^d \\
     \& \M_0 \& \\};
\path[normal line]
(m-1-1) edge node[auto] {} (m-1-2)
(m-1-1) edge node[auto] {$i_\M$} (m-2-2)
(m-1-2) edge node[auto] {$g$} (m-2-2)
(m-1-2) edge node[auto] {} (m-1-3);
\end{tikzpicture}
}
\end{equation}
Note that $g$ is injective as it is a quasi-isomorphism.
Moreover,
we know that the cokernel of $\iota_\M$ is annihilated by a power of 
$u^\alpha/\pi^\beta$ and a power of $\pi$, which implies that $g$ is
surjective. Thus, $g$ is an isomorphism and by Lemma
\ref{sec3.1:lemmax1},  $\Max(\M)$ exists and $\Max(\M) = \M_\free$ as claimed. The second part of the 
proposition is clear from the above diagram.
\end{proof}

It follows directly from Proposition \ref{prop:qis} that the 
intersection of two maximal modules is maximal. The same is however not 
true for the sum: in general the $\mSnu$-module $\M + \M'$ is not 
maximal even if $\M$ and $\M'$ are (take for example $\M = u \mS_0$ and 
$\M' = \pi \mS_0$). This leads us to define the new operation $+_\free$ 
(which is much more pleasant than the usual sum of modules) on the set 
of maximal submodules of $\mSnu^d$ as follows: $$\M +_\free \M' = 
\Max(\M + \M').$$

We also deduce from Proposition \ref{prop:qis} that a
$S_0$-module $\M$ is free if and only if $\M = \M_\free$. This gives a nice
criterion to check if a $\mS_0$-module is free. It is not true
in general for a sub-$\mSnu$-module $\M$ of $\mSnu^d$ that $\Max(\M)$
is free (this will become apparent when we give the general shape of a
maximal $\mSnu$-module in \S \ref{subsec:iwasawa}).
However, by Lemma \ref{sec1:lempiso}, 
every $\mSnu$ becomes isomorphic to $\mS_0$
over a finite extension $\Ri'=\Ri[\Rpi]$ (where $\Rpi$ depends on $\nu$). 
Set $\eSnu = \mSnu \otimes_{\Ri} \Ri'$. For all submodule $\M$ of
$\mSnu^d$, we obtain that $\Max(\M \otimes \eSnu)$ is a free submodule of
$({\eSnu})^d$. Denote by $\Max^d_{\mSnu}$ the set of maximal
sub-$\mSnu$-modules of $\mSnu^d$ and by $\Free^d_{\eSnu}$ the set of
free sub-$\eSnu$-module of $({\eSnu})^d$. 

\begin{proposition}\label{prop:tech}
The natural map
$$\begin{array}{rcl}
  \Phi \quad : \Max_{\mSnu}^d & \longrightarrow &
\Free_{\eSnu}^d \\
\M & \mapsto & \Max(\M \otimes_{\mSnu} \eSnu)
\end{array}$$
is injective. A left inverse of $\Phi$ is given by $\M' \mapsto \M'
\cap \mSnu^d$.  Moreover, the image of $\Phi$ contains the subset of
$\Free^d_{\eSnu}$ of free modules which admit a basis $(e_i)_{i \in
I}$ where $e_i \in (\eSnu)^d$ and $e_i =\Rpi^{\alpha_i} e'_i$ with
$e'_i \in  (\mSnu)^d$ and $\alpha_i \in \N$.
\end{proposition}
\begin{remark}
 Actually, we will prove later (see Lemma \ref{sec3.2:lemmamax}) that the image of
  $\Phi$ is exactly the subset of $\Free^d_{\eSnu}$ verifying the
  condition of Proposition \ref{prop:tech}. 
\end{remark}
\begin{proof}
  In order to prove that $\Phi$ is injective, it is enough to prove
  that $\Phi$ has a left inverse. For this, let $\M \in
  \Max^d_{\mSnu}$ and let $\M' = \Max(\M \otimes_{\mSnu} {\eSnu}) \in
  \Free_{{\eSnu}}^d$. Then it suffices to prove that $\M_2=\M' \cap
  \mSnu^d$ is a maximal sub-$\mSnu$-module of $\mSnu^d$. Indeed, 
  as it is clear that $\M_2$ contains $\M$ and that the injection $\M
  \rightarrow \M_2$ is a quasi-isomorphism since the injection $\M
  \rightarrow \M'$ is a quasi-isomorphism,
  we remark that by the maximality of $\M$ it would
  imply that
  $\M = \M_2$.

  For this let $x \in \mSnu^d$ and suppose that there exists $n \in
  \N$ such that $\pi^n\cdot x \in \M_2$ and
  $(u^\alpha/\pi^\beta)^n\cdot x \in
  \M_2$. As $\M'$ is maximal and $\M_2 \subset \M'$, by Proposition
  \ref{prop:qis}, it means that $x \in \M'$. hence $x
  \in \M_2$. Using again Proposition \ref{prop:qis}, we deduce that
  $\M_2$ is maximal.

Let us now prove the last claim of the proposition. Let $\M' \in
\Free^d_{\eSnu}$ which admits a basis $(e_i)_{i \in I}$ where $e_i \in
(\eSnu)^d$ and $e_i =\Rpi^{\alpha_i} e'_i$ with $e'_i \in (\mSnu)^d$
and $\alpha_i \in \N$. We have to find a sub-$\mSnu$-module $\M$ of
$\mSnu^d$ such that $\M \otimes_{\mSnu} \eSnu$ is quasi-isomorphic to
$\M'$.  As $\M'=\bigoplus e_i \eSnu$, it is enough to treat the case
$d=1$. Let $0 \leq \alpha_1$ be an integer and let $\M'$ be the
sub-$\eSnu$-module of $\eSnu$ generated by $\Rpi^{\alpha_1}$. Let
$\lambda$ be a positive integer such that
$\frac{\alpha_1}{\alpha}+\lambda \frac{\beta}{\alpha}=\gamma \in \Z$.
Such a $\lambda$ exists because $\alpha$ and $\beta$ are relatively
prime. Let $\M$ be the sub-$\mSnu$-module of $\mSnu$ generated by
$\pi$ and $\frac{u^\lambda}{\pi^\gamma}$.  Let $\mu =
\Rpi^{-\alpha_1}\frac{u^\lambda}{\pi^\gamma}$, it is clear that
$v_\nu(\mu)=0$ so that $\mu$ is a distinguished element of $\eSnu$.
Thus, we have $\Rpi^{\alpha_1}.\mu \in \M  \otimes_{\mSnu} \eSnu$ and
$\Rpi^{\alpha_1}\cdot \Rpi^{\alpha-\alpha_1} \in \M  \otimes_{\mSnu}
\eSnu$ therefore $\M \otimes_{\mSnu} \eSnu$ is quasi-isomorphic to $\M'$.
\end{proof}

\subsection{An approach based on localisation}
We have seen that in a class of quasi-isomorphism of a finite type
torsion-free $\mSnu$-module $\M$ there exists a distinguished element $\Max(\M)$.
In  this section, we use this fact in order to represent the
quasi-isomorphism class of $\M$ by localizing with respect to
$u^\alpha/\pi^\beta$ and
$\pi$. We thus obtain a representation of finite type torsion-free
$\mSnu$-modules amenable to computations.

\subsubsection{A useful bijection}\label{sec:useful}

We keep our fixed positive integer $d$.  We recall that $$\mE = \big\{
\sum_{i \in \Z} a_i u^i, \,a_i \in K, \, v_K(a_i)+\nu i \,\, \text{bounded
below} \, \text{and} \lim_{i \rightarrow -\infty} v_K(a_i)+\nu i
=+\infty \big\}$$ is a field containing $\mSpnu$ and $\mSunu$.  If $\M$
is a sub-$\mSnu$-module of $\mE^d$, we shall denote by $\M_\pi$ (resp.
$\M_u$) the sub-$\mSpnu$-module (resp. the sub-$\mSunu$-module) of
$\mE^d$ generated by $\M$.  For example, if $\M$ is free over $\mSnu$
with basis $(e_1, \ldots, e_h)$, then $\M_\pi$ (resp. $\M_u$) is also
free over $\mSpnu$ (resp. $\mSunu$) with the same basis.  As $\M$ is
torsion free, and as $\mSunu$ and $\mSpnu$ are principal ideal
domains, $\M_\pi$ and $\M_u$ are free.  We denote by $\Max^d_{\mSnu}$
the set of maximal sub-$\mSnu$-modules of $\mSnu^d$ and for $A =
\mSnu,\mSpnu \text{ or } \mSunu$, let $\Free^d_A$ denote the set of
sub-$A$-modules of $A^d$, which are free over $A$. Recall that
$\Max^d_{\mS_0}=\Free^d_{\mS_0}$. Thus, the following lemma provides a
useful description of maximal $\mS_0$-modules.

\begin{lemma}
\label{lemma:bijection}
Let $\mS =\mS_0$.
The natural map
$$\begin{array}{rcl}
  \Psi' \quad : \Free_{\mS}^d & \longrightarrow &
\Free_{\mSp}^d \times \Free_{\mSu}^d \\
\M & \mapsto & (\M_\pi,\M_u).
\end{array}$$
is injective. If a pair $(A,B)$ is in the image of $\Psi'$, its unique preimage under $\Psi'$ is given by 
$A \cap B$.
\end{lemma}
\begin{proof}
From the descriptions of elements of $\mS$, $\mSp$, $\mSu$ and $\mE$ in 
terms of series, it follows that $\mS = \mSp \cap \mSu$. If $\M 
\in \Free^d_{\mS}$, it is isomorphic to $\mS^{h}$ for $h \leq d$
and, by applying the preceding remark component by component, we get $\M = 
\M_\pi \cap \M_u$. This  implies the injectivity of $\Psi'$ and the 
given formula for its left-inverse.
\end{proof}

Using Lemma \ref{lemma:bijection}, we can prove:
\begin{theorem}
\label{prop:bijection}
The natural map
$$\begin{array}{rcl}
  \Psi \quad : \Max_{\mSnu}^d & \longrightarrow &
\Free_{\mSpnu}^d \times \Free_{\mSunu}^d \\
\M & \mapsto & (\M_\pi,\M_u).
\end{array}$$
is injective and its image consists of pairs $(A,B)$ such that $A$ and 
$B$ generate the same $\mE$-vector space in $\mE^d$. If a pair $(A,B)$ 
satisfies this condition, its unique preimage under $\Psi$ is given by 
$A \cap B$.

Furthermore, we have the following equalities:
\begin{eqnarray*}
\Psi(\M \cap \M') & = & (\M_\pi \cap \M'_\pi, \M_u \cap \M'_u) \\
\Psi(\M +_\free \M') & = & (\M_\pi + \M'_\pi, \M_u + \M'_u)
\end{eqnarray*}
for all $\M,\M' \in \Max_{\mSnu}^d$.
\end{theorem}

\begin{proof} 
  
  Let $\Rpi$ in an algebraic closure of $K$, be such that
  $\Rpi^\alpha=\pi$.  Let $\Ri'=\Ri[\Rpi]$ and ${\eSnu}=\mSnu
  \otimes_{\Ri} \Ri'$. We know by Lemma \ref{sec1:lempiso} that $\eSnu$ is
  isomorphic to $\Ri'[[u]]$.  Then, the map $\Psi$
  sits in the following commutative diagram:
\begin{equation}\label{sec3.2:diagprobij1}
  \raisebox{-1cm}{
    \begin{tikzpicture}[normal line/.style={->}, font=\scriptsize]
      \matrix(m) [ampersand replacement=\&, matrix of math nodes, row
   sep = 3em, column sep = 4.5em] 
   { \Max_{\mSnu}^d \&  \Free_{\mSpnu}^d \times \Free_{\mSunu}^d \\ 
      \Free_{\eSnu} \& \Free_{\eSpnu}^d \times \Free_{\eSunu}^d \\ }; 
   \path[normal line]
      (m-1-1) edge node[auto] {$\Psi$}  (m-1-2) 
      (m-1-1) edge node[auto] {$\Max(.\otimes_{\mSnu} \eSnu)$} (m-2-1) 
      (m-1-2) edge node[auto] {$.\otimes_{\mSnu} \eSnu$} (m-2-2) 
      (m-2-1) edge node[auto] {$\Psi'$} (m-2-2);
  \end{tikzpicture}
}
\end{equation} 
By Proposition \ref{prop:tech}, the map $\M \mapsto \Max(\M
\otimes_{\mSnu} \eSnu)$ is injective and $\Psi'$ is injective by 
Lemma \ref{lemma:bijection}, from which we deduce that $\Psi$ is
injective by the commutativity of (\ref{sec3.2:diagprobij1}).

We want to prove now that if the pair $(A,B)$ belongs to $\Free_{\mSpnu}^d 
\times \Free_{\mSunu}^d$ and satisfies the condition of the theorem, 
then $\M = A \cap B$ is maximal over $\mSnu$ and $\Psi(\M) = (A,B)$. We claim 
that there exists a basis $(e_1, \ldots, e_h)$ of $A$ (over $\mSpnu$) such 
that $\M$ is included inside the $\mSnu$-module generated by the $e_i$'s. 
Indeed, let us first consider $(e_1, \ldots, e_h)$ a basis of $A$ and 
denote by $\M'$ the $\mSnu$-module generated by the $e_i$'s. Now, remark 
that, by our assumption on the pair $(A,B)$, every element $x \in B$ can 
be written as a $\mE$-linear combination of the $e_i$'s. Taking for $n$ 
the smallest valuation of the coefficients appearing in this writing, we 
get $x \in \pi^{-n} \M'_u$. Moreover, since $B$ is finitely generated over 
$\mSunu$, we can choose a uniform $n$. Replacing $e_i$ by $\pi^{-n} e'_i$ 
for all $i$, we then get $A = \M'_\pi$ and $B \subset \M'_u$. Thus $\M = A 
\cap B \subset \M'_\pi \cap \M'_u = \M'$.

Since $\mSnu$ is a noetherian ring (recall that $\nu$ is rational), we
find that $\M$ is finitely generated over $\mSnu$. Furthermore, one can
compute $\Max(\M)$ using Proposition \ref{prop:qis}: if $x$ is an
element of $\mSnu^d$ for which there exists $n$ such that $\pi^n x$
and $(u^\alpha/\pi^\beta)^n x$ belong to $\M$, then $x \in A$ (since $\pi$ is invertible
in $\mSpnu$) and $x \in B$ (since $u^\alpha/\pi^\beta$ is invertible in $\mSunu$). Thus
$x \in \M$ and $\Max(\M) = \M$, \emph{i.e.} $\M$ is maximal.

Let us prove now that $\Psi(\M) = (A,B)$. By the same argument as before, 
we find that there exists a positive integer $n$ such that $\pi^n \M' 
\subset \M \subset \M'$, from what it follows that $\M_\pi = \M'_\pi = A$. The 
method to prove that $\M_u = B$ is analogous: we first show that there 
exists a basis $(e_1, \ldots, e_h)$ of $B$ over $\mSunu$ and some elements 
$s_1, \ldots, s_h \in \mSnu$ such that:
\begin{itemize}
\item all $s_i$'s are invertible in $\mSunu$, and
\item we have $\sum s_i e_i \mSnu \subset \M \subset \sum e_i \mSnu$.
\end{itemize}
From these conditions, it follows that $\M_u$ is generated by the $e_i$'s 
over $\mSu$ and, consequently, that $\M_u = B$.

It remains to prove the claimed formulas concerning intersections and 
sums. For the intersection, we note that if $\M \cap \M' = (\M_\pi \cap 
\M_u) \cap (\M'_\pi \cap \M'_u) = (\M_\pi \cap \M'_\pi) \cap (\M_u \cap 
\M'_u)$. Hence, we just need to justify that $\M_\pi \cap \M'_\pi$ and $\M_u 
\cap \M'_u$ are free over $\mSpnu$ and $\mSunu$ respectively, and that they 
generate the same $\mE$-vector space. The freedom follows from the 
classification theorem of finitely generated modules over principal 
rings, whereas the second property is a consequence of the flatness of 
$\mE$ over $\mSpnu$ and $\mSunu$.

For the sum, we have to justify that $(\M +_\free \M')_\pi = \M_\pi + 
\M'_\pi$ and $(\M +_\free \M')_u = \M_u + \M'_u$. It is clear that $(\M 
+ \M')_\pi = \M_\pi + \M'_\pi$ and $(\M + \M')_u = \M_u + \M'_u$. Hence, 
it is enough to prove that, given a finitely generated $\mSnu$-module $N 
\in \mSnu^d$, we have $\Max(N)_\pi = N_\pi$ and $\Max(N)_u = N_u$. It is 
obvious by Proposition \ref{prop:qis}.
\end{proof}

\paragraph{Reinterpretation in the language of categories}

We introduce the ``fiber product'' category $\uFree_{\mSpnu} \otimes_ 
{\uFree_\mE} \uFree_{\mSunu}$ whose objects are triples $(A,B,f)$ where $A \in 
\uFree_{\mSpnu}$, $B \in \uFree_{\mSunu}$ and $f : \mE \otimes_{\mSpnu} A \to \mE 
\otimes_{\mSunu} B$ is an $\mE$-linear isomorphism. We have natural functors 
in both directions between $\uMax_{\mSnu}$ and $\uFree_{\mSpnu} \otimes_ {\uFree_\mE} 
\uFree_{\mSunu}$: to an object $\M$ of $\Max^d_{\mSnu}$, we associate the triple 
$(\mSpnu \otimes_S \M, \mSunu \otimes_S \M, f)$ where $f$ is the canonical 
isomorphism, and conversely, to a triple $(\M_\pi, \M_u, f)$, we associate 
the fiber product of the following diagram (which turns out to be
free of finite rank over $\mSnu$):
\begin{equation}
\raisebox{-1cm}{
\begin{tikzpicture}[normal line/.style={->}, font=\scriptsize]
\matrix(m) [ampersand replacement=\&, matrix of math nodes, row sep = 3em, column sep =
2.5em]
{ \& \& \M_u  \\
  \M_\pi   \& \mE \otimes_{\mSpnu} \M_\pi \& \mE \otimes_{\mSunu} \M_u \\};
\path[normal line]
(m-1-3) edge node[auto] {} (m-2-3)
(m-2-1) edge node[auto] {} (m-2-2)
(m-2-2) edge node[auto] {} (m-2-3);
\end{tikzpicture}
}
\end{equation}

Theorem \ref{prop:bijection} then says that these two functors are 
equivalences of categories inverse one to the other. Actually, this
result can be generalized to non-free modules as follows.

\begin{proposition}
  The functor $\uMod_{\mSnu} \to \uMod_{\mSpnu} \otimes_ {\uMod_\mE}
\uMod_{\mSunu}$, $\M \mapsto (\mSpnu \otimes_S \M, \mSunu \otimes_S \M)$ 
factors through $\uMod^\qis_{\mSnu}$ and the resulting functor
$$\uMod^\qis_{\mSnu} \to \uMod_{\mSpnu} \otimes_ {\uMod_\mE}
\uMod_{\mSunu}$$
is an equivalence of categories.
\end{proposition}

\begin{proof}
Left to the reader.
\end{proof}

\subsubsection{Normal forms for modules over $\mSpnu$ and
$\mSunu$}\label{subsec3.1}

As $\mSpnu$ and $\mSunu$ are Euclidean rings there exists a good notion 
of rank as well as Hermite Normal Forms for matrix over these rings. In 
this section, we state propositions giving the shape of Hermite Normal 
Form together with algorithms with oracles to compute them. We recall 
that an algorithm with oracle is a Turing machine which has access to 
oracles to store elements of the base ring and perform all usual ring 
operations: test equality, computation of the valuation, addition, 
opposite, multiplication and Euclidean division.  We will measure the 
time complexity of the algorithms by counting the number of calls to the 
oracles. 
Classically, we then derive some consequences which will be used in this 
paper. For the complexity analysis, we denote by $\theta$ a real number 
such that product of two $d\times d$ matrices with coefficient in 
$\mSnu$ can be done in $O(d^\theta)$ ring operations. With a naive 
algorithm, we can take $\theta=3$ and with the current best known 
algorithm of Coppersmith and Winograd \cite{MR1056627}, $\theta=2.376$.

\begin{proposition}\label{sec3.1:prop1}
  Let $M=(m_{ij})\in M_{d\times d'}(\mSpnu)$,
let $r$ be the rank of $M$. Then, there exists an invertible matrix 
$P$ such that $M.P=T$ with 
\begin{equation}\label{eqformmat1}
T = 
\raisebox{-1.2cm}{
\begin{tikzpicture}[ms/.style={densely dotted}]
\matrix [ampersand replacement=\&, left delimiter=(, right delimiter=)]
{
\node (T1) {$t_{1}$}; \& \& \&  \& \node (z1) {$0$}; \& \node ()
{}; \&
\node (z2) {$0$}; \\
\node (star1) {$\star$}; \& \& \&  \& \& \& \\
 \& \& \& \& \& \& \\
 \& \& \& \& \& \& \\
\& \& \& \node (Td) {$t_{r}$}; \& \& \& \\
\& \& \& \node (star2) {$\star$}; \& \& \& \\
\node () {}; \& \& \& \& \& \& \\
\node (star3) {$\star$}; \& \& \& \node (star4) {$\star$}; \& \node
(z3) {$0$}; \& \& \node (z4) {$0$}; \\
};
\draw[ms] (T1) -- (Td);
\draw[ms] (star1) -- (star2);
\draw[ms] (star1) -- (star3);
\draw[ms] (star2) -- (star4);
\draw[ms] (star3) -- (star4);
\draw[ms] (z1) -- (z2);
\draw[ms] (z1) -- (z3);
\draw[ms] (z3) -- (z4);
\draw[ms] (z2) -- (z4);
\end{tikzpicture}
},
\end{equation}
where 
\begin{itemize}
\item for $i=1, \ldots , r$, $t_{i} = u^{d_j}+ \sum_{i=0}^{d_j-1} b_j
u^j$ with $v_K(b_j)+\nu
(j-d_j) > 0$ ;
\item for $i=1, \ldots, r$, $T_{l(i),i}=t_i$ and $l$ is a scrictly increasing function from $\{1, \ldots r\}$ to
$\{ 1, \ldots , d \}$ such that $l(1)=1$.
\end{itemize}
The matrix $T$ is said to be an echelon form of $M$. 
Let $d_{\max}$ be the maximal Weierstrass degree of the
entries of $M$, an echelon form of $M$ can be computed in
$O(d\cdot d'\cdot d_{\max} + \max(d^\theta\cdot d', d'^{\theta}\cdot d)
\log(2d'/d))$  ring operations

If the echelon form moreover
satisfies:
\begin{itemize}
\item all entries on the $l(i)^{th}$-row are elements of
$K[u]$ of degree
$< d_i$.
\end{itemize}
then $T$ is unique with these properties and is called the Hermite
Normal Form. The Hermite Normal form of $M$ can be computed from an
echelon form of $M$ at the expense of an additional $O(r^2)$ ring
operations.
\end{proposition}
\begin{proposition}\label{sec3.1:prop2} Let $M\in M_{d\times
  d'}(\mSunu)$, let $r$ be the rank of $M$. Then there exists an
invertible matrix $P$ such that $M.P=T$ and
\begin{equation}\label{eqformmat2} T = \raisebox{-1.2cm}{
\begin{tikzpicture}[ms/.style={densely dotted}] \matrix [ampersand
replacement=\&, left delimiter=(, right delimiter=)] { \node (T1)
{$\pi^{d_1}$}; \& \& \&  \& \node (z1) {$0$}; \& \node () {}; \& \node
(z2) {$0$}; \\ \node (star1) {$\star$}; \& \& \&  \& \& \& \\ \& \& \&
\& \& \& \\ \& \& \& \& \& \& \\ \& \& \& \node (Td) {$\pi^{d_r}$}; \&
\& \& \\ \& \& \& \node (star2) {$\star$}; \& \& \& \\ \node () {}; \&
\& \& \& \& \& \\ \node (star3) {$\star$}; \& \& \& \node (star4)
{$\star$}; \& \node (z3) {$0$}; \& \& \node (z4) {$0$}; \\ };
\draw[ms] (T1) -- (Td); \draw[ms] (star1) -- (star2); \draw[ms]
(star1) -- (star3); \draw[ms] (star2) -- (star4); \draw[ms] (star3) --
(star4); \draw[ms] (z1) -- (z2); \draw[ms] (z1) -- (z3); \draw[ms]
(z3) -- (z4); \draw[ms] (z2) -- (z4); \end{tikzpicture} },
\end{equation} where \begin{itemize} \item for $i=1, \ldots, r$,
$T_{l(i),i}=\pi^{d_i}$ where $l$ is a strictly increasing function
from $\{1, \ldots r\}$ to $\{ 1, \ldots , d \}$ such that $l(1)=1$.
\end{itemize} The matrix $T$ is said to be an echelon form of $M$.
An echelon form of $M$
can be computed in $O(d.d')+\max
(d^\theta\cdot d',d'^\theta\cdot d)\log(2d'/d))$ ring operations.  

If the echelon form moreover satisfies
\begin{itemize}
\item the
entries on the $l(i)^{th}$-row are representatives modulo
$\pi^{d_i}$.  
\end{itemize} 
then $T$ is unique with these
properties and the called the Hermite Normal Form of $M$. The Hermite
Normal form of $M$ can be computed at the expense of an additional
$O(r^2)$ ring operations.
\end{proposition} 

\begin{proof}
The proof of the previous propositions as well as algorithms to
compute the echelon form of $M$ with the given complexity is an
immediate consequence of \cite[Theoreme 3.1]{MR1135749} together with
the fact that $\mSpnu$ and $\mSunu$ are Euclidean rings.  Moreover for
all  $x,y \in \mSpnu$ one can compute the $\gcd(x,y)$ in
$O(\deg_W(y))$ ring operations. From its triangle form, one can
then compute the Hermite Form of $M$ with coefficients in $\mSpnu$ at
the expense of $O(d\cdot r\cdot d_{\max})$ ring operations.
\end{proof}
\begin{remark}
  We deduce from this proposition that if $M\in M_{d\times
  d'}(\mSpnu)$
is a full rank matrix, there
exists $P$ such that $M\cdot P$ is a matrix of the form
(\ref{eqformmat1}) with all coefficients in $K[u]$.
In the same way, if $M\in M_{d \times d'}(\mSunu)$ is a full rank matrix
then there exists an invertible matrix $P$
such that $M\cdot P$ has the form (\ref{eqformmat2}) with all entries
defined modulo $\pi^{\max\{d_1, \ldots , d_r\}}$.
\end{remark}

Let $\mSloc$ be $\mSunu$ or $\mSpnu$. We derive some consequences of
the existence of triangle forms and Hermite Normal Form for the
representation and computation with finitely generated sub-$\mSloc$-modules
of $\mSloc^d$.  We can represent a finitely generated
sub-$\mSloc$-module $\M$ of $\mSloc^d$ by a $d\times d$ matrix $M$
giving $d$ generators of $\M$ in the canonical basis of $\mSloc^d$
since every sub-module of $\mSloc^d$ has dimension at most $d$.
Keeping the same notations, one can compute the module of syzygies of
$\M$. For this it is enough to compute $R$, a matrix of maximal rank such that
$M\cdot R=0$ which can easily be done by computing an echelon form of $M$.
Given a vector $\mathscr{V} \in \mSloc^d$ provided by its coordinates
vector $V$ in the canonical basis, one can check efficiently if
$\mathscr{V} \in \M$ by finding a vector $X$ such that $M.X=V$ which
can also be done with the echelon form of $M$.

Let $M$ and $M'$ representing the modules $\M$ and $\M'$, one can
compute a matrix representing the module $\M+\M'$ by computing the
echelon form of the matrix $(M M')$ and taking the $d$ first columns. 
One can compute the intersection of $\M$ and $\M'$ in the same way by
finding $R$ and $R'$ such that $(MM')\left( \begin{array}{c} R \\ R'
\end{array}  \right)=0$.

\subsubsection{Consequences for algorithmics}

In view of the results of \S \ref{sec:useful} and \S
\ref{subsec3.1}, 
we shall represent a maximal $\mSnu$-module 
$\M$ living in some $\mSnu^d$ as a pair $(A,B)$ where $A$ (resp. $B$) is 
the matrix with coefficients in $\mSpnu$ (resp. in $\mSunu$) in Hermite 
Normal Form representing $\mSpnu \otimes_{\mSnu} \M$ (resp. $\mSunu
\otimes_{\mSnu} 
\M$). 

The second part of Theorem \ref{prop:bijection} tells us that it
is very easy to compute intersections and ``maximal-sums'' of
$\mSnu$-modules with this representation. Indeed, we just have to
perform the same operations on each component, and we have already
explained in \S \ref{subsec3.1} how to do it efficiently. As the
Hermite Normal Form is unique, it is also very easy to check the
equality of two maximal sub-$\mSnu$-modules of $\mSnu^d$. Using only
the echelon form of the matrices $A$ and $B$ it is also possible to
test membership.

Even better, this representation is also very convenient for many other 
operations we would like to perform on $\mSnu$-modules. Below we detail 
three of them.
First, let $\M \subset \mSnu^d$ be a maximal $\mSnu$-module. By definition, 
the \emph{saturation} of $\M$ in $\mSnu^d$ is the module
$$\M_\sat = \big\{ \,x \in \mSnu^d \quad | \quad \exists n \in \N, \,\,
\pi^n x \in \M \big\}.$$
It follows from Proposition \ref{prop:qis} that $\M_\sat$ is maximal over 
$\mSnu$, and we would like to compute it. For that, working with our
representation, we need to compute $(\M_\sat)_\pi$ and $(\M_\sat)_u$.
But, we have $(\M_\sat)_\pi = \M_\pi$ and 
$$(\M_\sat)_u = \big\{ \,x \in \mSunu^d \quad | \quad \exists n \in \N, 
\,\, \pi^n x \in \M_u \big\}.$$
The computation of $(\M_\sat)_\pi$ is then for free, whereas the computation 
of $(\M_\sat)_u$ can be achieved using Smith forms, which is here quite 
efficient due to the fact that $\mSunu$ is a discrete valuation ring.
An important special case is when $\M$ has rank $d$ over $\mSnu$. Then
$(\M_\sat)_u$ is always equal to $\mSunu^d$. Thus, in this case, if $\M$
is represented by the pair of matrices $(A,B)$, then $\M_\sat$ is just
represented by the pair $(A, I)$ where $I$ is the identity matrix.

More generally, one can consider the following situation. Let $\M \in 
\Max^d_{\mSnu}$ and $\M' \in \Max^d_{\mSpnu}$. We want to compute $\M \cap 
\M'$, which is a maximal module over $\mSnu$. As before, we need to determine 
$(\M \cap \M')_\pi$ and $(\M \cap \M')_u$ and one can check that:
\begin{eqnarray*}
(\M \cap \M')_\pi & = & \M_\pi \cap \M'_\pi \\
(\M \cap \M')_u & = & \M_u \cap \M'_u.
\end{eqnarray*}
Note that, here, $\M'_u$ is vector space over $\mE$. As before, the 
intersection $\M_u \cap \M'_u$ can be computed using Smith forms and,
if $\M'$ has rank $d$ over $\mSpnu$, we just have $\M'_u = \mE^d$ and so
$(\M \cap \M')_u = \M_u$.

The third example we would like to present is obtained from the previous 
one by inverting the roles of $\mSpnu$ and $\mSunu$: we take $\M \in 
\Free^d_{\mSnu}$ and $\M' \in \Free^d_{\mSunu}$ and we want to compute $\M 
\cap \M'$. We then have $(\M \cap \M')_\pi = \M_\pi \cap \M'_\pi$ and $(\M 
\cap \M')_u = \M_u \cap \M'$. Here a new difficulty occurs: $\M'_\pi$ is a 
$\mE$-vector space and so, in previous formulas, it appears an 
intersection between a free module over $\mSpnu$ and a $\mE$-vector space. 
Again, one can compute this Smith form. However, it is not so efficient 
as before since $\mSpnu$ is just an Euclidean ring, and not a discrete 
valuation ring. Anyway, it remains true that, in the case where $\M'$ has 
full rank, then $\M'_\pi = \mE^d$. So, in this case, $(\M \cap \M')_\pi$ is 
just equal to $\M_\pi$ and the computation of $(\M \cap \M')_\pi$ becomes 
very easy.

\subsubsection{Further localisations}

We remark that the matrix appearing in Proposition \ref{sec3.1:prop2} 
has coefficients in $\mSunu$ which is a discrete valuation ring while 
the matrix of Proposition \ref{sec3.1:prop1} has coefficients in 
$\mSpnu$ which is only Euclidean. For certain applications, it can be 
more convenient to compute with elements in a discrete valuation ring; 
for instance, the computation of the Smith Normal Form can be made 
faster in a discrete valuation ring.

It is actually possible to work only over discrete valuation rings by 
localising further. More precisely, for any element $a \in \bar K$ 
(where $\bar K$ is an algebraic closure $\bar K$ of $K$) with valuation 
$> \nu$, we have a canonical injective morphism $\mSpnu \rightarrow \bar 
K[[u-a]]$ which maps a series to its Taylor expansion at $a$. Hence, if 
$\M_p$ is a sub-$\mSpnu$-module of $\mSpnu^d$, one can consider 
$\M_{p,a} = \M_p \otimes_{\mSpnu} \bar K[[u-a]] \subset \bar K[[u-a]]^d$ 
for all element $a$ as before. Moreover, if $\M_p$ has maximal 
rank, all $\M_{p,a}$'s are trivial (\emph{i.e.} equal to $\bar 
K[[u-a]]$) expect a finite number of them (which are those for which $a$ 
is a root of one of the $t_i$'s of Proposition \ref{sec3.1:prop1}). In 
addition, the map:
$$\begin{array}{rcl}
  \Xi \quad : \Mod_{\mSpnu}^d & \longrightarrow &
  \prod_{a \in \overline{\Ri}} \Mod_{\bar K[[u-a]]}^d \\
    \M_p & \mapsto & (\M_{p,a})_a
\end{array}$$
is injective and commutes with sums and intersections. Hence, one can
substitute to $\M_p$, the (finite) family consisting of all non trivial
$\M_{p,a}$'s. This way, we just have to work with modules defined over
discrete valuation rings.

Note finally that there exist algorithms to compute one 
representation from the other. Indeed, remark first that computing the 
image of $\M_p$ by $\Xi$ is trivial if $\M_p$ is represented by a matrix 
of generators: it is enough to map all coefficients of this matrix to 
all $\bar K [[u-a]]$'s. Going in the other direction is more subtle but
is explained In \cite{caruso-LU}, \S 2.3.

\subsection{A generalisation of Iwasawa's theorem and
applications}\label{subsec:iwasawa}

The aim of this subsection is to present an algorithm with oracle to 
compute the maximal module associated to a $\mSnu$-module. Moreover, as 
a byproduct of our study, we will derive an upper bound on the number of 
generators of a maximal sub-$\mSnu$-module of $\mSnu^n$.

The idea of our construction (inspired by an algorithm of Cohen) is to 
consider the matrix of relations of a module and to perform elementary 
operations preserving quasi-isomorphisms to put this matrix in a certain 
form. In order to do so, we first need a way to compute the matrix of 
relations of a module or at least a certain approximation of it.
Let $\M$ be a torsion-free finitely generated $\mSnu$-module and let
$(e_1, \ldots, e_k) \in \M^k$ be a family of generators of $\M$.  We denote by
$\Rc$ the module of relations of $(e_1, \ldots, e_k)$ that is the set
of $(\lambda_1, \ldots, \lambda_k) \in \mSnu^k$ such that $
\sum_{i=1}^k \lambda_i e_i=0$. Let $r$ be the rank of $\M
\otimes_{\mSnu} \mSpnu$. From the exact sequence
\begin{equation}\label{eq:exact}
0 \rightarrow \Rc
\otimes_{\mSnu} \mSpnu \rightarrow \mSpnu^k \rightarrow \M
\otimes_{\mSnu} \mSpnu \rightarrow 0,
\end{equation}
we deduce that $\Rc \otimes_{\mSnu} \mSpnu$ is a free module over 
$\mSpnu$ of rank $\ell=k-r$. Let $(f_1, \ldots, f_{\ell})$ be a basis of 
$\Rc \otimes_{\mSnu} \mSpnu$ and set $\Rc'=\oplus_{i=1}^{\ell} 
(\mSpnu\cdot f_i \cap \mSnu^k)$. Apparently, $\Rc'$ is a 
sub-$\mSnu$-module of $\Rc$ which is free of rank $\ell$. Indeed, if 
$n_i$ denotes the smallest integer such that $\pi^{n_i}\cdot f_i \in 
\mSnu^k$, then the family $(\pi^{n_i}\cdot f_i)$ is a basis of $\Rc'$.
Moreover, we have the inclusion $\Rc' \supset \pi^N \Rc$ for
a certain $N$ since $\Rc' \otimes_{\mSnu} \mSpnu = \Rc \otimes_{\mSnu}
\mSpnu$. Now, from the knowledge of the matrix $M \in M_{d \times
k}(\mSnu)$ whose column vectors are the coordinates of $e_i$ in the
canonical basis of $\mSnu^d$, we can compute a matrix $R' \in M_{k
\times \ell}(\mSnu)$  of generators of $\Rc'$ using the algorithms of
\S \ref{subsec3.1}.  We have by definition $M.R'=0$. Of course in the
above construction, we can replace, {\it mutatis mutandis} the
localisation with respect to $\pi$ by the localisation with respect to
$u^\alpha/\pi^\beta$. 

\subsubsection{An algorithm to compute the maximal module}
We start with a couple of matrices $M=(m_{i,j}) \in M_{d \times
k}(\mSnu)$ and $R=(r_{i,j}) \in M_{k \times \ell}(\mSnu)$
representing the generators of $\M$ embedded in $\mSnu^d$ and a
sub-module of $\Rc$ containing $\pi^N \Rc$ for
a certain $N$. We are going to prove by induction that we can put $R$
in triangular form by using elementary operations on the rows of $R$
and the columns of $M$ which preserve $\M$ up to quasi-isomorphism. We 
suppose that for a positive integer $i_0$ there is a strictly increasing 
function $t : [1,i_0] \rightarrow \N^*$ such that
\begin{itemize}
  \item for all $i=1, \ldots, i_0-1$, for
$j > i$, and $t(i) \leq m < t(i+1)$, $r_{j,m}=0$ ; 
\item for
all $i=1, \ldots, i_0$,
for all
$j > t(i)$, $r_{i,j}=0$.
\end{itemize}
The matrix $R$ has the following shape:

\begin{equation}\label{sec3.2:formR} 
R = \raisebox{-1.2cm}{
\begin{tikzpicture}[ms/.style={densely dotted}] 
\matrix [ampersand replacement=\&, left delimiter=(, right delimiter=)] 
{ \node (T1) {$r_{1,t(1)}$}; \& \& \&  \& \&  \\ 
\node (star1) {}; \& \& \&  \& \& \& \\ \& \& \&	\& \& \& \\ 
\& \& \& \& \& \& \\ 
\& \& \& \node (Td) {$r_{i_0,t(i_0)}$}; \& \& \& \\ 
\& \& \& \& \& \& \\ 
\node () {}; \& \& \& \& \node (z1) {$\star$}; \& \node () {}; \& \node (z2)
  {$\star$}; \\ 
  \node () {}; \& \& \& \& \& \& \\ 
\node (star3) {}; \& \& \& \& \node (z3) {$\star$}; \& \& \node (z4) {$\star$}; \\ };
\draw[ms] (T1) -- (Td); 
\draw[ms] (z1) -- (z2); 
\draw[ms] (z1) -- (z3); 
\draw[ms] (z3) -- (z4); 
\draw[ms] (z2) -- (z4); 
\end{tikzpicture} },
\end{equation}
where the blanks represent $0$ entries.

We set $t(i_0+1)$ to be the first integer $t$ 
such that $t(i_0) < t \leq \ell$
and there exists a $j \geq i_0+1$ with $r_{j,t}\neq 0$. If no such
integer exists then we have finished.
In order to describe operations on rows (resp. columns) of a matrix
$T$ of dimension $k \times \ell$
it is convenient to denote the row vectors of $T$ (resp. the column
vectors of $T$) by $L_i(T)$ for $i=1,\ldots, k$ (resp. $C_i(T)$ for
$i=1,\ldots, \ell$). We say that the
condition $\Cond(i)$ on $R$ is satisfied if there exist two
different indices $j_0, j_1 \in \{ 1, \ldots, k\}$ such that
$r_{j_0,t(i)}\cdot r_{j_1,t(i)}\neq 0$, $v_\nu(r_{j_0,t(i)})
\leq v_\nu(r_{j_1,t(i)})$ and $\deg_W(r_{j_0,t(i)}) \leq
\deg_W(r_{j_1,t(i)})$. We apply the algorithm ColumnReduction (see
Algorithm \ref{sec3.2:algo1}) on $R,M,i_0+1,t(i_0+1)$.

\begin{algorithm}
\SetKwInOut{Input}{input}\SetKwInOut{Output}{output}

\Input{
  \begin{itemize}
\item $M \in M_{d \times k}(\mSnu)$
    \item 
  $R \in M_{k \times \ell}(\mSnu)$ in the form (\ref{sec3.2:formR}),
\item $i,t(i) \in \N$
\end{itemize}
}
\Output{$R$,$M$ such that $M\cdot R=0$ and $R$ does not satisfy condition $\Cond(t(i))$}
\BlankLine
  \While{$\Cond(t(i))$ is satisfied}{
Pick up $j_0,j_1 \in \{ 1, \ldots, k\}$ such that
$r_{j_0,t(i)}\cdot r_{j_1,t(i)}\neq 0$, $v_\nu(r_{j_0,t(i)})
\leq v_\nu(r_{j_1,t(i_0+1)})$ and $\deg_W(r_{j_0,t(i)}) \leq\deg_W(r_{j_1,t(i)})$\;
$(q,r) \leftarrow \mathrm{EuclideanDivision}(r_{j_0,t(i)},r_{j_1,t(i)})$\;
$C_{j_0}(M) \leftarrow C_{j_0}(M)+qC_{j_1}(M)$\;
$L_{j_1}(R) \leftarrow L_{j_1}(R) - q L_{j_0}(R)$\;\label{step4algo1}
}
\Return{$M,R$}\;
\caption{ColumnReduction (preliminary version)}\label{sec3.2:algo1}
\end{algorithm}

It is clear that the matrix $M$ returned by Algorithm \ref{sec3.2:algo1}
represents the same module $\M$ since it modifies $M$ by performing
elementary operations on the columns. 
Moreover, the algorithm preserves the relation $M\cdot R=0$.
The effect of the operation of
Step \ref{step4algo1} of Algorithm \ref{sec3.2:algo1} on the entry
$r_{j_1,t(i)}$ of $R$ is either 
\begin{itemize}
  \item
replace it by $0$,
\item or it decreases strictly its Weierstrass
degree and it increases its Gauss valuation.
\end{itemize}

Hence, it is easily seen that after a finite number of
loops the conditions $\Cond(t(i_0+1))$ will no longer be satisfied on $R$.  
It may happen that there is only one nonzero entry on the
$t(i_0+1)^{th}$ column of $R$ and in this case, we are basically done:
by permuting the rows of $R$ we can suppose that the non zero entry is
$r_{i_0+1,t(i_0+1)}$. Next, we remark that the vector $v$ of $\M$ whose
coordinates in the canonical basis of $\mSnu^d$ is given by the
$(i_0+1)^{th}$ column of $M$ verifies $r_{i_0+1,t(i_0+1)}\cdot v=0$ which
means that $v=0$ and we can set $r_{i_0+1,j}=0$ for $j> t(i_0+1)$.

If there are several nonzero entries on the $t(i_0+1)^{th}$ column of
$R$ and the condition $\Cond(t(i_0+1))$ is not satisfied on $R$, we
let $j_0$ be such that $v_\nu(r_{j_0,t(i_0+1)}) = \min_{1 \leq j \leq
k}\{ v_\nu(r_{j,t(i_0+1)})\}$. Note that we have
$v_\nu(r_{j_0,t(i_0+1)}) < v_\nu(r_{j,t(i_0+1)})$ for $j \neq j_0$
because on the contrary, the condition $\Cond(t(i_0+1))$ would be
satisfied on $R$.  By multiplying the $t(i_0+1)^{th}$ column of $R$ by an
element of $\mSpnu$ with valuation $-v_\nu(r_{j_0,t(i_0+1)})$, we can
moreover suppose that $v_\nu(r_{j_0,t(i_0+1)})=0$. Let $\delta =
\min_{j\neq j_0}(v_\nu(r_{j,t(i_0+1)}))$.

\paragraph{The case $\nu=0$} First, we suppose that $\nu=0$ from which we
deduce that $\delta$ is a positive integer. Denote
by $e_1, \ldots, e_k$ the generators of $\M$ represented by the column
vectors of the matrix $M$. Denote by $\M_1$ the module generated by
$(e'_j)_{j=1 \ldots k}$ with $e'_j=e_j$ for $j\neq j_0$ and
$e'_{j_0}=\frac{1}{\pi} e_{j_0}$.  The identity of $\mSnu^d$ induces
an inclusion $f:\M \rightarrow \M_1$. It is clear that the cokernel of
$f$ is annihilated by $\pi$.  Moreover, we have 
\begin{equation}\label{sec3.2:eq2} 
r_{j_0,t(i_0+1)}.e'_{j_0}= \sum_{j \neq j_0} \frac{r_{j,t(i_0+1)}}{\pi} e_j.  
\end{equation} 
As the right hand side of (\ref{sec3.2:eq2}) is in $\M$ since
$\frac{r_{j,t(i_0+1)}}{\pi} \in \mSnu$, the cokernel of $f$ is also
annihilated by $r_{j_0,t(i_0+1)}$ which is a distinguished element of
$\mSnu$. We conclude that $f$ is a quasi-isomorphism.

We denote by $O_1(j)$ the operation on the couple of matrices $(M,R)$ 
which consists in multiplying by $\frac{1}{\pi}$ the $(j)^{th}$ column 
of $M$ and multiplying by $\pi$ the $(j)^{th}$ row of $R$. Keeping the 
hypothesis and notations of the preceding paragraph, it is clear that if 
$(M,R)$ represents the module $\M$ and its relations, then the matrices 
resulting of the operation of $O_1(j_0)$ represents the module $\M_1$ 
which is quasi-isomorphic to $\M$. By repeating operations of the form 
$O_1(j)$ a finite number of time, we can suppose that $\delta =0$. But 
it means that the condition $\Cond(t(i_0+1))$ is not satisfied on $R$ 
and we can call again Algorithm \ref{sec3.2:algo1}.

We thus obtain the algorithm ColumnReduction (final version) which
takes a relation matrix of the form (\ref{sec3.2:formR}) for $i_0$ and
returns a relation matrix of the same form for $i_0+1$.  The algorithm
MatrixReduction (final version), Algorithm \ref{algo:matrixreduction}, uses
ColumnReduction in order to compute a new set of generators of a
module quasi-isomorphic to $\M$ the relation matrix of which has a
triangular form.

\begin{algorithm}
\SetKwInOut{Input}{input}\SetKwInOut{Output}{output}
\Input{
  \begin{itemize}
    \item 
  $R \in M_{k \times \ell}(\mSnu)$, 
\item $M \in
M_{d \times k}(\mSnu)$ such that $M\cdot R=0$.
\end{itemize}
}
\Output{
$R \in M_{k \times \ell}(\eSnu)$, $M \in
M_{d \times k}(\eSnu)$  such that $M\cdot R=0$ and $R$ is a triangular
matrix.}
\BlankLine
$i_0 \leftarrow 0$\;
$t(i_0) \leftarrow 1$\;
\While{$i \leq k$}{
$t(i_0) \leftarrow \min\{t | t>t(i_0)\, \text{and}\, \exists
j > 0, \text{with}\, r_{j,t} \neq 0$ \}\;
$i_0 \leftarrow i_0 +1$\;
$M,R \leftarrow \mathrm{ColumnReduction}(M,R,i_0,t(i_0))$\;
\For{$j \leftarrow t(i_0)+1$ {\bf to} $\ell$}{
$r_{i_0,j} \leftarrow 0$}
}
\caption{MatrixReduction for the case $\nu=0$}\label{sec3.2:algo2}
\end{algorithm}

\begin{algorithm}
\SetKwInOut{Input}{input}\SetKwInOut{Output}{output}

\Input{
  \begin{itemize}
\item $M \in
M_{d \times k}(\mSnu)$,
      \item 
  $R \in M_{k \times \ell}(\mSnu)$ in the form (\ref{sec3.2:formR}), 
\item $i,t(i) \in \N$ the position of the last non zero "diagonal"
  entry of $R$.
\end{itemize}}
\Output{$R$,$M$ such that $M.R=0$ and $R$ is triangular up to the
$i+1$ row.}
\BlankLine
\While{$\exists j_0,j_1$ such that $j_0 \neq j_1$ and $r_{j_0,t(i)}\cdot r_{j_1,t(i)} \neq 0$}{
  \While{$\Cond(t(i))$ is satisfied}{
Pick up $j_0,j_1 \in \{ 1, \ldots, k\}$ such that
$r_{j_0,t(i)}\cdot r_{j_1,t(i)}\neq 0$,
$v_\nu(r_{j_0,t(i)})
\leq v_\nu(r_{j_1,t(i)})$ and $\deg_W(r_{j_0,t(i)}) \leq\deg_W(r_{j_1,t(i)})$\;
$(q,r) \leftarrow \mathrm{Euclidean Division}(r_{j_0,t(i)},r_{j_1,t(i)})$\;
$C_{j_0}(M) \leftarrow C_{j_0}(M)+qC_{j_1}(M)$\;
$L_{j_1}(R) \leftarrow L_{j_1}(R) - q L_{j_0}(R)$\;
}
Let $j_0$ be such that $\deg_W(r_{j_0,t(i)}) = \max_{1 \leq j \leq k}\{ \deg_W(r_{j,t(i)})\}$\;
$\delta \la \min_{j\neq j_0}(v_\nu(r_{j,t(i)}))-v_\nu(r_{j_0,t(i)})$\;
$C_{j_0}(M) \leftarrow \frac{1}{\pi^{\delta}}
C_{j_0}(M)$\;
$L_{j_0}(R) \leftarrow \pi^{\delta}
L_{j_0}(R)$\;
}
\Return{$M,R$}\;
\caption{ColumnReduction (final version) for
$\nu=0$}\label{sec3.2:algo3}
\end{algorithm}

\paragraph{The general case}
We reduce the general case to the case $\nu=0$, by using Lemma
\ref{sec1:lempiso}.
Let $\Rpi$ in an algebraic closure of $K$ be such
that $\Rpi^\alpha=\pi$.  Let $\Ri'=\Ri[\Rpi]$, ${\eSnu}=\mSnu
\otimes_{\Ri} \Ri'$ and $\M'= \M \otimes_{\mSnu} {\eSnu}$. The
valuation on $\Ri$ (resp. the Gauss valuation on ${\mSnu}$) extends 
uniquely to $\Ri'$ (resp. to $\eSnu$). We have
$v_\nu(\Rpi)=1/\alpha$. The algorithm for the general case is exactly
the same as for the case $\nu=0$ up to the point when
$\Cond(t(i_0+1))$ is not satisfied. 
By multiplying the $t(i_0+1)^{th}$ column of $R$ by
$\Rpi^{-v_\nu(r_{j_0,t(i_0+1)})\cdot \alpha}$, we can
moreover suppose that $v_\nu(r_{j_0,t(i_0+1)})=0$. Let $\delta =
\min_{j\neq j_0}(v_\nu(r_{j,t(i_0+1)}))$.

With this setting, we can define a quasi-isomorphism in the same
manner as before.  Namely, let $e_1, \ldots, e_k$ be the generators of
$\M'$ as a sub-module of ${\eSnu}^d$ represented by the column vectors
of the matrix $M$.  Denote by $\M'_1$ the module generated by
$(e'_j)_{j=1\ldots k}$ where $e'_j=e_j$ for $j\neq j_0$ and $e'_{j_0}=
\frac{1}{\Rpi^\delta} e_{j_0}$. Then the natural injection $\M'
\rightarrow \M'_1$ is a quasi-isomorphism.  We denote by
$O_2(j,\delta)$ the operation on the couple of matrices $(M,R)$ with
coefficients in $\eSnu$ which consists in multiplying by
$\frac{1}{\Rpi^\delta}$ the $(j)^{th}$ column of $M$ and multiplying
by $\Rpi^\delta$ the $(j)^{th}$ row of $R$. With the hypothesis and
notations of this paragraph (\emph{i.e.} $M$ has the form
(\ref{sec3.2:formR})), if $(M,R)$ represents the module $\M'$ and its
relations, then the matrices $(M',R')$ resulting of the operation of
$O_2(j_0,\delta)$ represents the module $\M'_1$ which have been shown
to be quasi-isomorphic to $\M'$ (as a $\eSnu$-module). Moreover, $R'$
verifies the condition $\Cond(t(i_0+1))$.

The matrix $M'$ (resp. $R'$), resulting of the operation
$O_2(j,\delta)$ is made of column (resp.  row) vectors with
coefficients in $\mSnu$ multiplied by $\Rpi^\delta$ for a certain
$\delta \in \frac{1}{\alpha} \Z$. An important claim is that this
structure will be kept intact in the course of the computations
involving all the elementary operations introduced up to now.
In  fact, these operations on the
rows of $R$ are:
\begin{itemize}
  \item multiplication of a row by a $\Rpi^\alpha$, for $\alpha$ an
    integer ;
  \item permutation of the rows ;
  \item for $j_0, j_1 \in \{1, \ldots, k \}$, replacing $L_{j_1}(R)$
    by $L_{j_1}(R) - q' L_{j_0}(R)$ where $q'$ is the quotient of
    $\Rpi^\alpha_1\cdot y$ by $\Rpi^{\alpha_0}\cdot x$ for $x,y \in
    \mSnu$ and $\alpha_0, \alpha_1 \in \N$.
\end{itemize}
It is clear that the two first operations does not change the
structure of $R$ and the same thing is true for the last operation.
Indeed,
let $q \in \mSpnu$ and $r \in \mSpnu \cap K[u]$ with $\deg(r)
\leq \deg_W(x)$, be such that $y=q\cdot x + r$, then for $\alpha_0,
\alpha_1 \in \N$, we have $\Rpi^{\alpha_0}\cdot y =
\Rpi^{\alpha_0-\alpha_1}q\cdot \Rpi^{\alpha_1}x+\Rpi^{\alpha_0}r$ so
that we have $q'= \Rpi^{\alpha_0-\alpha_1}q$ with $q \in \mSnu$.

In order to prove formality this claim and take advantage of it to carry
out all the computations in the smaller $\mSnu$ coefficient ring, we
represent the couple of matrices $(M',R')$ with coefficients in
$\eSnu$ by a triple $(M,R,L)$ where $M,R$ are matrices with
coefficients in $\mSnu$ and $L=[\alpha_1, \ldots, \alpha_k]$ is a list
of integers such that for $i=1,\ldots,k$, $C_i(M') = \Rpi^\alpha_i
C_i(M)$ and $L_i(R') = \Rpi^{-\alpha_i} L_i(R)$.  We say that the
condition $\Cond'(i)$ on $R$ is satisfied if there exists two
different $j_0, j_1 \in \{ 1, \ldots, k\}$ such that
$r_{j_0,t(i)}\cdot r_{j_1,t(i)}\neq 0$,
$v_\nu(r_{j_0,t(i)})+\frac{\alpha_{j_0}}{\alpha} \leq
v_\nu(r_{j_1,t(i)}) +\frac{\alpha_{j_1}}{\alpha} $ and
$\deg_W(r_{j_0,t(i)}) \leq \deg_W(r_{j_1,t(i)})$. With these
notations, we can write the final version of the MatrixReduction
algorithm (see Algorithm \ref{algo:matrixreduction}) which encode the matrices $M',R'$ with coefficients in
$\eSnu$ with a couple $M,R$ of matrices with coefficients in $\mSnu$
and a list of integers.

\begin{algorithm}
\SetKwInOut{Input}{input}\SetKwInOut{Output}{output}
\Input{
  \begin{itemize}
    \item 
  $R \in M_{k \times \ell}(\mSnu)$, 
\item $M \in
M_{d \times k}(\mSnu)$ such that $M\cdot R=0$.
\end{itemize}
}
\Output{
$R \in M_{k \times \ell}(\eSnu)$, $M \in
M_{d \times k}(\eSnu)$, $L$  such that $M\cdot R=0$ and $R$ is a triangular
matrix.}
\BlankLine
$i_0 \leftarrow 0$\;
$t(i_0) \la 1$\;
$L \leftarrow [0, \ldots, 0]$\;
\While{$i \leq k$}{
$i_0 \leftarrow i_0 +1$\;
$t(i_0) \leftarrow \min\{t | t>t(i_0)\, \text{and}\, \exists
j > 0, \text{with}\, r_{j,t} \neq 0$ \}\;
\While{$\exists j_0,j_1$ such that $j_0 \neq j_1$ and $r_{j_0,t(i_0)}\cdot r_{j_1,t(i_0)} \neq 0$}{
  \While{$\Cond'(t(i_0))$ is satisfied}{
Pick up $j_0,j_1 \in \{ 1, \ldots, k\}$ such that
$r_{j_0,t(i_0)}\cdot r_{j_1,t(i_0)}\neq 0$,
$v_\nu(r_{j_0,t(i_0)})+\frac{L[j_0]}{\alpha}
\leq v_\nu(r_{j_1,t(i_0)})+\frac{L[j_1]}{\alpha}$ and $\deg_W(r_{j_0,t(i_0)}) \leq\deg_W(r_{j_1,t(i_0)})$\;
\If{$v_\nu(r_{j_0,t(i_0)}) > v_\nu(r_{j_1,t(i_0)})$}{
  $\delta_0 \leftarrow \lceil v_\nu(r_{j_0,t(i_0)}) -
  v_\nu(r_{j_1,t(i_0)}) \rceil$\;
  $L_{j_1}(R) \leftarrow \pi^{\delta_0} L_{j_1}(R)$\;
  $C_{j_1}(M) \leftarrow \pi^{-\delta_0} C_{j_1}(M)$\;
  $L[j_1] \leftarrow L[j_1] + \alpha\cdot \delta_0$\;
}
$(q,r) \leftarrow \mathrm{Euclidean Division}(r_{j_0,t(i_0)},r_{j_1,t(i_0)})$\;
$C_{j_0}(M) \leftarrow C_{j_0}(M)+qC_{j_1}(M)$\;
$L_{j_1}(R) \leftarrow L_{j_1}(R) - q L_{j_0}(R)$\;
}
Let $j_0$ be such that $\deg_W(r_{j_0,t(i_0)}) = \max_{1 \leq j \leq k}\{ \deg_W(r_{j,t(i_0)})\}$\;
$\delta \la \min_{j\neq j_0}(v_\nu(r_{j,t(i_0)}))-v_\nu(r_{j_0,t(i_0)})$\;
$C_{j_0}(M) \leftarrow \frac{1}{\pi^{\lfloor \delta \rfloor}}
C_{j_0}(M)$\;
$L_{j_0}(R) \leftarrow \pi^{\lfloor \delta \rfloor}
L_{j_0}(R)$\;
$L[j_0] \leftarrow L[j_0] + \delta -\lfloor \delta
\rfloor$\;
}
\For{$j \leftarrow t(i_0)+1$ {\bf to} $\ell$}{
$r_{i_0,j} \leftarrow 0$}
Let $j_0 \in \{1, \ldots, k\}$ be such that $r_{j_0,t(i_0)}\neq 0$\;
$(C_{j_0}(M), C_{i_0}(M)) \la (C_{i_0}(M), C_{j_0}(M))$\;
$(L_{j_0}(R), L_{i_0}(R)) \la (L_{i_0}(R), L_{j_0}(R))$\;
}
\caption{MatrixReduction}\label{algo:matrixreduction}
\end{algorithm}

\begin{example}
  We illustrate the operation of the algorithm on the
  module of example \ref{mysimpleexample}.  Recall that $\M$ is the
  submodule of $\mS_0$ generated by $(\pi^2,\pi  u^3)$.
It is represented in the canonical basis of $\mS_0$ by the matrices
$M$ of generators and $R$ of relation :
$$M=\left( \begin{matrix} \pi^2 & \pi u^3 \end{matrix}\right), R=\left( \begin{matrix} 
u^3 \\ -\pi \end{matrix} \right).$$
It is clear that $Cond(1)$ is not verified on $R$ since there is no
division possible between its entries. As a consequence, we apply
operation $O_1(1)$ on the couple $(M,R)$ to obtain:
$$M=\left( \begin{matrix} \pi & \pi u^3 \end{matrix}\right), R=\left( \begin{matrix} \pi
u^3 \\ -\pi \end{matrix} \right).$$
Now, we have $\pi u^3=-u^3\cdot \pi$ and by applying on $M$  (resp.
$R$) an
elementary operation on the columns (resp. rows), we get finally :
$$M=\left( \begin{matrix} \pi & 0 \end{matrix}\right), R=\left(
\begin{matrix} 0
\\ -\pi \end{matrix} \right).$$
An we deduce that the maximal module associate to $\M$ is $\pi.\mS_0$.
\end{example}

\subsubsection{Computation of $\Max(\M)$}\label{subsec:compmax}
Let $M_1,R_1,L_1=\mathrm{MatrixReduction}(M,R,L=[0,\ldots, 0])$. Let
$L_1=[\beta_1, \ldots, \beta_k]$.
We denote by $\M'_1$ the sub-$\eSnu$-module of $(\eSnu)^d$ generated
by the vectors given in the canonical basis of $(\eSnu)^d$ by the
column vectors $\Rpi^{\beta_i}\cdot C_i(M_1)$ for $i \in \{1,\ldots , k \}$
such that $L_i(R_1)$ is the zero vector.
\begin{lemma}\label{sec3.2:lemmamax}
  We have $\M'_1=\Max(\M \otimes_{\mSnu} \eSnu)$.
\end{lemma}
\begin{proof} 
 Let $\M'=\M
 \otimes_{\mSnu} \eSnu$ and
  let $\M_1$ be the sub-$\eSnu$-module of $(\eSnu)^d$
  generated by the column vectors of $M_1$. It is clear
  that $\M_1=\M'_1$ since for $i \in \{1,\ldots , k \}$ such that
  $L_i(R_1)$ is not the zero vector, we have $C_i(M_1)=0$ (because
  $\M_1$ is torsion free).  As $\M_1$ is obtained from $\M'$ by a
  sequence of quasi-isomorphisms, it means that there exists a
  quasi-isomorphism $q':\M' \rightarrow \M'_1$. If
  we prove that $\M'_1$ is a free $\eSnu$-module, we are done by
  Lemma \ref{sec3.2:lemma1}.

 Consider the exact sequence $0 \rightarrow \Rc \rightarrow \mSnu^k
 \rightarrow \M \rightarrow 0$ associated to the family $(e_1, \ldots,
 e_k)$ of generators of $\M$. As $\eSnu$ is flat over $\mSnu$, and as
 $\Rc' \otimes_{\mSnu} \eSnu[1/\Rpi] = \Rc \otimes_{\mSnu}
 \eSnu[1/\Rpi]$ by definition of $\Rc'$, we have
 an exact sequence 
 \begin{equation}\label{exactseq1}
 0 \rightarrow \Rc' \otimes_{\mSnu} \eSnu[1/\Rpi]
 \rightarrow ({\eSnu}^k)[1/\Rpi] \rightarrow \M'[1/\Rpi] \rightarrow 0
 \end{equation}
 defined by the generators $(e_1, \ldots, e_k)$ of $\M'[1/\Rpi]$. It
 is clear that at each step, the algorithm ReduceMatrix describes an exact sequence 
 of the form (\ref{exactseq1}) for
 a different map $({\eSnu}^k)[1/\Rpi] \rightarrow \M'[1/\Rpi]$ since
 it preserves the relation $MR=0$. From this and the definition of $M'_1$, we deduce that
 if $\Rc_1$ is the module of relations of $\M'_1$ then
 $\Rc_1[1/\Rpi]=0$ from which we deduce that $\Rc_1=0$ and we are
 done.  \end{proof}

%

 \begin{remark}\label{rqbasis}
As a byproduct of the preceding proof, we see that 
 the vectors given in the canonical basis of $(\eSnu)^d$ by the
column vectors $\Rpi^{\beta_i}\cdot C_i(M_1)$ for $i \in \{1,\ldots , k \}$
such that $L_i(R_1)$ is the zero vector form a basis of $\M'_1$.
 \end{remark}

\begin{corollary}
  Let $\M_2 = \M'_1 \cap \mSnu^d$. Then, $\M_2=\Max(\M)$.
\end{corollary}
\begin{proof}
  The corollary is an immediate consequence of Proposition
  \ref{prop:tech} and Lemma
  \ref{sec3.2:lemmamax}. 
\end{proof}

\subsubsection{Computation with $\mSnu$-modules}

Proposition \ref{prop:tech}
and Lemma \ref{sec3.2:lemmamax} establish a one-to-one correspondence
$\Phi : \Max_{\mSnu}^d \rightarrow \Free^d_{\eSnu}$, defined by $\M
\mapsto \Max(\M \otimes_{\mSnu} \eSnu)$. 
Moreover, the image of $\Phi$ is exactly the set of free sub-$\eSnu$-modules of
${\eSnu}^d$ which admit a basis $(e_i)_{i \in I}$ where $e_i \in
(\eSnu)^d$ and $e_i =\Rpi^{\alpha_i} e'_i$ with $e'_i \in (\mSnu)^d$
and $0 \leq \alpha_i \leq \alpha$.   We have seen that a $\M \in
\Phi(\Max_{\mSnu}^d)$ can be represented by a couple $(M,L)$ where $M
\in M_{d \times k}(\mSnu)$ and $L$ is a list of
positive integers $\leq \alpha$.

From the data of a matrix representing an element of $\M \in
\Max^d_{\mSnu}$ the algorithm MatrixReduction computes the couple
$(M,L)$ representing $\Phi(\M)$. Moreover, if $\M' \in
\Phi(\Max^d_{\mSnu})$, the Algorithm \ref{sec3.3:algo1} allows to
recover $\Phi^{-1}(\M')$.  We see that we can easily go back and forth
between the different representations.  For most of the applications
however, it is convenient to represent an element of $\M \in
\Max^d_{\mSnu}$ by a couple $(M,L)$. Indeed, we have the lemma:

\begin{lemma}\label{sec:3.3.4:lemmcomp}
Let $\M_1,\M_2 \in \Max^d_{\mSnu}$, then
\begin{eqnarray*}
  \Phi(\M_1 \cap \M_2) & = & \Phi(\M_1) \cap \Phi(\M_2), \\
  \Phi(\M_1 +_\free \M_2) & = & \Phi(\M_1) +_\free \Phi(\M_2).
\end{eqnarray*}
\end{lemma}
\begin{proof}
  For the first claim, we have $\Phi^{-1}(\Phi(\M_1) \cap
  \Phi(\M_2))=\Phi(\M_1) \cap \Phi(\M_2) \cap \mSnu^d= (\Phi(\M_1)
  \cap \mSnu^d) \cap (\Phi(\M_2) \cap \mSnu^d)=\M_1 \cap \M_2$.
 
  Next, we prove the second claim. We have the following diagram of
  quasi-isomorphisms:
\begin{equation}\label{sec:3.3.4:diag1}
\raisebox{-1cm}{
\begin{tikzpicture}[normal line/.style={->}, font=\scriptsize]
\matrix(m) [ampersand replacement=\&, matrix of math nodes, row sep =
  1.5em, column sep =
0.5em]
{ \& (\M_1 + \M_2) \otimes_{\mSnu} \eSnu \&  \\
  \Max(\M_1 + \M_2) \otimes_{\mSnu} \eSnu   \&  \& \Max(\M_1
\otimes_{\mSnu} \eSnu) + \Max(\M_2 \otimes_{\mSnu} \eSnu) \\};
\path[normal line]
(m-1-2) edge node[auto] {} (m-2-3)
(m-1-2) edge node[auto] {} (m-2-1);
\end{tikzpicture}
}
\end{equation}
Thus, we have $\Max(\Max(\M_1 + \M_2) \otimes_{\mSnu} \eSnu)=\Max
((\M_1 + \M_2) \otimes_{\mSnu} \eSnu) 
=\Max(\Max(\M_1
\otimes_{\mSnu} \eSnu) + \Max(\M_2 \otimes_{\mSnu} \eSnu))$ which is
exactly the desired result.
\end{proof}

Let $\M_1, \M_2 \in \Phi(\Max^d_{\mSnu})$ be represented respectively
by the couples $(M_1,L_1)$ and $(M_2,L_2)$. Then, by Lemma
\ref{sec:3.3.4:lemmcomp} one can represent the sum $\M_1+_\free \M_2$
by applying the algorithm MatrixReduction on the couple $((M_1 M_2),
L_1+L_2)$ (where $L_1+L_2$ is the concatenation of the lists $L_1$ and
$L_2$). The representation as a couple $(M,L)$ is however not well
suited to the computation of the intersection of modules, since it
implies the computation of the kernel of a matrix with
coefficient in $\mSnu$ which is not Euclidean.

\subsubsection{The generators of a maximal module}

In order to have a complete algorithm (with oracles) to compute 
$\Max(\M)$, it remains to explain how to recover $\M_2 = \M'_1 \cap 
\mSnu^d$ from the knowledge of $\M'_1$ (see \S \ref{subsec:compmax}
for the definition of $\M'_1$). We would like also to obtain a 
bound on the number of generators of $\M_2$. By the construction of $\M'_1$, 
there exists a basis $(e_1, \ldots, e_k) \in \mSnu^d$ and $\delta_i \in 
\N$ for $i=1, \ldots, k$, such that
$\M'_1 = \bigoplus_{i=1}^k \eSnu.\Rpi^{\delta_i} e_i$. Then, we have
$\M_2=\bigoplus_{i=1}^k (\eSnu.\Rpi^{\delta_i} \cap \mSnu). e_i$.
Hence, it is enough to explain
how to compute $\M'_1 \cap \mSnu^d$ when $\M'_1$ has dimension $1$. In this
case, $\M'_1$ is generated by an element of the form
$\frac{1}{\Rpi^\delta}\cdot y$ where $y \in \mSnu$ and by definition,
we want to find generators for the $\mSnu$-module $\{ x \in \mSnu |
v_\nu(x) \geq v_\nu(\frac{1}{\Rpi^\delta}\cdot y) \}$. We are reduced
to the problem of finding generators of the $\mSnu$-module $\Nc=\{ x
\in \mSnu | v_\nu(x) \geq -\delta / \alpha \}$.

\begin{lemma}\label{sec3.2:lemma2}
  Let $\delta \in \{0, \ldots ,\alpha-1 \}$.
  We define inductively a sequence of couple of integers $(\alpha_i, \beta_i)$
  by setting $\alpha_0=0$, $\beta_0=0$. Then for $i>0$, while
  $\beta_{i-1} + \alpha_{i-1} \nu > - \frac{\delta}{\alpha}$, we
  let $(\alpha_i,\beta_i)$ be the unique couple of integers such that
\begin{itemize}
  \item
    $\beta_i + \alpha_i
  \nu \geq -\frac{\delta}{\alpha}$, 
\item for all $(x,y)\neq (\alpha_i,\beta_i) \in
  \Z^2$ such that $0 \leq x \leq \alpha_i$ and $y+x \nu \geq
  -\frac{\delta}{\alpha}$, we
  have $\beta_i + \alpha_i \nu < y + x
  \nu$,
\item $\alpha_i$ is the smallest integer strictly greater than $\alpha_{i-1}$ such that
  there exists an integer $\beta_i$ with $(\alpha_i, \beta_i)$ satisfying the two
  conditions above. 
  \end{itemize}

The family $(\pi^{\beta_i}\cdot u^{\alpha_i})$ has cardinality bounded by
$\alpha$ and is a system of generators of the $\mSnu$-module $\Nc=\{ x
\in \mSnu | v_\nu(x) \geq -\delta / \alpha \}$.

\end{lemma}
\begin{proof}
First, it is clear by definition that all the
$\pi^{\beta_i}\cdot u^{\alpha_i}$ are elements of $\Nc$. Moreover, it is
clear that $\alpha_i$ is bounded by $-\delta/\beta \mod \alpha$.

Denote by $\Nc_0$ the sub-$\mSnu$-module of $\Nc$ generated by the
family $(\pi^{\beta_i}\cdot u^{\alpha_i})$. Let $x \in \Nc$, we prove
inductively on $\deg_W(x)$ that $x$ is in $\Nc_0$. If $\deg_W(x)=0$
then $v_\nu(x) \geq 0$ so that $x=x\cdot 1$ with $x \in \mSnu$.
Suppose that $d=\deg_W(x) > 0$. As $v_\nu(x) \geq -\delta/\alpha$, by
applying Corollary \ref{cor:weierstrass}, we can write $x=q\cdot h$, with $q
\in \mSnu$ invertible and $h \in K[u]$ is a degree $d$ polynomial such
that $v_\nu(h) \geq -\delta/\alpha$ and $\deg_W(h)=d$. We have to show
that $h$ is in $\Nc_0$.  Let $i_0$ be the greatest index such that
$\alpha_{i_0} \leq d$.  Then by construction of the family
$(\alpha_i,\beta_i)$, we have $v_\nu(\pi^{\beta_{i_0}}\cdot
u^{\alpha_{i_0}}) \leq v_\nu(h)$.
Indeed, if $t$ is the term of $h$ of degree $d$ then $t \in \Nc$ and
if we write $t=\pi^{\mu}\cdot u^\chi$, we have by construction
$\beta_{i_0}+\alpha_{i_0}\nu \leq \mu + \chi \nu$.
Thus we can write $h=q_1\cdot
\pi^{\beta_{i_0}}\cdot u^{\alpha_{i_0}}+r$ where $q_1 \in \mSnu$,
$\deg_W(r) < \alpha_{i_0}$ and $v_\nu(r) \geq -\delta/\alpha$.  We can
then apply the induction hypothesis on $r$ to conclude.  \end{proof}

From the above lemma, one can easily deduce an algorithm to
compute the generators of $\Nc=\{ x \in \mSnu | v_\nu(x) \geq -\delta
/ \alpha \}$ as well as an upper bound on the number of generators.
In order to find the $\alpha_i$ we just run over all the values between
$1$ and $-\delta/\beta \mod \alpha$ and check for each of them if it
satisfies the conditions of Lemma \ref{sec3.2:lemma2}.  Nevertheless this algorithm is
inefficient and the obtained bound is far from tight.  In the
following, we explain how to obtain a tight bound as well as an
efficient algorithm to compute a family of generators of $\Nc$ by
using the theory of continued fractions.  In order to set up the
notations, we briefly recall the results from this theory that we need
(see \cite{MR0161833}). For $a_0, \ldots, a_n$
integers, the notation $[a_0; a_1, \ldots, a_n]$ refers to the value
of the continued fraction $$a_0+\cfrac{1}{a_1+\cfrac{1}{\ddots \, +
\cfrac{1}{a_n}}}.$$ 

We take the convention that $a_n\ne 1$ in $[a_0; a_1, \ldots, a_n]$ so
that every rational number can be written uniquely as a finite
continued fraction. Let $r = [a_0; a_1, \ldots, a_n]$. We let
$p_0=a_0$, $q_0=1$, $p_1=a_0 a_1+1$, $q_1=a_1$ and define inductively
$p_k = a_k p_{k-1} + p_{k-2}$, $q_k=a_k q_{k-1}+q_{k-2}$. 
The fractions $p_k/q_k$ are called the $k^{th}$ convergent of the
continued fraction $[a_0; a_1, \ldots , a_n]$. We have the properties:
\begin{itemize}
  \item the integers
$p_k$ and $q_k$ are relatively prime (see \cite[Th. 2]{MR0161833});
\item $p_k/q_k
=[a_0; a_1, \ldots, a_k]$.
\end{itemize}

\begin{definition} 

  Let $r$ be a real number, and let $\gamma$ be a positive integer.
  We say that a fraction $\frac{a}{b}$ ($b\geq\gamma$) is a  \emph{best
  approximation} (resp. a \emph{positive best approximation}) of $r$
  relatively to $\gamma$ if for all integers $c,d$ such that
  $\gamma\leq d\leq b$ and $c/d \ne a/b$ (resp. such that $\gamma\leq
  d\leq b$, 
 $dr-c > 0$ and $c/d \ne a/b$), we have $|dr-c| > |br-a|$ (resp. $dr-c > br-a>0$). We
 say simply that $\frac{a}{b}$ is a \emph{best approximation} (resp.
 \emph{a
 positive best approximation}) of $r$ if $\frac{a}{b}$ is a best
  approximation (resp. a positive best approximation) relatively to
  $1$.
\end{definition}

\begin{remark}
  Our definition of best approximation corresponds to what is often
  called in the literature \emph{best approximation of second kind}
  (see \cite{MR0161833}). 
\end{remark}

Everything we need about continued fractions is
contained in the following theorem (see \cite[Th. 15 and Th.
16]{MR0161833}).

\begin{theorem}\label{th:continued}
  Let $x=[a_0; a_1, \ldots, a_n]$. 
  \begin{enumerate}
    \item Every convergent $p_k/q_k$ is
  a best approximation of $x$. 
\item Reciprocally, every
  best approximation of $x$ is a convergent, the
  only exceptions being the cases $x=a_0+\kappa$, with $\kappa
\in [1/2,1[$,
  $\frac{p_0}{q_0}=\frac{a_0}{1}$.
  \end{enumerate}

  Moreover, for $i=0, \ldots, n-1$,
$x-\frac{p_i}{q_i}>0$ for $i$ even and $x-\frac{p_i}{q_i}<0$ for $i$
odd.
\end{theorem}

Let $r$ be a real number and $b$ an integer. In the following, it is
convenient to denote by $\min(r,b)$ (resp. $\minplus(r,b)$) the integer $a$ such that
$|b\cdot r-a|=\min\{ |b\cdot r-k|, k\in \Z \}$ (resp. such that
$b\cdot r-a=\min\{
b\cdot r-k, k\in \Z\,\text{with}\, b\cdot r-k >0\}$).
Then, for $r$ a real number and $b$ a positive integer, we let
$\proche r b=b\cdot r - \min(r,b)$ and
$\procheplus r b=b\cdot r-\minplus(r,b)$.  

\begin{example}
  Let $r=0.9$ and $b=2$. Then we have $\min(r,b)=2$,
  $\minplus(r,b)=1$, $\proche r b=-0.2$ and $\procheplus r b=0.8$.
\end{example}

We need the following lemma:
\begin{lemma}\label{sec3:remark1} We have:
  \begin{itemize}
    \item for all $j \in \{0, \ldots, n \}$,
      $\proche x {q_{j}} > 0$ if $j$ is even, $\proche x {q_{j}} < 0$
      if $j$ is odd;
    \item for $j \in \{1, \ldots, n-2\}$ for all $\zeta$ integer such
      that
      $0 \leq \zeta < a_{j+2}$,  $\zeta \cdot \proche x
      {q_{j+1}} + \proche x {q_j}$  has the same sign has
      $\proche x {q_{j}}$.
  \end{itemize}
  Moreover for all $j \in \{1, \ldots, n-2 \}$ and 
all $\zeta$ integer such
      that
      $0 \leq \zeta < a_{j+2}$,     
      $$\proche x {\zeta\cdot q_{j+1}+q_{j}} = \zeta \cdot \proche x
      {q_{j+1}} + \proche x {q_j}.$$   

\end{lemma}
\begin{proof}
  The fact that $\proche x {q_{j}} > 0$ if $j$ is even, $\proche x
  {q_{j}} < 0$ if $j$ is odd
  is an immediate consequence of Theorem \ref{th:continued}. 

   If $\zeta=0$, there is nothing
  to prove. We suppose for instance that 
  $\proche x {q_j} > 0$ and $\proche x {q_{j+1}} < 0$ (the other case
  can be treated in a similar manner). Suppose that for $0<\zeta < a_{j+2}$, we have 
\begin{equation}\label{lemma:continued:eq1}
\proche x {q_j} + \zeta \cdot \proche
  x {q_{j+1}}<0.
\end{equation}
Let $\zeta$ be the smallest verifying (\ref{lemma:continued:eq1}),
then $\zeta \geq 2$ since we have by definition of a best
approximation $|\proche x {q_{j}}| > |\proche x {q_{j+1}}|$. Then,
as $\proche x {q_j} + (\zeta-1) \cdot \proche x {q_{j+1}} > 0$,
we have
$|\proche x {q_j} + \zeta \cdot \proche x {q_{j+1}}|< |\proche x
{q_{j+1}}|$ which is a contradiction with the fact that there is no
best approximation of $x$ the denominator of which is between $q_{j+1}$ and
$q_{j+2}=a_{n+2}q_{j+1}+q_j > \zeta\cdot  q_{j+1}+q_j$.

With our hypothesis, for all
  integer $\zeta$ such that $0< \zeta < a_{j+2}$, we have $\proche x {q_j} > \proche x {q_j} + \zeta \cdot \proche
  x {q_{j+1}}$.   
  Thus we have 
  we have $\proche x {q_j} > 
\zeta (q_{j+1}\cdot x - \min(x,q_{j+1})) + q_j \cdot x -
\min(x,q_j)> 0$, so that $1/2 > (\zeta q_{j+1}+q_j)\cdot x -
\zeta \min(x,q_{j+1}) - \min(x,q_j) > 0$ (remember that as $j\geq 1$,
$\proche x {q_j} \leq 1/2$). As a consequence,  $\zeta
\min(x,q_{j+1}) +\min(x,q_j)=\min(x,\zeta q_{j+1}+q_j)$ thus
$\proche x {\zeta\cdot q_{j+1}+q_{j}} = \zeta \cdot \proche x
      {q_{j+1}} + \proche x {q_j}.$
\end{proof}

For $x=[a_0; a_1, \ldots, a_n] \in \Q$ and $\gamma$ a positive
integer, we would like to be able to obtain the list of positive best
approximations of $x$ relatively to $\gamma$. The lemma tells us that
not only the convergents $p_{2i}/q _{2i}$ for $i \in \{0, \ldots,
\lfloor n/2 \rfloor \}$ are positive best approximations of $x$ but
also the $\minplus (x, q_{2i}+\mu q_{2i+1})/(q_{2i}+\mu
q_{2i+1})$ for $i \in \{0, \ldots, \lfloor (n-2)/2 \rfloor \}$ and
$\mu$ integer such that $1 < \mu <a_{2i+2}$. The following
proposition states that these are all the positive best
approximations of $x$ and gives a generalisation for the case of a
positive $\gamma$.

\begin{proposition}\label{sec3:propfrac}
  Let $x=a/b$ where $a,b$ are relatively prime integers. Write $x=
  [a_0; a_1, \ldots, a_n]$ and denote by $p_k/q_k$ the
  sequence of convergents associated to the continued fraction $[a_0;
  a_1, \ldots, a_n]$.
  Let $\gamma<b$ be a positive integer. Let $\gamma\leq d \leq b$ be
  an integer such that
  $\frac{\min^+(x,d)}{d}$ is a positive best approximation of $x$ relatively to $\gamma$.
  Let $i$ be the biggest index such that $d-q_{2i+1} \geq
  \gamma$ and let $\lambda$ be the biggest integer such that
  $d-q_{2i+1}-\lambda\cdot q_{2i+2} \geq \gamma$. Then
  \begin{enumerate}[1)]
    \item
  $\frac{\minplus(x,d-q_{2i+1}-\lambda\cdot
  q_{2i+2})}{d-q_{2i+1}-\lambda\cdot q_{2i+2}}$ is a
  positive
  best approximation of $x$ relatively to $\gamma$. 
\item If $e$
  is such that $d-q_{2i+1}-\lambda\cdot q_{2i+2}<e<d$ then $\minplus(x,e)/e$ is not a positive best
  approximation of $x$ relatively to $\gamma$.
  \end{enumerate}

  Moreover, we have 
  \begin{equation}\label{sec3:prop:eq2}
  \procheplus x {d-q_{2i+1}-\lambda\cdot
  q_{2i+2}} - \procheplus x d =  \lambda \cdot
  \proche x {q_{2i+2}} -\proche x {q_{2i+1}}> 0.
\end{equation}
\end{proposition}

\begin{proof}

Let $i$ and $\lambda$ be defined as in the statement.  We remark that
we have $\lambda < a_{2i+3}$. Indeed, by hypothesis
$d-q_{2i+1}-\lambda\cdot q_{2i+2} \geq \gamma$, but we have
$q_{2i+3}=a_{2i+3}\cdot q_{2i+2}+q_{2i+1}$ and we know
that $d-q_{2i+3} < \gamma$.  For $0 \leq \zeta < a_{2i+3}$ an integer,
let $\mu(\zeta)=q_{2i+1}+\zeta.q_{2i+2}$, $h=d-\mu(\lambda)$.

  First, we prove that \begin{equation}\label{sec3:prop:eq1}
    \procheplus x d - \proche x {\mu(\zeta)} = \procheplus x
    {d-\mu(\zeta)}, \end{equation} if $0 \leq \zeta < a_{2i+3}$.
    Using Lemma \ref{sec3:remark1}, we obtain
    \begin{equation}\label{prop3.26:eq1} 0 \leq \min(x,\mu(\zeta)) -
      \mu(\zeta)\cdot x < 1.  \end{equation} As $0\leq d\cdot
      x-\minplus(x,d) < 1$, we have $0 \leq (d-\mu(\zeta))\cdot
      x-\minplus(x,d)+\min(x,\mu(\zeta)) < 2$.  We have to prove that
      $(d-\mu(\zeta))\cdot x-\minplus(x,d)+\min(x,\mu(\zeta)) < 1$.
      Suppose, on the contrary, that $(d-\mu(\zeta))\cdot
      x-\minplus(x,d)+\min(x,\mu(\zeta)) \geq 1$, then because of
      (\ref{prop3.26:eq1}), we have:
      \begin{equation}\label{prop3.26:eq2} 0 \leq (d-\mu(\zeta))\cdot
	x-\minplus(x,d)+\min(x,\mu(\zeta))-1 < d\cdot x-\minplus(x,d).
      \end{equation} If $\zeta \leq \lambda$ this is a contradiction
      with the hypothesis that $\frac{\min^+(x,d)}{d}$ is a positive
      best approximation of $x$ relatively to $\gamma$.  If $\zeta >
      \lambda$ then  $(d-\mu(\zeta))\cdot
      x-\minplus(x,d)+\min(x,\mu(\zeta)) < (d-\mu(\lambda))\cdot
      x-\minplus(x,d)+\min(x,\mu(\lambda))$ because $\procheplus x
      {\mu(\zeta)} > \procheplus x {\mu(\lambda)}$ by Lemma
      \ref{sec3:remark1}. Next, we remark that $(d-\mu(\lambda))\cdot
      x-\minplus(x,d)+\min(x,\mu(\lambda))<1$ by what we have just
      proved, so that we have
      $(d-\mu(\zeta))\cdot x-\minplus(x,d)+\min(x,\mu(\zeta))<1$.  In
      any case, we are done.

Now, suppose that there exists $\gamma \leq e < d$ such that 
\begin{equation}\label{prop3.eq3}
  \procheplus x d < \procheplus x e \leq \procheplus x h.
\end{equation}
For $0\leq \zeta < a_{2i+3}$ a non negative integer, let $e(\zeta)=d-\mu(\zeta)$.
Choose $\zeta$ so that $|\procheplus x e - \procheplus x {e(\zeta)}|$ is minimal.
By (\ref{sec3:prop:eq1}), we know that $\procheplus x {e(\zeta)}= \procheplus
x d - \proche x {\mu(\zeta)}$. As moreover $\procheplus
x d - \proche x {\mu(a_{2i+3})} \leq \procheplus x d$ (following
Lemma \ref{sec3:remark1}) and $\procheplus x {e(\lambda)}=
\procheplus x h$, we deduce that $\lambda \leq \zeta <
a_{2i+3}$. Suppose that $\procheplus x  e  - \procheplus x {e(\zeta)} \neq 0$. 
As for all $\zeta \in \{ \lambda, \ldots , a_{2i+3}-1
\}$, $|\procheplus x {e(\zeta+1)} - \procheplus x {e(\zeta)}|=
|\procheplus x {\mu(\zeta)} - \procheplus x {\mu(\zeta+1)}| =
\proche x {q_{2i+2}}$, we deduce that
$|\proche x {e -
e(\zeta)} |<  \proche x {q_{2i+2}}$ and the fact that $|e-e(\zeta)| <
q_{2i+3}$ contradicts the second statement of
Theorem \ref{th:continued}.

Thus, we have that $\procheplus x e  =
\procheplus x  {e(\zeta)}$. Then, from (\ref{prop3.eq3}), we can write
$\procheplus x e = \procheplus x d - \proche x {\mu(\zeta)} \leq 
\procheplus x {h} = \procheplus x d - \proche x {\mu(\lambda)}$ so that 
$\proche x {\mu(\zeta)} \geq  \proche x {\mu(\lambda)}$. 
Suppose that $\proche x {\mu(\zeta)} >  \proche x {\mu(\lambda)}$
then, as $\lambda \leq \zeta < a_{2i+3}$, it means that $\zeta >
\lambda$. But then, $e=e(\zeta)=d-\mu(\zeta) < \gamma$ which is a
contradiction with the hypothesis $\gamma \leq e$.
As a consequence, we have $\lambda=\zeta$ and $e=h$.

To finish the proof, we note that (\ref{sec3:prop:eq2}) is an
immediate consequence of (\ref{sec3:prop:eq1}) and Lemma
\ref{sec3:remark1}.
\end{proof}

Let $x$ be a rational and $\gamma$ a positive integer.  From the
Proposition \ref{sec3:propfrac}, we immediately obtain an algorithm
(see Algorithm \ref{sec3.3:algo1}) to compute the reserve ordered
list of the integers $q$ such that $\minplus(x,q)/q$ is a positive
best approximation of $x$ relatively to $\gamma$.

\begin{algorithm}
\SetKwInOut{Input}{input}\SetKwInOut{Output}{output}

\Input{
  \begin{itemize}
    \item $x=a/b=[a_0; a_1, \ldots, a_n]$ a rational number ;
    \item the lists of integers $p[k],q[k]$ for $k=0, \ldots, n$, such
      that $p[k]/q[k]$ are the convergents associated to $[a_0; a_1,
      \ldots, a_n]$; 
\item $\gamma\leq b$ a positive integer.
\end{itemize}
}
\Output{$L$ a reverse ordered list of the integers $q$ such that
$\minplus(x,q)/q$ is a positive best approximation of $x$ relatively
to $\gamma$}
\BlankLine
${\tt L} \la [b]$\;
${\tt last} \la b$\;
$t \la n$\;
\eIf{$(t+1) \mod 2 = 0$}{
  ${\tt nextqk} \la t-2$\;
}{${\tt nextqk} \la t-1$\;
}
\While{${\tt nextqk} \geq 0$}{
  \If{${\tt last} - q[{\tt nextqk}] \geq \gamma$}{
      $\displaystyle \lambda \la {\tt floor}\Big( \: \frac{{\tt last} -q[{\tt nextqk}]-\gamma} {q[{\tt nextqk+1]}}\:\Big)$ \;
      ${\tt last} \la {\tt last} - \lambda.q[{\tt nextqk}+1]$ \;
  }
  \While{${\tt last} - q[{\tt nextqk}] \geq \gamma$\label{looparith}}{
      ${\tt last} \la {\tt last} - q[{\tt nextqk}]$\;
      ${\tt L} \la {\tt last} \cup {\tt L}$ \;\label{looparithend}
  }
  ${\tt nextqk} \la {\tt nextqk} -2$\;
}
\If{$L[1] > \gamma$}{
$L \leftarrow \gamma \cup L$\;}
\Return{$L$}\;

\caption{Reverse order list of positive best approximations}\label{sec3.3:algo1}
\end{algorithm}

From Algorithm \ref{sec3.3:algo1}, it is possible to obtain a bound on
the number of positive best approximations of a rational number $x$.
In order to state the following corollary, we introduce a notation:
for $(\mu, \rho,
\chi) \in \R^2 \times \N$, we denote by $L(\mu, \rho, \chi)$
the finite arithmetic sequence with first term $\mu$, common
difference $\rho$ and length $\chi$ (if $\chi$ is zero then the
sequence is considered as empty).
\begin{corollary}
  Let $x=[a_0; a_1, \ldots, a_n]$ be a rational number, denote by
  $p_k/q_k$ for $k=0, \ldots, n$ the associated sequence of
  convergents. Let $\gamma$ be a positive
  integer. The list a positive best approximations of $x$ relatively
  to $\gamma$ has cardinality bounded by $2+\sum_{i=1}^{\lfloor n/2
  \rfloor} a_{2i}$.
  
  Denote by $L$ the finite sequence of increasing integers
  $q$ such that $\minplus(x,q)/q$ is a positive best approximation
  relatively to $\gamma$.
Let $I=\{0, \ldots, \lfloor (n-1)/2 \rfloor\}$. There exist
two sequences $(\mu_i)_{i \in I}$ and $(\chi_i)_{i \in  I}$ with
coefficients respectively in $\Q$ and $\N$ such that $L=\cup_{i \in I}
L(\mu_i, q_{2i+1}, \chi_i)$. 
Moreover, for $i \in I$, the sequence $(\procheplus x q)_{q \in
L(\mu_i,q_{2i+1},\chi_i)}$ is also an arithmetic sequence with common
difference $\proche x {q_{2i+1}} < 0$.
\end{corollary}
\begin{proof} 
  To prove the first part of the statement, it suffices to show that
  the number of elements of the list generated by the loop
  beginning in line \ref{looparith} of Algorithm \ref{sec3.3:algo1}
  for a given value of {\tt nextqk} is less than $a_{{\tt nextqk}+1}$.
  Indeed, it is clear from the initialisation of Algorithm
  \ref{sec3.3:algo1} that ${\tt nextqk}$ is
  running through the odd indices in $\{ 0, \ldots, n-1 \}$.
  Now the relation $q[{\tt
  nextqk+1}]=a_{{\tt nextqk}+1} \cdot q[{\tt nextqk}] + q[{\tt nextqk}-1]$
  implies that the loop on line \ref{looparith} is executed at most
  $a_{{\tt nextqk}+1}$ times.
  Taking into account the first and last element in the
  list $L$, we obtain that its cardinality is bounded by
  $2+\sum_{i=1}^{\lfloor n/2 \rfloor} a_{2i}$.

  The second part of the statement is clear, since
  the while loop on line \ref{looparith} build a (reverse ordered)
  arithmetic sequence of common difference $q[{\tt nextqk}]$ and
  the last point is an immediate consequence of
  (\ref{sec3:prop:eq2}).
\end{proof}

\begin{remark}\label{remark:arithseq}
  Denote by $L$ the output of Algorithm \ref{sec3.3:algo1}. By the
  corollary, $L$ is a union of arithmetic sequences each of which can
  be encoded by a triple of integers giving the first term of the
  sequence, its common difference and the number of terms of the
  sequence. Recall that $x=[a_0; a_1, \ldots,
  a_n]$. Using this encoding, the list $L$ can be represented
  (as a data structure) by $O(n)$ bits of
  information. Moreover, it is easy to modify Algorithm
  \ref{sec3.3:algo1} so that it returns the list $L$ encoded in that
  way and have running time $O(n)$. For this, we just have to replace
  lines \ref{looparith}-\ref{looparithend} by:
$$\begin{array}{l}
\displaystyle
{\tt length} \la \tt{floor}\Big( \frac{{\tt last}-\gamma}{q[{\tt nextqk}]}\Big); \smallskip \\
{\tt first} \la {\tt last} - {\tt length} \cdot q[{\tt nextqk}]; \smallskip \\
L \la ({\tt first}, q[{\tt nextqk}], {\tt length}) \cup L; \smallskip \\
{\tt last} \la {\tt first}
\end{array}$$
\end{remark}

We have everything at hand in order to compute efficiently the
generators of $\Nc=\{ x \in \mSnu | v_\nu(x) \geq -\delta / \alpha
\}$. Indeed, consider the line $\mathcal{L}$ given by the equation
$y+x.\frac{\beta}{\alpha}=-\frac{\delta}{\alpha}$. Let $\gamma =
\frac{\delta}{\beta} \mod \alpha$, where
$\frac{\delta}{\beta} \mod \alpha$ is considered as a positive integer
in $\{0, \ldots, \alpha -1 \}$. Then $-\gamma$ is the abscissa of the
first point of the line $\mathcal{L}$ with integer coordinates to the
left of the origin point. Denote by $(q_i)_{i \in I}$ the list of
integers $q_i$ such that $\minplus (\beta/\alpha,q_i)/q_i$ is a positive
best approximation of $\beta/\alpha$ relatively to $\gamma$. Then if
we set $\alpha_i = q_i - \gamma$, it is easily seen that the $\alpha_i$
are precisely the same as the one defined in the Lemma
\ref{sec3.2:lemma2}. 

\begin{corollary}\label{sec3:coro:gen}
  Let $\nu=\beta/\alpha=[a_0; a_1, \ldots, a_n]$.
  Let $\delta$ be an integer. Set
$\Nc=\{ x \in \mSnu | v_\nu(x) \geq -\delta / \alpha
\}$. Then $\Nc$ is generated elements of the form
$(\pi^{\beta_i}.u^{\alpha_i})_{i \in J}$ where the cardinality of $J$
is bounded by $2+\sum_{i=1}^{\lfloor n/2 \rfloor} a_{2i}$. 
Let $I=\{1, \ldots, \lfloor n/2 \rfloor\}$. 
There exist
two sequences $(\mu_i)_{i \in I}$ and $(\chi_i)_{i \in  I}$ with
coefficients respectively in $\Q$ and $\N$ such that $(\alpha_i)_{i
\in J}=\cup_{i \in I}
L(\mu_i, q_{2i+1}, \chi_i)$. Moreover, the sequence
$v_\nu(\pi^{\beta_i}.u^{\alpha_i})_{\alpha_i \in
L(\mu_i,q_{2i+1},\chi_i)}$ is also an arithmetic sequence. 
\end{corollary}

By gathering all the results of this section, we obtain: 
\begin{theorem}\label{th:main}
  Let $\nu = [a_0; a_1, \ldots, a_n]$.
Let $\M$ be a sub-$\mSnu$-module of $\mSnu^d$. Then a bound on the
number of generators of $\Max(\M)$ is $d.(2+\sum_{i=1}^{\lceil n/2
\rceil}
a_{2i})$. These generators can be represented by $d$ vectors of
$\mSnu^d$ and $d\cdot \lfloor n/2 \rfloor$ arithmetic sequences of the
form $L(\mu,q,\chi)$ where $q$ is the denominator of a convergent of
odd index associated to  $[a_0; a_1, \ldots, a_n]$.
\end{theorem}

\subsubsection{Application: scalar extension of $\mSnu$-modules}

Let $\nu', \nu \in \Q$ such that $\nu' > \nu$, there is a natural
inclusion $\theta_{\nu,\nu'} : \mSnu \rightarrow \mSnup$. Given a
module $\M$ over $\mSnu$, We would
like to compute the module $\Max(\M \otimes_{\mSnu} \mSnup) \in
\Max_{\mSnup}^d$. If $M=(m_{ij})\in M_{d \times k}(\mSnu)$ is a matrix
representing  $\M$, it can be done by calling the algorithm
MatrixReduction on the matrix $(\theta_{\nu,\nu'}(m_{ij}))$. 

Nevertheless, \emph{if $\M$ is maximal}, there is another better way to 
carry out this computation. Assume that $\M$ is
represented by a couple $(M',L')$ with $M' \in M_{d \times k}(\mSnu)$
and $L'=[\alpha_1, \ldots, \alpha_k]$ is a list of integers. Let
$(f_1, \ldots, f_k)$  with $f_i=\Rpi^{\alpha_i}\cdot e_i$ for
$i=1,\ldots, k$ and $e_i \in \mSnu^d$ be the basis of $\Phi(\M)$
given by the column vectors associated to the couple $(M',L')$ (see
Remark \ref{rqbasis}). Then by definition $\M$ is generated by the
sub-$\mSnu$-modules $F_i=f_i.\eSnu \cap \mSnu^d$. Moreover, using
Algorithm \ref{sec3.3:algo1}, one can recover a family of generators
of $F_i$ which are of the form $s_j \cdot e_i$ with $s_j \in \mSnu$
and following Remark \ref{remark:arithseq} it is possible to encode
the generators of $F_i$ by a list of arithmetic sequences. As this
representation is very compact, we would like take advantage of it in order to
compute the scalar extension. By working component by component,
we only have
to consider the case of a sub-$\mSnu$-module of $\mSnu$, $\Nc=\{ x \in
\mSnu | v_\nu(x) \geq -\delta / \alpha\}$ for $\delta \in \N$. Then it
has been seen in Corollary \ref{sec3:coro:gen} that $\Nc$ is generated
elements of the form $(\pi^{\beta_i}.u^{\alpha_i})_{i \in J}$.  More
precisely, write $\nu=[a_0; a_1, \ldots, a_n]$ and let $I=\{1,
\ldots, \lfloor n/2 \rfloor\}$. Then, there exists  three sequences
$(\mu_i)_{i \in I}$, where $(\zeta_i)_{i \in I}$ and $(\chi_i)_{i \in
I}$ with coefficients respectively in $\Q$, $\N$ and $\N$ such that
$(\alpha_j)_{j \in J}=\cup_{i \in I} L(\mu_i, \zeta_i, \chi_i)$. Let
$\Nc'=\Nc \otimes_{\mSnu} \mSnup$. Of course, the sequence
$(\pi^{\beta_j}.u^{\alpha_j})_{j \in J}$ has coefficients in $\mSnup$ and
is a family of generators of $\Nc'$. Hence, $\Max(\Nc')$ corresponds
to the couple $(M',L')$ where the unique element of $L'$ is given the
minimum of all quantites $\beta_j + \nu' \cdot \alpha_j$ when $j$ runs
over $J$. Now, we remark that the sequence $\beta_j + \nu' \cdot 
\alpha_j$ is arithmeric when $j$ runs over one subset 
$L(\mu_i, \zeta_i, \chi_i)$. On this subset, the minimum is reached
for the first index or the last one. Thus, to compute $L'$, it is enough 
to take the minimum over these particular indices.
It yields an algorithm whose complexity is $O(n)$ --- or $O(nd)$ for
the $d$-dimensional case --- where we recall that $n$ is the length of 
the continued fraction of $\nu$ (in particular $n = O(1 + \min(\log 
|\alpha|, \log |\beta|))$ if $\nu = \frac {\alpha}{\beta}$.)

\subsection{Comparing the two approaches}

We have introduced two different ways to represent $\mSnu$-modules and 
compute with them. It is important to compare the two approaches since 
they are well suited for different kind of applications.  We call the 
representation of \S \ref{sec:useful} the $(M_\pi, M_u)$-representation 
and the representation of \S \ref{subsec:iwasawa} the 
$(M,L)$-representation.

First, we explain how to go back and forth between the two
representations. Let $\M \in \Max_{\mSnu}^d$ given with the
$(M,L)$-presentation by the couple $(M,L)$ with $M \in M_{d \times
k}(\mSnu)$
and $L$ is a list of integers. We can recover a matrix $M_1$ with
coefficients in $\mSnu$ whose columns vectors gives generators of $\M$
in the canonical basis of $\mSnu^d$. Then to obtain the couple
$(M_\pi,M_u)$ representing $\M$ we just have to compute the Hermite
Normal Forms of $M_1 \otimes_{\mSnu} \mSpnu$ and $M_1 \otimes_{\mSnu}
\mSunu$. 

We explain how to compute the $(M,L)$-representation associated to a
$(M_\pi,M_u)$-representation in the case that the associated module
$\M \in \Max_{\mSnu}^d$ has full rank. Suppose we are given the couple
$(M_\pi, M_u)$ representing $\M$ where $M_\pi=(m_{\pi,i,j}) \in M_{d
\times k}(\mSpnu)$ and
$M_u=(m_{u,i,j}) \in M_{d \times k}(\mSunu)$.  We can suppose, by multiplying $M_\pi$ by a certain
power of $\pi$ (which is invertible in $\mSpnu$), that all the
$m_{\pi,i,j} \in \mSnu$. As the coefficients of $M_u$ are defined
modulo a certain power of $\pi$ (namely the determinant of $M_u$), we
can also suppose, by multiplying $M_u$ by a certain power of
$u^\alpha/\pi^{\beta}$ (which is invertible in $\mSunu$), that all the
coefficients of $M_u$ belongs to $\mSnu$.  Let $D_u=\det(M_u) \in
\mSnu$. On the other side, let $D_\pi =
\det(M_\pi)/\Rpi^{\alpha\cdot v_{\nu}(\det(M_\pi))} \in \eSnu$. By
definition, we have $v_\nu(D_\pi)=0$.  Denote by $\M_0^\pi$
(resp. $\M_0^u$) the sub-$\eSnu$-module of $({\eSnu})^d$ generated by the
column vectors of $D_u M_\pi$ (resp. $D_\pi M_u$), considered as matrices with
coefficients in $\eSnu$. We can prove:

\begin{lemma}\label{sec3.4:lastlemma}
  Keeping the above notations, we have:
  $$\Max((\M_u \cap \M_\pi) \otimes_{\mSnu} \eSnu) = \Max(\M_0^\pi +
  \M_0^u).$$
\end{lemma}
\begin{proof}
  Using the formula $\mathrm{adj}(M)=\det(M).M^{-1}$, 
it is clear that the column vectors of the matrix $D_u M_\pi$ (resp.
$D_\pi M_u$) belong to the $\eSunu$-module
generated by the column vectors of $M_u$ (resp.  the $\eSpnu$-module
generated by the column vectors of $M_\pi$). As a consequence, we have 
$\M_0^\pi \subset
(\M_u \cap \M_\pi) \otimes_{\mSnu} \eSnu$ and $\M_0^u \subset (\M_u
\cap \M_\pi) \otimes_{\mSnu} \eSnu$.
We deduce that $\M_0^\pi + \M_0^u \subset
(\M_u \cap \M_\pi) \otimes_{\mSnu} \eSnu$. Thus, we have $\Max((\M_u
\cap \M_\pi) \otimes_{\mSnu} \eSnu) \supset \Max((\M_0^\pi + \M_0^u)
\otimes_{\mSnu} \eSnu).$

Next, suppose that $x \in \Max((\M_u \cap \M_\pi) \otimes_{\mSnu}
\eSnu)$. By Proposition \ref{prop:qis}, it means that there exists $n
\in \N$ such that $\pi^n\cdot x \in (\M_u \cap \M_\pi)\otimes_{\mSnu}
\eSnu$ and $(u/\Rpi^\beta)^n\cdot x \in (\M_u\cap \M_\pi)
\otimes_{\mSnu} \eSnu$. Note that $D_u$ is a power of $\pi$, as a
consequence there exists $n_0 > n$ such that 
\begin{equation}\label{lemma3.35:eq1}
\pi^{n_0} \cdot x \in
\M_0^\pi \subset \M_0^\pi + \M_0^u.  
\end{equation}

We would like to prove that there exists $n_1\in \N$
such that $(u/\Rpi^\beta)^{n_1} x \in \M_0^\pi + \M_0^u$. For this, it
suffices to prove that $(u/\Rpi^\beta)^{n_1} x \mod \M_0^\pi \in
\M_0^u / (\M_0^\pi \cap \M_0^u) \subset \eSnu/\M_0^\pi$.
As $D_\pi$ is invertible in
$\eSunu$ (remember that $v_\nu(D_\pi)=0$) there exists $t \in \eSnu$
and $n_2 \in \N$ 
such that $t.D_\pi = (u/\Rpi^\beta)^{n_2} \mod \pi^{n_0}\eSnu$.
Denote by $f_1, \ldots, f_k$ the vectors whose coordinates in the
canonical basis of $(\eSnu)^d$ are given by the column vectors of
$\M_0^u$.
Now, as $(u/\Rpi^\beta)^n\cdot x \in
\M_u$ there exist $\lambda_i \in \eSunu$, for $i=1,\ldots, k$, such that 
$$(u/\Rpi^\beta)^{n}.x=\sum_{i=1}^k \lambda_i f_i.$$
But we have $(u/\Rpi^\beta)^{n}.x \in \M_\pi$ 
so that $(u/\Rpi^\beta)^{n}.x \in (\eSnu)^d$
and using the triangular form of the matrix $M_u$ (see Proposition
\ref{sec3.1:prop2})
we have that $\lambda_i \in \eSnu$ for $i=1, \ldots, k$.
By multiplying the preceding equation by $t.D_\pi$, we obtain:
$$(u/\Rpi^\beta)^{n+n_2}.x+\lambda (u/\Rpi^\beta)^n  \pi^{n_0} \cdot  x = \sum_{i=1}^k
(t.\lambda_i) (D_\pi f_i),$$
for $\lambda \in \eSnu$.
Recall that we have seen that $\pi^{n_0}\cdot x \in \M_0^\pi$, thus
$(u/\Rpi^\beta)^{n+n_2}\cdot x \mod \M_0^\pi \in \M_0^u/(\M_0^\pi \cap \M^\pi_0)$.
As a consequence by taking $n_1=n+n_2$, we have:
\begin{equation}\label{lemma3.35:eq2}
(u/\Rpi^\beta)^{n_1}\cdot x \in 
\M_0^\pi +
\M_0^u
\end{equation}

By (\ref{lemma3.35:eq1}) and (\ref{lemma3.35:eq2}), there exists a $m
\in \N$ such that $\pi^m\cdot x \in
\M_0^\pi + \M_0^u$ and $(u/\Rpi^\beta)^{m}\cdot x \in \M_0^\pi +
\M_0^u$. By applying Proposition \ref{prop:qis}, we deduce that
$x \in \Max((\M_0^\pi + \M_0^u) \otimes_{\mSnu} \eSnu)$ and we are
done.
\end{proof}

\begin{remark}
In the preceding construction, we need the extension $\eSnu$ of
$\mSnu$ just to ensure that $v_\nu(D_\pi)=0$. Thus, if
$v_\nu(\det(M_\pi)) \in \Z$, this extension is not necessary.
\end{remark}

Now, let $\M \in \Max^d_{\mSnu}$ be represented by a couple
$(M_\pi,M_u)$. As $M_\pi$ and $M_u$ are given in Hermite Normal Form,
we can easily compute $D_\pi$ and $D_u$. Let $M'_\pi=D_u M_\pi$ and
$M'_u=D_\pi M_u$. Lemma \ref{sec3.4:lastlemma} tells us that we can
then obtain the $(M,L)$-representation of $\M$ by calling the
MatrixReduction algorithm on the matrix $(M'_\pi M'_u)$.

The main advantage of the $(M_\pi,M_u)$-representation is that is
provides unique representation of maximal modules over $\mSnu$,
because of the same property for Hermite Normal Forms. Thus, it allows
to test equality between modules. We have seen also that the echelon
form is well suited to test whether $x \in \mSnu^d$ is an element of $\M
\in \Max_{\mSnu}^d$ as well as to computation the intersection of two
modules. On the other side the $(M,L)$-representation provides an
actual basis of module in $\Max_{\mSnu}^d$. Moreover, the base change
operation $\otimes_{\mSnu} \mSnup$ only makes sense in the
$(M,L)$-representation and we will see in \S \ref{sec:precision}, an
important application of this operation. Indeed, if $\nu' \geq \nu$,
altough there is a natural inclusion morphism $\mSnu \subset \mSnup$,
the two sub-rings of $\mE$, $\mSunu$ and $\mSunup$ are not comparable
by the inclusion relation.

\section{Representation and precision}\label{sec:precision}

In the previous sections, we have presented algorithms to compute with 
$\mSnu$-modules by using, as a black-box, the ring operations of 
$\mSnu$. As elements of $\mSnu$ can not be coded with a finite data 
structure, these procedures are not algorithms \emph{stricto sensus} 
since they can not be implemented on a Turing machine for instance. In 
order to turn them into algorithms, we have to explain how to represent 
mathematical objects by finite data structures. Much in the same way as 
we compute with approximations of real numbers, we can represent power 
series with coefficients $\Ri$ by truncating them up to a certain 
precision. Then we have to ensure the stability of the computations, 
\emph{i.e.} that the result is independent of the part of the input that 
we ignore. In the following, we proceed in an incremental manner. First, 
we explain how to represent the elements of the coefficient ring $\Ri$ 
of $\mSnu$ by a finite structure, then we deal with elements of $\mSnu$ 
and finally with more complex structures with coefficients in $\mSnu$ 
such as $\mSnu$-modules.

\subsection{Generality with precision}

We recall from the introduction that $\Ri$ is a complete discrete
valuation ring, and that for algorithmic applications we are mostly
interested by: 
\begin{itemize}
  \item $\Z_p$ or more generally the ring a integer
of a finite extension of $\Q_p$, 
\item the ring $k[[X]]$ of formal power
series with coefficients in a (finite) field $k$. 
\end{itemize}
In any case, if
$\pi$ denote the uniformizer element of $\Ri$ and $p_\pi$ is a
positive integer, we shall represent an element of $\Ri$ by its image
in the quotient $\Ri/\pi^{p_\pi} \Ri$. 
We suppose that there exists algorithms to compute the arithmetic
operations of the ring
$\Ri/\pi^{p_\pi} \Ri$.
We say that an element
$\overline{x} \in \Ri / \pi^{p_\pi} \Ri$ is the data of element of $x
\in \Ri$ up to $\pi$-adic precision $p_\pi$ if $x \mod \pi^{p_\pi} =
\overline{x}$.

For the complexity analysis, we shall assume that we have efficient
algorithms to perform all standard operations in quotients
$\Ri/\pi^{p_\pi} \Ri$ for all integers $p_\pi$. We discuss briefly  the
validity of this assumption for the aforementioned classical examples
of rings $\Ri$.  In the case that $\Ri=k[[X]]$, we suppose that the
operations in the field $k$ costs one unit of time and can be
represented by one unit of memory. With that in mind,
if $\Ri=k[[X]]$ there exists a trivial algorithm to perform additions.
It is optimal in the sense that its complexity is equal to the size of
the inputs.  The same thing is true if $\Ri$ is the ring of integers
of any finite extension of $\Q_p$. Things are more complicated for the
multiplication of two elements of $\Ri/\pi^{p_\pi} \Ri$, whose time
will be denoted by $T_0(p_\pi)$ in the rest of this paper. In the case
$\Ri=\Z_p$, using Strassen algorithm \cite{MR2001757}, we have
$T(p_\pi)=\sO(p_\pi)$ where the soft-O notation means that we neglect
logarithmic factors. If $\Ri$ is the ring of integer of a degree $d$
finite extension of $\Q_p$, we can represent elements of $\Ri$ with a
degree $d-1$ polynomial with coefficients in $\Z_p$ and using again
Strassen algorithm for polynomials, we have $T_0(p_\pi)=\sO(d\cdot
p_\pi)$.  If $\Ri=k[[X]]$ using again Strassen algorithm for
polynomials, we have $T(p_\pi)=\sO(p_\pi)$ (we suppose here that
operation in $k$ costs one unit of time). We can summarize these
results by saying that with the best known algorithms, the time
$T_0(p_\pi)$ is quasi-linear $\log(|\Ri / \pi^{p_\pi} \Ri|)$.

An obvious way to obtain a finite approximation of an element of $\sum
a_i u^i \in \mSnu$ is to consider a representative modulo a certain
power $p_u$ of $u$. We, thus obtain a degree $p_u-1$ polynomial with
coefficients in $\Ri$ that we can represent by a vector of dimension
$p_u$ with coefficients in $\Ri$ up to precision $p_\pi$ as before.
We call this representation the \emph{flat approximation} of an
element of $\mSnu$ with $u$-adic precision $p_u$ and $\pi$-adic
precision $p_\pi$ or the $(p_u,p_\pi)$-flat approximation. The data of
a representative with $\pi$-adic precision $p_\pi$ and $u$-adic
precision $p_u$ of an element $x =\sum a_i u^i /\pi^{\lceil i\nu
\rceil} \in \mSnu$ is given by a polynomial $\sum_{i=0}^{p_u}
\overline{a}_i u^i/\pi^{\lceil i\nu \rceil}$ such that $\overline{a}_i
= a_i \mod \pi^{p_\pi}$.  It should be remarked however that the flat
approximation is not the only possible procedure to truncate an element
of $\mSnu$ in order to obtain a finite structure. For instance, one
can represent an element of $\mSnu$ up to a certain $u$-adic precision
$p_u$ by a polynomial $\sum_{i=0}^{p_u-1} a_i u^i$ with coefficients
in $\Ri$ of degree $p_u-1$. Such a polynomial may itself be
represented by the data of $a_i \mod \pi^{p_\pi}$ for $i=0, \ldots,
p_u-1$, as before but it is also possible to represent
$\sum_{i=0}^{p_u-1} a_i u^i$ by coefficients with different $\pi$-adic
precisions $a_i \mod \pi^{p_{\pi,i}}$. Put in another way, we want to
obtain a representative of $\sum_{i=0}^{p_u-1} a_i u^i$ modulo the
$\Ri$-module $\sum_{i=0}^{p_u-1} \pi^{p_{\pi_i}} u^i/\pi^{\lceil
i\nu \rceil}\cdot \Ri$.  We call this representation the \emph{jagged
approximation}. We can generalize even further the flat and jagged
approximations. For instance, we remark that for $f=\sum a_i u^i \in
\mSnu$ the flat and jagged approximations consist in the data of
$f^{(i)}(0)/i!$ for
$i=0, \ldots, p_u-1$ but we could also provide the data of
$f^{(i)}(x)/i!$ for any $x\in K$ in the radius of convergence of $f$.

Taking into account the previous examples, we say that a \emph{data of
precision} is given by any sub-$\Ri$-module $\Ap$ of $\mSnu$. Most of
the time, but not always, we want $\mSnu /\Ap$ to be $\Ri$-module of
finite length.  Indeed, it may happen that we compute with objects of
$\mSnu$ that can be represented exactly with a finite structure.  This
is the case for instance, if the characteristic of $\Ri$ is $0$, of any
element $\Z \subset \Ri$.  In this special case, it makes sense to
consider a data of precision $\Ap$ such that $\mSnu/\Ap$ is not of
finite length in order to take into account the fact that we know
certain elements of $\mSnu$ with "infinite precision". In general, in
order to represent an element of $\mSnu^d$ by a finite data structure,
one can consider a sub-$\Ri$-module $\Ap$ of $\mSnu^d$ such that most
of the time $\mSnu^d/\Ap$ has finite length.

Then, in order to compute a function $f: \mSnu^d \rightarrow \mSnu^d$,
we would like to replace it by its approximation $f: \mSnu^d/ \Ap
\rightarrow \mSnu^d / f(\Ap)$. This naive approach does not work in
general since, as $f$ is not always $\Ri$-linear, the image by $f$ of a
data of precision is not a data of precision. Though, it is possible to
approximate $f(\Ap)$ by the smallest possible data of precision. One
way to do this is to consider a \emph{regular data precision} which is
a data of precision that is a $\mSnu$-module. Then for $x \in
\mSnu^d$ and $h \in \Ap$, one can write the first order Taylor development of $f$ in
$x$
$$f(x+h)=f(x)+df_x(h) + O(h^2).$$
Most of the time (but not always), $df_x(\Ap)$ will be the correct
data of precision (see \cite{caruso-precision} for a full discussion
about this). Note that any flat approximation is a
regular data of precision but this is not always the case that a
jagged approximation is a regular data of precision.
The computation of the function $f$ reduces to the
computation of the function on the representative up to the given
precision and the computation of the precision of the result.  
A more general precision data is intuitively less
convenient for computations since it involves more complex data
structures. For instance, each coefficient of a polynomial
representing an element of $\mSnu$ with the jagged approximation may
have very unbalanced length so that it may be difficult to adapt
asymptotically fast arithmetic for such objects. On the other side, we
are going to see shortly that even for a very common operation in
$\mSnu$ such as the computation of the Euclidean division, one may
take advantage of the flexibility of the jagged approximation. Hence, 
the choice of a representation to compute with elements of $\mSnu$ is a 
non trivial trade off between space/time complexity on the one hand
and the quantity of precision we accept to loss on the other hand.

It is convenient to represent a jagged precision by a series. For
this, let $P_\pi=\sum_{i=0}^{\infty} a_{i}u^i/\pi^{\lfloor i\nu
\rfloor} \in \mSnu$. In the following, we denote by $\Ap(P_\pi)$ the
sub-$\Ri$-module of $\mSnu$ given by $\sum_{i=0}^{\infty} a_i
u^i/\pi^{\lfloor i\nu \rfloor} \cdot \Ri$.  Moreover, if $\Ap$ is
sub-$\Ri$-module of $\mSnu$, we denote by $\repr(\Ap) : \mSnu
\rightarrow \mSnu /\Ap$ the canonical projection of $\Ri$-modules. It
is clear that $\Ap(P_\pi)$ only depends on the valuation of the
coefficients $a_i$ of $P_\pi=\sum_{i=0}^{\infty} a_{i}u^i/\pi^{\lfloor
i\nu \rfloor} \in \mSnu$.  If $p_\pi$ is an integer, we will use the
notations $\Ap_f(p_u,p_\pi)$ for $$\Ap(\sum_{i=0}^{p_u-1} \pi^{p_\pi}
u^i/\pi^{\lfloor i\nu \rfloor}+\sum_{p_u}^{\infty} u^i/\pi^{\lfloor
i\nu \rfloor})$$ which corresponds to the
$(p_u,p_\pi)$-flat approximation.  If $\Ap'$ and $\Ap$ are two
sub-$\Ri$-modules of $\mSnu$ such that $\Ap' \subset \Ap$ then there
is a canonical projection $\mSnu / \Ap' \rightarrow \mSnu / \Ap$ that
we denote also (by abuse of notation) by $\repr(\Ap)$.  If $\lambda
\in \mSnu$, and $\Ap$ is a sub-$\Ri$-module of $\mSnu$, we denote by
$\lambda.\Ap=\{ \lambda.x, x \in \Ap \}$ the sub-$\Ri$-module of
$\mSnu$. If $\lambda$ is distinguished and $\mSnu / \Ap$ has finite
length then $\mSnu / (\lambda\cdot \Ap)$ has finite length. If $\Ap, \Ap'$
are sub-$\Ri$-modules of $\mSnu$, we denote by $\Ap\cdot \Ap'$ the
submodule generated by all products $xy$ for $(x,y) \in (\Ap \times \Ap')$. 
It is clear that if $\mSnu/\Ap$ and $\mSnu/\Ap'$ have finite length then 
$\mSnu / (\Ap\cdot \Ap')$ also have finite length.

\begin{lemma}\label{lemma:prectriv}
For all $\Ap, \Ap'$ sub-$\Ri$-modules of $\mSnu$ such that $\mSnu/\Ap$
and $\mSnu / \Ap'$ have finite length, for all $x,y \in \mSnu$ we have:
\begin{enumerate}
  \item if $\Ap' \supset \Ap$ then
    $\repr(\Ap')(\repr(\Ap)(x)) =\repr(\Ap')(x)$ ;
\item $\repr(\Ap + \Ap')(\repr(\Ap)(x))+
  \repr(\Ap + \Ap')(
  \repr(\Ap')(y))= \repr(\Ap + \Ap')(x+y)$ ;
\item let $\Ap_0= y\cdot \Ap + x\cdot \Ap'+\Ap\cdot \Ap'$, then 
$$\repr(\Ap_0)(\repr(\Ap)(x))\cdot \repr(\Ap_0)(\repr(\Ap')(y))=
\repr(\Ap_0)(x\cdot y);$$
\item if $\Ap' \supset \Ap$, then 
$\repr(\Ap')(\repr(\Ap)(x))\cdot \repr(\Ap')(y)=
\repr(\Ap')(x\cdot y)$ 
\end{enumerate}
\end{lemma}
\begin{proof}
  The fist claim is trivial.
Then we have $(x+\Ap)+(y+\Ap')=x+y + (\Ap+\Ap')$ and 
$(x+\Ap)\cdot (y+\Ap')=x\cdot y + x\cdot \Ap'+y\cdot \Ap+\Ap\cdot \Ap'$. The fourth claim, is an
immediate consequence of 1 and 3.
\end{proof}

We discuss briefly the complexity of the elementary arithmetic
operations in $\mSnu$ with the $(p_u,p_\pi)$-flat approximation.
First, we remark that the size of an element of $\mSnu$ with the
$(p_u,p_\pi)$-flat approximation is in the order of $p_\pi\cdot p_u$. As
before, the time of an addition in $\mSnu$ is linear in the size of a
representative of $\mSnu$ since it reduces to the addition of two
polynomials of degree $p_u-1$ with coefficients in $\Ri / \pi^{p_\pi}
\Ri$. We denote by $T(p_u,p_\pi)$ the time cost of the multiplication
of two elements of $\mSnu$ with the $(p_u, p_\pi)-$flat approximation.
Again, by using a tweaked Strassen's algorithm, we have $T(p_u,
p_\pi)=\sO(p_u\cdot T(p_\pi))=\sO(p_u\cdot p_\pi)$.  In the following, we study
the precision of some important functions using the flat and jagged
approximation.

\subsection{Finite precision computation with elements of 
$\mSnu$}\label{subsec:finiteprec}

Most of the time, even for very elementary function
dealing with elements of $\mSnu$, it is not possible to ensure the
stability of the result without some extra assumption. We illustrate
this fact with some important examples.

First, consider the Gauss valuation function $v_\nu : K[[u]]
\rightarrow \Q$. A natural way
to define $v_\nu$ on a representative modulo $\Ap_f(p_u,p_\pi)$, with
$p_u,p_\pi$ positive integers, is to compute the valuation of the
truncated representative in $\mSnu$. For instance let $x=\pi + u^{10}$, then
$v_0( \repr(\Ap_f(9,2))(x))=v_0(\pi)=1$.  We denote also this function
by $v_\nu$.  But then we have $v_0( \repr(\Ap_f(9,2))(x))=1$ and
$v_0(\repr(\Ap_f(10,2))(x))=0$.  From the previous example, one can
see that the Gauss valuation of an element $x \in \mSpnu$
can not be computed in general from the knowledge of its
approximation.  Still, it is possible to obtain the Gauss valuation of
an element $x \in \mSpnu$ from the knowledge of its approximation if
we are given some extra-information about $x$. For instance, if
$v_\nu(\repr(\Ap_f(p_u,p_\pi))(x))=0$ and if we know furthermore that
$x \in \mSnu$ then we are sure that $v_\nu(x)=0$. More generally, it
may happen than we have a guaranty that $x \in 1/\pi^\lambda.\mSnu$
for a $\lambda \in \Z$.  Then, if $\nu$ is big enough, it is possible
to compute the valuation of $x$ from the knowledge of
$\repr(\Ap_f(p_u,p_\pi))(x)$.  

\begin{lemma}\label{lemma:changenu} Let $x=\sum a_i u^i \in
  1/\pi^\lambda\cdot \mSnu$ for $\lambda$ a positive integer. Let $p_u$ be
  a positive integer and $\overline{x} \in K[u]$ be the unique
  representative of $x \mod u^{p_u}$ of degree $< p_u$. We suppose that 
  $x \ne 0$ and that $d=\deg_W(x) < p_u$.
  
  Let $\nu' \in \Q$ be such that 
  \begin{equation}\label{sec4:lemma1:eq1}
  \nu'-\nu \geq  \frac{\lambda+v_\nu(x)}{p_u-d},
\end{equation}
  then $v_{\nu'}(x)=v_{\nu'}(\overline{x})$.
\end{lemma}

\begin{proof} 
  Let $x=\sum a_i u^i \in 1/\pi^{\lambda}\cdot \mSnu$. By definition, we have
  $v_{\nu'}(\overline{x})\leq v_K(a_d)+\nu'\cdot d$ and on the other
  side we have of course $v_{\nu'}(x) \leq v_{\nu'}(\overline{x})$.
  Thus, in order to prove the lemma, we just have to check that for all $i
  \geq p_u$, we have 
  \begin{equation}\label{eq:random1:1}
  v_K(a_i)+\nu'\cdot i \geq v_K(a_d)+\nu'\cdot d.
\end{equation}
But, using $x \in 1/\pi^{\lambda}\cdot  \mSnu$, we get
\begin{equation}\label{eq:random1:3}
v_K(a_i)+\nu\cdot i \geq -\lambda
\end{equation}
for all $i\geq p_u$. 
From  (\ref{eq:random1:1}) and (\ref{eq:random1:3}), we deduce that it
is enough to prove that $-\lambda + i(\nu' - \nu) \geq
 v_K(a_d)+\nu'.d$. Since $\nu'-\nu \geq 0$ by
 hypothesis, it suffices to prove that
 $-\lambda + p_u(\nu'-\nu) \geq v_k(a_d)+\nu'.d$ for all $i \geq p_u$.
 This is
 equivalent to 
 \begin{equation}\label{eq:random1:2}
   \nu' \geq \frac{v_K(a_d)+p_u\cdot \nu +\lambda}{p_u-d},
 \end{equation}
 which is exactly (\ref{sec4:lemma1:eq1}).
\end{proof}

This lemma, while totally elementary, shows the following very important 
fact: by increasing the $\nu$ parameter of the $\mSnu$-module, one can 
obtain guaranties on the valuation of a certain $x=\sum a_i u^i \in 
\mSnu$ from the knowledge of its representative $x=\sum_{i=1}^{p_u-1} 
a_i u^i$ with bounded Weierstrass degree under the general hypothesis of 
a lower bound on the valuation of the coefficients $a_i$.

Another important operation for the arithmetic of $\mSnu$ is the inversion.
\begin{lemma}\label{lemma:divprec}
 Let $x \in \mSnu$ and suppose that $\deg_W(x)=0$ and that $v_\nu(x)
=0$ so that by Corollary \ref{sec2:cor1}, $x$ is invertible.
Let $p_u,p_\pi$ be positive integers. Then
$\repr(\Ap_f(p_u,p_\pi))(x) \in \mSnu / \Ap_f(p_u, p_\pi)$ is also invertible and we have
$\repr(\Ap_f(p_u,p_\pi))(x)^{-1}=\repr(\Ap_f(p_u,p_\pi))(x^{-1})$.
\end{lemma}
\begin{proof}

 Write $x=\sum a_i u^i/\pi^{\lfloor i\nu \rfloor}$, $x^{-1}=\sum b_i
 u^i/\pi^{\lfloor i\nu \rfloor}$ and $c=1 = \sum c_i u^i/\pi^{\lfloor
 i\nu \rfloor}$ with $c_j =\sum_{i=0}^j a_i\cdot b_{j-i}$. We have
 $v_K(a_0)=0$ so that we can compute ${a_0}^{-1} \mod p_\pi=b_0 \mod
 p_\pi$. Then, using the formula 
 $$\frac{b_j}{\pi^{\lfloor j\nu \rfloor}} =
 \frac{1}{a_0}\cdot \sum_{i=0}^{j-1} \frac{a_i b_{j-i}}{\pi^{\lfloor i\nu
 \rfloor}\pi^{\lfloor (j-i)\nu \rfloor}},$$ 
 together with the remark that $\pi^{\lfloor j\nu \rfloor}/(\pi^{\lfloor i\nu
 \rfloor}\pi^{\lfloor (j-i)\nu \rfloor})$ is equal to $1$ or $\pi$,
 we obtain by induction for $j=1, \ldots,
 p_u-1$, $b_j \mod p_\pi$.
\end{proof}

Let $x,y \in \mSnu$.  In order to be able to compute an approximation
of the Euclidean division of $y$ by $x$ , it is necessary to know that
$v_\nu(y) \geq v_\nu(x)$. One way to have that guaranty is to be given
$x$ with enough precision to know that it is distinguished. Then one
can compute its Weierstrass degree.  In the following proposition, we
keep the notations of Proposition \ref{sec2:prop1}.

\begin{proposition}\label{sec4:prop1}
  Let $x,y \in \mSnu$. Suppose that $x$ is distinguished and let $d =
  \deg_W(x)$. Let $q \in \mSnu$ and $r \in K[u] \cap \mSnu$ be such that
  $\deg(r) < d$ and $y=q\cdot x+r$.
  Put $e=v_\nu(\Lo(x))>0$,
  let $p_\pi$ be a positive integer and 
  $p_x= \lceil p_\pi /e
  \rceil d$. Let $$P_y=\sum_{i=0}^{p_x-1} \pi^{\max\{p_\pi -
  \lfloor i /d \rfloor e-\lfloor i\nu \rfloor, -\lfloor i\nu \rfloor\}
} u^i+\sum_{i=p_x}^{\infty} \frac{u^i}{\lfloor i\nu \rfloor},$$ 
  $$P_q=\sum_{i=0}^{p_x-d-1} \pi^{\max\{p_\pi -
  \lfloor i /d+1 \rfloor e-\lfloor i\nu \rfloor, -\lfloor i\nu \rfloor\}
} u^i+\sum_{i=p_x-d}^{\infty} \frac{u^i}{\lfloor i\nu \rfloor}.$$
  There exists an algorithm which takes as input
  $\repr(\Ap_f(p_x,p_\pi))(x)$ and $\repr(\Ap(P_y))(y)$ and
  outputs
  $\repr(\Ap_f(p_x,p_\pi))(q)$ and $\repr(\Ap_f(\infty,p_\pi))(r)$.
\end{proposition}


\begin{figure}[h]
\begin{center}
\begin{tikzpicture}
\draw[step=0.5cm,lightgray, very thin] (0,-1) grid (12,2.5);
\draw[->, thick] (0,0) -- (0,2.5);
\draw[->, thick] (0,0) -- (12,0);
\draw[color=blue] (0,2) -- (2,1.8333) -- (2,1.5) -- (4,1.3333) --
(4,1) -- (6,0.83333) -- (6,0.5) -- (8,0.3333) -- (8,0) -- (10,
-0.1666) -- (10,-0.5) -- (12,-0.6666) -- (12,-1) ;
\draw[color=red] (0,0) -- (12,-1) node[anchor=north] {$y=-\nu x$};
\draw[thick] (2,0) -- (2,0.1);
\node[above] (d) at (2,0) {$d$};
\node[right] (p) at (0,2) {$\pi$};
\node[above] (u) at (12,0) {$u$};
\end{tikzpicture}
\end{center}
\caption{The form of the precision of $y$ in the Euclidean division
for $d=2$ and $\nu=1/6$.}
\end{figure}
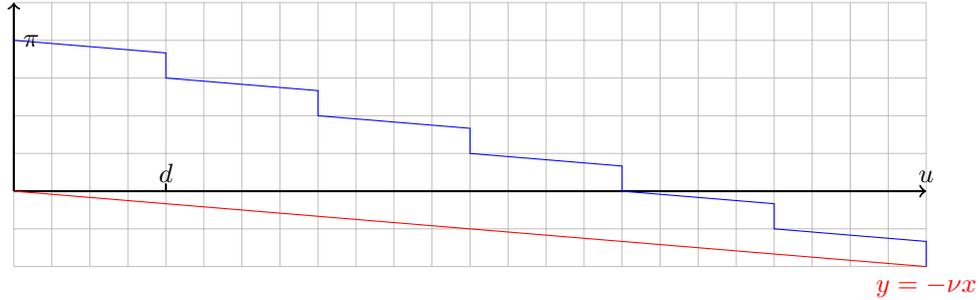

\begin{proof}
  Recall that, from Proposition \ref{sec2:prop1}, $q,r$ are the limits of
  the sequences $(q_j,r_j)$ defined by $q_0=0$ and $r_0=y$ and 
$$q_{j+1}=q_j + \frac{\Hi(r_j,d)}{\Hi(x,d)},$$
$$r_{j+1}=\Lo(r_j) - \frac{\Hi(r_j,d)}{\Hi(x,d)}\cdot \Lo(x).$$
For $j=0, \ldots, \lceil p_\pi/e \rceil$, let 
$$P_{y,j}=\sum_{i=0}^{ (\lceil p_\pi/e \rceil -j) \cdot d-1}
\pi^{\max\{p_\pi - \lfloor i /d \rfloor e -\lfloor i\nu \rfloor,
-\lfloor i\nu \rfloor\}}u^i+\sum_{i=(\lceil p_\pi/e \rceil -j)
\cdot d}^{\infty}
 \pi^{\max\{j\cdot e -\lfloor i\nu \rfloor,
-\lfloor i\nu \rfloor\}}u^i ,$$ 
and let 
 $t(j)=\repr(\Ap(P_{y,j}))(r_j)$. 
 It is clear that $t(0)=\repr(\Ap(P_y))(y)$.
 We are going to prove that if
 we know $\repr(\Ap_f(p_x,p_\pi))(x)$ and $t(j)$ then we can compute
 $t(j+1)$.  Write $\Hi(x,d)=u^d/\pi^{\nu d}\cdot x_0$, with $x_0$ an invertible element
 of $\mSnu$. Then from $\repr(\Ap_f(p_x,p_\pi))(\Hi(x,d))$, 
 we immediately obtain $\repr(\Ap_f(p_x-d,p_\pi))(\Hi(x_0,d))$,
and 
 by Lemma
 \ref{lemma:divprec}, we can compute
 $\repr(\Ap_f(p_x-d,p_\pi))(1/(\Hi(x_0,d)))$. As $\Ap_f(p_x-d,p_\pi) \subset \Ap(p_x-d\cdot j,P_{y,j})$,
 applying Lemma
 \ref{lemma:prectriv}, we deduce that 
 \begin{multline}\label{div:eq1}
 \repr(\Ap(P_{y,j}))(\Hi(r_j,d)/\Hi(x_0,d))= \\
\repr(\Ap(P_{y,j})(\repr(\Ap_f(p_x-d,p_\pi))(1/(\Hi(x_0,d)))).
\repr(\Ap(P_{y,j})(\Hi(r_j,d)).
\end{multline}
We remark that $\repr(\Ap(P_{y,j})(\Hi(r_j,d))=\Hi(t(j),d)$ so
that the left hand side of (\ref{div:eq1}) can be computed from the
known data.
Dividing $\repr(\Ap(P_{y,j}))(\Hi(r_j,d)/\Hi(x_0,d))$ by
$u^d/\pi^{\lfloor d\nu \rfloor}$, we obtain
 $\repr(\Ap(P_j))(\Hi(r_j,d)/\Hi(x,d))$, where
 \begin{multline}
 P_j=\sum_{i=0}^{
(\lceil p_\pi/e \rceil -(j+1))
\cdot d-1} \pi^{\max \{ p_\pi -
\lfloor i /d+1 \rfloor e- \lfloor i\nu \rfloor, - \lfloor i\nu
\rfloor}u^i\\
+\sum_{i=(\lceil p_\pi/e \rceil -(j+1))
\cdot d}^{\infty}
 \pi^{\max\{(j\cdot e -\lfloor i\nu \rfloor,
-\lfloor i\nu \rfloor\}}u^i.
\end{multline}
 
 Next, as
 $v_\nu(\Lo(x))=e$, still be applying Lemma \ref{lemma:prectriv}, and
 remarking that $t(j+1)=\pi^e\cdot P_j$, we
 obtain 

\begin{multline}
  \repr(\Ap(t(j+1)))(\frac{\Hi(r_j,d)}{\Hi(x,d)}\cdot \Lo(x))= \\
 \repr(\Ap(
 t(j+1)))(\repr(\Ap(P_j))(\Hi(r_j,d)/\Hi(x,d))).
\repr(\Ap_f(\infty,
 p_\pi))(\Lo(x)).
\end{multline}

From the above, we deduce by induction that we can compute
$\repr(\Ap(P_{y,\lceil p_\pi/e \rceil}))(r)$. 
But we have $\Ap(P_{y,\lceil p_\pi/e \rceil})=\Ap_f(\infty,p_\pi)$ so
that we can compute $\repr(\Ap_f(\infty,p_\pi))(r)$ as
claimed. Moreover, as we can compute
$\repr(\Ap(P_j))(\Hi(r_j,d)/\Hi(x,d))$, by Lemma \ref{lemma:prectriv}, we can compute 
$\repr(\sum \Ap(P_j))(q)=\repr(P_0)(q)$ and we are done.
\end{proof}
\begin{remark}
  In the preceding proposition, we see that in order to be able to
  compute $\repr(\Ap_f(d,p_\pi))(r)$, we really use all the information
  contained in $\repr(\Ap(p_y,P_y))(y)$. If we use the flat precision,
  then to obtain $\repr(\Ap(d,p_\pi))$ we need to
  know $\repr(\Ap(\lceil p_\pi /e \rceil.d,p_\pi))(y)$. This shows
  that the flat precision is not well adapted to the computation of
  the Euclidean division in $\mSnu$ since a lot of information about
  the operands is not useful for the computation.
\end{remark}
The proposition shows that following the computations of Algorithm
\ref{algo1} on representatives modulo the given precision, we obtain
the outputs with the guaranty that the result has the claimed
precision. 

The last operation in $\mSnu$ (actually in $\mSpnu$) that we would
like to consider is the gcd computation. To begin with, we consider
some very simple examples, for elements of $\mSnu$ which are
polynomials. Suppose that $\Ri=\Z_5$, $\nu=0$ so that
$\mSnu=\Z_5[[u]]$. Let $\overline{P_1}=\repr(\Ap_f(\infty,2))(u-1)$ and
$\overline{P_2}=\repr(\Ap_f(\infty,2))(u-2)$. Then it is clear that for
all $P_1, P_2 \in \mSnu$ such that $P_1 = \overline{P_1} \mod
\Ap_f(\infty,2)$ and $P_2=\overline{P_2} \mod \Ap_f(\infty,2)$ then
$\gcd(P_1,P_2)=1$.  This can be seen by using the Euclidean algorithm
to compute the extended gcd of $\overline{P_1}$ and $\overline{P_2}$
in $\mSnu / \Ap_f(\infty,2)$ which obviously returns $1$. In this
case, it is safe to claim that
$\gcd(\overline{P_1},\overline{P_2})=1$. 

Next, consider $\overline{P_3}=\repr(\Ap_f(\infty,2))(u-1)$ and 
$\overline{P_4}=\repr(\Ap_f(\infty,2))(u-1)$. In this case, it is very 
easy to find different representatives of $\overline{P_3}$ and 
$\overline{P_4}$ the gcd of which is not equal. For instance, we can 
take $P_3=P_4=u-1$ in this case $\gcd(P_3,P_4)=u-1$ but if we take 
$P_3=u-1$ and $P_4=u-6$ then $\gcd(P_3,P_4)=1$. If we compute the gcd of 
$\overline{P_3}$ and $\overline{P_4}$ using the Euclidean algorithm, we 
obtain $u-1$ and we do not have enough precision on the next remainder 
to decide whether it vanishes or not. This example shows that, in the 
case that the gcd of the representatives is not \emph{surely} $1$ it is 
not even clear how to define it since the result may change depending on 
the representatives in $\mSnu$ that we use in order to compute it.

\subsection{Finite precision computation with modules with coefficients in $\mSnu$}

Let $\M_1$ and $\M_2$ be two maximal sub-$\mSnu$-modules of $\mSnu^d$.
In this section, we are interested by the computation of the
maximal sum $\M_1 +_\free \M_2$ of $\M_1$ and $\M_2$. We would like to
carry out computations with finite precision and have a guaranty on
the precision of the results.  The preceding example suggests that
even in the case $d=1$, we can not hope much in that direction.
Indeed, the computation of the maximal sum of two
sub-$\mSnu$-modules of $\mSpnu$ reduces to the computation of the gcd
of two elements of $\mSpnu$ and we have seen in
\S \ref{subsec:finiteprec}, that unless this sum is $\mSpnu$, we
can not guaranty that the result computed with finite precision is an
approximation of the result computed on representatives in $\mSpnu$.

As before, we need some extra-information, that we can get from the
mathematical context of our computation, in order to guaranty the
precision of the output. A very natural extra-information that can arise in practise is
the following: let $\M_1$ and $\M_2$ be two sub-$\mSnu$-modules of
$\mSpnu^d$ and we know that there exists a positive integer $c$ such
that $\M_2 \subset 1/\pi^{c} \M_1$. 
We recognize a generalisation of
the hypothesis of Lemma \ref{lemma:changenu} where we have shown in
the case that $d=1$ that we can obtain a guaranty on the valuation
$v_\nu$ of approximations of elements of $K[[u]]$ for well chosen
$\nu$. This situation is also crucial in the paper \cite{caruso-lubicz}. We are
going to see that, although we don't know how to compute an approximation
of $\M_1 +_\free \M_2$, we can describe an algorithm which outputs an
approximation of $(\M_1 \otimes_{\mSnu} \mSnup) +_\free ({\M_2}
\otimes_{\mSnu} \mSnup)$ for a well chosen $\nu' > \nu$.

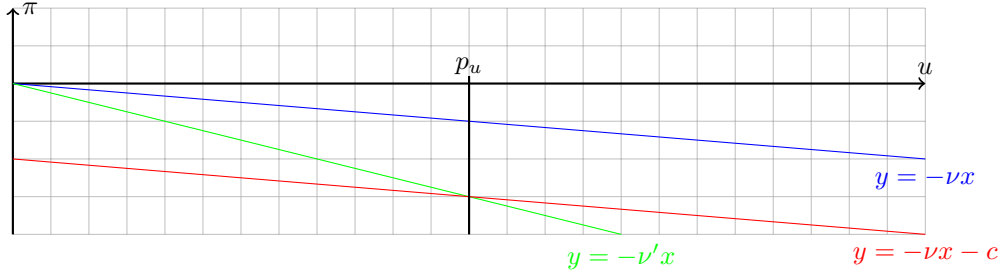
\begin{figure}[h]
\begin{center}
\begin{tikzpicture}
\draw[step=0.5cm,lightgray, very thin] (0,-2) grid (12,1);
\draw[->, thick] (0,-2) -- (0,1);
\draw[->, thick] (0,0) -- (12,0);
\draw[color=blue] (0,0) -- (12,-1) node[anchor=north] {$y=-\nu x$};
\draw[color=green] (0,0) -- (8,-2) node[anchor=north] {$y=-\nu' x$};
\draw[color=red] (0,-1) -- (12,-2) node[anchor=north] {$y=-\nu x-c$};
\draw[thick] (6,-2) -- (6,0.1);
\node[above] (d) at (6,0) {$p_u$};
\node[right] (p) at (0,1) {$\pi$};
\node[above] (u) at (12,0) {$u$};
\end{tikzpicture}
\end{center}
\caption{The computation of $\nu'$ from $p_u$ and $\nu$.}\label{newnu}
\end{figure}

In order to compute $\M_1 +_\free \M_2$, it is enough to be able to
compute $\M_1 +_\free \mSnu\cdot t$ where $t \in \M_2$. Indeed, let
$(t_1, \ldots, t_{h'})$ be a family of generators of $\M_2$, we have
$\M_1 +_\free \M_2 = \M_1 +_\free \mSnu\cdot t_1 +_\free \ldots
+_\free \mSnu\cdot t_{h'}$.  Let $t \in \M_2$ and let $(e_1, \ldots,
e_h)$ be a family of generators of $\M_1$. By our hypothesis, we know
that there exists $\lambda_i \in 1/\pi^c\cdot \mSnu$ such  that
$t=\sum \lambda_i e_i$. We remark that if all the $\lambda_i$ are in
$\mSnu$ then $t \in \M_1$ so that
$\M_1 + \mSnu\cdot t = \M_1$ and
there is nothing to do. Write $\lambda_i
= \sum_{j\geq 0} a^i_j u^j$ with $v_K(a^i_j)+\nu\cdot j \geq -c$.  Let
$p_u$ a positive integer, we are going to choose $\nu'$, as it is
explained in figure \ref{newnu}, such that $\sum_{j \geq p_u} a^i_j
u^i \in \mSnup$. For this it is enough to take $\nu' \geq \nu+c/p_u$.
Let $t'= \sum_i \lambda'_i e_i$ with 
$\lambda'_i=\sum_{j=0}^{p_u-1} a^i_j u^j$ and $t''=\sum_i \lambda''_i
e_i$ with $\lambda''_i=\sum_{p_u}^\infty a^i_j u^j$.
Using the same remark as above, we have: 
$$(\M_1 \otimes_{\mSnu} \mSnup) +_\free
(t\cdot \mSnup) = (\M_1 \otimes_{\mSnu} \mSnup) +_\free (t'\cdot
\mSnup) +_\free (t'' \cdot \mSnup) = (\M_1 \otimes_{\mSnu} \mSnup) +_\free (t'\cdot
\mSnup),$$
since $t''\cdot \mSnup \in \M_1$. Now, as $\lambda'_i$ is a
polynomial in $u$, we can obtain its valuation, greatest common
divisor and all the operations that we need in order to compute $(\M_1
\otimes_{\mSnu} \mSnup) +_\free (t\cdot \mSnup)$.  

We recall that we write
$\nu=\beta/\alpha$ with $\alpha,\beta$ relatively prime numbers and
let $\Rpi$ in an algebraic closure of $K$, be such that
$\Rpi^\alpha=\pi$.  Let $\Ri'=\Ri[\Rpi]$ and ${\eSnu}=\mSnu
\otimes_{\Ri} \Ri'$. The algorithm AddVector is an adaptation of the
algorithm MatrixReduction.

\begin{algorithm}
\SetKwInOut{Input}{input}\SetKwInOut{Output}{output}

\Input{
  \begin{itemize}
\item $M \in
M_{d\times h}(\mSnu)$, a matrix whose column vectors $C(i)$ for
$i=1,\ldots, h$ give generators of
$\M_1$ in the canonical basis of $\mSnu^d$ ;

\item a list $\lambda[1], \ldots, \lambda[h]$ such that $\sum \lambda_i
  C_i(M)=t$, $\lambda[i] \in 1/\pi^c\cdot \mSnu \cap K[u]$ and $\deg \lambda[i]
  \leq p_u-1$ for $i=1, \ldots, k$.
\end{itemize}}
\Output{$M \in M_{d \times h}(\mSnu)$ and a list $L$
  a matrix such that the column vectors $\Rpi^{L[i]}\cdot C_i(M)$ give
generators of $\M_1 +_\free t$ in the canonical basis of ${\eSnu}^d$}
\BlankLine
$L \leftarrow [0, \ldots, 0]$\;
\While{$\exists j \in \{1, \ldots, h\}$ such that 
$v_\nu(\lambda[j])-\frac{L[j]}{\alpha} < 0$}{
\While{$\Cond(\lambda,L)$ is satisfied}{
Pick up $j_0,j_1 \in \{ 1, \ldots, h\}$ such that
$\lambda[j_0]\cdot \lambda[j_1]\neq 0$,
$v_\nu(\lambda[j_0])-\frac{L[j_0]}{\alpha}
\leq v_\nu(\lambda[j_1])-\frac{L[j_1]}{\alpha}$ and
$\deg_W(\lambda[j_0]) \leq\deg_W(\lambda[j_1])$\;
\If{$v_\nu(\lambda[j_0]) > v_\nu(\lambda[j_1])$}{
$\delta_0 \leftarrow \lceil v_\nu(\lambda[j_0]) -
  v_\nu(\lambda[j_1]) \rceil$\;
  $\lambda[j_0] \leftarrow \pi^{-\delta_0} \lambda[j_0]$\;
  $L[j_0] \leftarrow L[j_0] - \alpha\cdot \delta_0$\;
}
$(q,r) \leftarrow \mathrm{Euclidean Division}(\lambda[j_0],\lambda[j_1])$\;
$\lambda[j_1] \leftarrow \lambda[j_1]-q\lambda[j_0]$\;
$C_{j_1}(M) \la C_{j_0}(M)+q C_{j_1}(M)$\;
}
Let $j_0$ such that $v_\nu(\lambda[j_0])-\frac{L[j_0]}{\alpha} =\min_{j=1, \ldots,
h}(\lambda[j])-\frac{L[j]}{\alpha})$\;
Let $j_1$ such that $v_\nu(\lambda[j_1])-\frac{L[j_1]}{\alpha}
=\min_{j\ne j_0}(\lambda[j])-\frac{L[j]}{\alpha})$\;
$L[j_0]  \la L[j_0]+\alpha v_\nu(\lambda[j_0])-L[j_0]-
\alpha v_\nu(\lambda[j_1])+L[j_1]$\;  

}
\Return{$M,L$}\;
\caption{AddVector}\label{sec4:addvertor}
\end{algorithm}
In the preceding algorithm, $Cond(\lambda,L)$ returns true if there
exists $j_0,j_1 \in \{ 1, \ldots, h\}$ such that
$\lambda[j_0]\cdot \lambda[j_1]\neq 0$,
$v_\nu(\lambda[j_0])-\frac{L[j_0]}{\alpha}
\leq v_\nu(\lambda[j_1])-\frac{L[j_1]}{\alpha}$ and
$\deg_W(\lambda[j_0]) \leq\deg_W(\lambda[j_1])$

We want to give a
consequence of this algorithm. We first need a definition.
\begin{definition}
  Let $\M$ be a sub-$\mSnu$-module of $\mSnu^d$.
  Let $\Ap$ be a sub-$\Ri$-module of $\mSnu$. We say that a matrix
  $M^r=(m^r_{ij})\in M_{d \times d'}(\mSnu / \Ap)$ is an
  $\Ap$-approximation of $\M$ is there exists a matrix $M=(m_{ij})\in
  M_{d\times d'}(\mSnu)$ whose columns
  are the coordinates of generators of $\M$ in the canonical basis of
  $\mSnu^d$ and such that $m^r_{ij}=\repr(\Ap)(m_{ij})$.
\end{definition}

By iterating this algorithm AddVector on a set of $(t_1, \ldots, t_{h'})$ of
generators of $\M_2$, we obtain the following theorem:

\begin{theorem}\label{th:main2}

  Let $\M_1$ and $\M_2$ be two finitely generated sub-$\mSnu$-modules
  of $\mSnu^d$ such that $\M_2 \subset 1/\pi^c \M_1$ for a positive
  integer $c$. Let $M_1=(m^1_{ij})$ and $M_2=(m^2_{ij})$ be the
  matrices with coefficients in $\mSnu$ of generators of $\M_1$ and
  $\M_2$ in the canonical basis of $\mSnu^d$. Let $p_u, p_\pi$ be
  positive integers and suppose that we are given 
  $M^r_1=(\repr(\Ap_0(p_u,p_\pi))(m^1_{ij}))$ 
and $M^r_2=(\repr(\Ap_0(p_u,p_\pi))(m^2_{ij}))$. 
Let $e$ be the number of columns of $M_2$ and let $\nu'=\nu +
ec/p_u$.
Then
there exists a polynomial time algorithm in the length of the
representation of $M^r_1$ and $M^r_2$ to compute a matrix 
$M^r_3=(M^3_{ij})$ with coefficients in $\mSnup / \Ap_0(p_u,p_\pi)$
which is a $\Ap_0(p_u,p_\pi)$-approximation of $$(\M_1 \otimes_{\mSnu}
\mSnup) +_\free
(\M_2 \otimes_{\mSnu} \mSnup).$$
\end{theorem}
\begin{remark}
If we suppose in the theorem that $\M_2$ is maximal, then by Theorem
\ref{th:main} we can take $e=d.(2+\sum_{i=1}^{\lceil n/2 \rceil}
a_{2i})$ where
$\nu = [a_0; a_1, \ldots, a_n]$.
\end{remark}

\bibliographystyle{plain}
\bibliography{linalgwu-hal.bib} 

\begin{thebibliography}{10}

\bibitem{caruso-precision}
Xavier Caruso.
\newblock $p$-adic precision.
\newblock {\em in preparation}.

\bibitem{MR2951749}
Xavier Caruso.
\newblock {$F_p$}-repr\'esentations semi-stables.
\newblock {\em Ann. Inst. Fourier (Grenoble)}, 61(4):1683--1747 (2012), 2011.

\bibitem{caruso-LU}
Xavier Caruso.
\newblock Random matrix over a dvr and lu factorization.
\newblock {\em preprint}, 2012.

\bibitem{caruso-lubicz}
Xavier Caruso and David Lubicz.
\newblock Semi-simplifi\'ee modulo $p$ des repr\'esentations semi-stables~: une
  approche algorithmique.
\newblock {\em in preparation}.

\bibitem{MR1228206}
Henri Cohen.
\newblock {\em A course in computational algebraic number theory}, volume 138
  of {\em Graduate Texts in Mathematics}.
\newblock Springer-Verlag, Berlin, 1993.

\bibitem{MR1056627}
Don Coppersmith and Shmuel Winograd.
\newblock Matrix multiplication via arithmetic progressions.
\newblock {\em J. Symbolic Comput.}, 9(3):251--280, 1990.

\bibitem{MR1135749}
James~L. Hafner and Kevin~S. McCurley.
\newblock Asymptotically fast triangularization of matrices over rings.
\newblock {\em SIAM J. Comput.}, 20(6):1068--1083, 1991.

\bibitem{MR0124316}
Kenkichi Iwasawa.
\newblock On {$\Gamma $}-extensions of algebraic number fields.
\newblock {\em Bull. Amer. Math. Soc.}, 65:183--226, 1959.

\bibitem{MR0161833}
A.~Ya. Khinchin.
\newblock {\em Continued fractions}.
\newblock The University of Chicago Press, Chicago, Ill.-London, 1964.

\bibitem{MR0485768}
Serge Lang.
\newblock {\em Cyclotomic fields}.
\newblock Springer-Verlag, New York, 1978.
\newblock Graduate Texts in Mathematics, Vol. 59.

\bibitem{MR2001757}
Joachim von~zur Gathen and J{\"u}rgen Gerhard.
\newblock {\em Modern computer algebra}.
\newblock Cambridge University Press, Cambridge, second edition, 2003.

\end{thebibliography}
\end{document}